\documentclass[UTF-8,reqno]{amsart}
\usepackage{enumerate}
\usepackage{mhequ}
\usepackage{amsthm,amsmath,amssymb,url,color, booktabs,nccmath}
\usepackage[left=3.2cm,right=3.2cm,top=4cm,bottom=4cm]{geometry}
\usepackage{mathrsfs}
\usepackage{enumitem,dsfont}
\usepackage{subfigure}
\usepackage[graphicx]{realboxes}
\usepackage{tikz}
\usepackage{color}
\usepackage[colorlinks=true]{hyperref}
\hypersetup{
    linkcolor=blue,          
    citecolor=red,        
    filecolor=blue,      
    urlcolor=cyan
}

\definecolor{darkergreen}{rgb}{0.0, 0.5, 0.0}

\numberwithin{equation}{section}
\def\theequation{\arabic{section}.\arabic{equation}}
\newcommand{\be}{\begin{eqnarray}}
\newcommand{\ee}{\end{eqnarray}}
\newcommand{\ce}{\begin{eqnarray*}}
\newcommand{\de}{\end{eqnarray*}}
\newtheorem{theorem}{Theorem}[section]
\newtheorem{lemma}[theorem]{Lemma}
\newtheorem{proposition}[theorem]{Proposition}
\newtheorem{Examples}[theorem]{Example}
\newtheorem{corollary}[theorem]{Corollary}
\newtheorem{assumption}[theorem]{Assumption}
\newtheorem{definition}[theorem]{Definition}
\theoremstyle{definition}
\newtheorem{remark}[theorem]{Remark}

\DeclareMathOperator{\supp}{supp}

\def\[{{\Big[}}
\def\]{{\Big]}}
\def\<{{\langle}}
\def\>{{\rangle}}
\def\({{\Big(}}
\def\){{\Big)}}

\def\bx{{\mathbf{x}}}

\def\dif{{\mathord{{\rm d}}}}

\def\min{{\mathord{{\rm min}}}}

\def\={&\!\!=\!\!&}

\newcommand{\norm}[1]{\left\|#1\right\|}

\def\mN{{\mathbb N}}

\def\mP{{\mathbb P}}

\def\mR{{\mathbb R}}

\def\mT{{\mathbb T}}

\def\mZ{{\mathbb Z}}

\def\1{{\mathbf{1}}}

\def\geq{\geqslant}
\def\leq{\leqslant}

\def\le{\leqslant}

\def\div{\mathord{{\rm div}}}

\def\[{{\Big[}}
\def\]{{\Big]}}
\def\<{{\langle}}
\def\>{{\rangle}}
\def\({{\Big(}}
\def\){{\Big)}}

\def\bx{{\mathbf{x}}}

\def\dif{{\mathord{{\rm d}}}}

\def\min{{\mathord{{\rm min}}}}

\def\={&\!\!=\!\!&}
\def\bt{\begin{theorem}}
\def\et{\end{theorem}}
\def\bl{\begin{lemma}}
\def\el{\end{lemma}}
\def\br{\begin{remark}}
\def\er{\end{remark}}
\def\bx{\begin{Examples}}
\def\ex{\end{Examples}}
\def\bd{\begin{definition}}
\def\ed{\end{definition}}
\def\bp{\begin{proposition}}
\def\ep{\end{proposition}}
\def\bc{\begin{corollary}}
\def\ec{\end{corollary}}

\def\geq{\geqslant}
\def\leq{\leqslant}

\def\le{\leqslant}

\def\div{\mathord{{\rm div}}}

\def\Id{\textrm{Id}}

 \def\R{\mathbb R}
 \def\R{\mathbb R}    
\def\N{\mathbb N}  
   
\def\<{\langle} \def\>{\rangle}

\allowdisplaybreaks

\begin{document}

\title[Non-uniqueness  for  nFPEs and DDSDEs  ]{
Non-Uniqueness for Nonlinear Fokker--Planck Equations and Their Associated Distribution-Dependent SDEs}

\author{Huaxiang L\"u}
\address[H. L\"u]{Academy of Mathematics and Systems Science, Chinese Academy of Sciences, Beijing 100190, China, and  Fakult\"at f\"ur Mathematik, Bielefeld Universit\"at, D 33615 Bielefeld, Germany}
\email{lvhuaxiang22@mails.ucas.ac.cn }

\begin{abstract}

In this paper, we study distribution-dependent stochastic differential
equations on the domain $\mathcal O=\mathbb T^d$ or $\mathbb R^d$,
$d\geq 2$, of the form
\begin{align*}
{\rm d}X_t
=
v(t,X_t,\rho_t)\,{\rm d}t
+
\sqrt{2}\,
\sigma(t,X_t,\rho_t)\,{\rm d}W_t,
\qquad
\rho_t:=\frac{{\rm d}\mu_t}{{\rm d}x},
\end{align*}
where $\mu_t=\operatorname{Law}(X_t)$.  
Our main construction is carried out at the level of the associated nonlinear
Fokker--Planck equations. We first build non-unique probability solutions to
these PDEs  and then use the superposition principle to
obtain non-unique martingale solutions to the corresponding DDSDEs.

We establish two main non-uniqueness results concerning stationary states,
both on the torus and in the whole space, under the corresponding structural
assumptions.  First, we
construct a   divergence-free drift
$v\in C_tL^{d-}$ such that the DDSDE admits \emph{infinitely many} distinct
  solutions starting from the  stationary initial density. This
result lies at the natural critical regularity threshold: in several models, well-posedness is expected for drifts in $C_tL^{d+}$.
Second, for $d\geq 3$ and every prescribed $N\in\mathbb{N}$, we construct a
 divergence-free drift for which the DDSDE admits at least
$N$ distinct stationary martingale solutions. The resulting multiplicity of
equilibrium states is reminiscent of multistability and phase-transition
phenomena in physical systems.
\end{abstract}


\date{\today}

\maketitle

\tableofcontents

\section{Introduction}

Interacting particle systems provide effective models for a wide range of
collective phenomena arising in physics, biology, kinetic theory, and other
scientific fields. As a prototypical example, consider a system of $N$ weakly
interacting particles on $\mathbb{R}^d$ governed by
\begin{align}
\dif X_t^{i,N}
&=
(V*\mu_t^N)(X_t^{i,N})\,\dif t
+
\sqrt{2}\,
(B*\mu_t^N)(X_t^{i,N})\,\dif W_t^i,
\qquad
i=1,\ldots,N,
\notag
\end{align}
where
\begin{align*}
\mu_t^N
:=
\frac1N\sum_{j=1}^N\delta_{X_t^{j,N}}
\end{align*}
is the empirical measure of the particle system,
$V:[0,T]\times\mathbb{R}^d\to\mathbb{R}^d$ and
$B:[0,T]\times\mathbb{R}^d\to\mathbb{R}^{d\times d}$ are measurable
interaction kernels, and $\{W^i\}_{i\in\mathbb{N}}$ is a family of independent
standard   Brownian motions. 
Under suitable assumptions on the coefficients and the initial particle
configuration, the propagation of chaos principle asserts that, as
$N\to\infty$, the empirical measure $\mu_t^N$ converges to a deterministic
probability measure $\mu_t$, which is the law of a solution to a
McKean--Vlasov stochastic differential equation, also known as a
distribution-dependent stochastic differential equation (DDSDE):
\begin{align}
\dif X_t
=
b\bigl(t,X_t,\mu_t\bigr)\,\dif t
+
\sqrt{2}\,
\sigma\bigl(t,X_t,\mu_t\bigr)\,\dif W_t,
\qquad
\mu_t=\operatorname{Law}(X_t).
\label{eq:sde}
\end{align}
Here
$
b:[0,T]\times\mathbb{R}^d\times\mathcal P(\mathbb{R}^d)
\to\mathbb{R}^d,
\sigma:[0,T]\times\mathbb{R}^d\times\mathcal P(\mathbb{R}^d)
\to\mathbb{R}^{d\times d}
$
are measurable coefficients. In the convolution-type example above, they are
given by
$
b(t,x,\mu)=
\int_{\mathbb{R}^d}
V(t,x-y)\,\mu(\dif y), 
\sigma(t,x,\mu)=
\int_{\mathbb{R}^d}
B(t,x-y)\,\mu(\dif y).
$ 
 
The (linear) Kolmogorov operator associated with \eqref{eq:sde} is
\begin{align*}
L_\mu 
=
\sigma(t,x,\mu)\sigma(t,x,\mu)^T:\nabla^2
+
b(t,x,\mu)\cdot\nabla.
\end{align*}
Formally, and  under suitable integrability assumptions,
It\^o's formula shows that the marginal laws
$\mu_t=\operatorname{Law}(X_t)$ satisfy the nonlinear Fokker--Planck equation (nFPE)
\begin{align*}
\partial_t\mu_t
=
L_{\mu_t}^*\mu_t.
\end{align*}

Besides arising as mean-field limits of interacting particle systems,
nonlinear Fokker--Planck equations also appear in many fundamental models from
physics, biology, and kinetic theory. At the macroscopic level, these equations
describe the time evolution of  particle densities
under the combined effects of diffusion, transport, and self-consistent
interaction forces.  
Typical examples include  nonlinear   porous-medium type equation,
\begin{align*}
\partial_t\rho
=
\Delta\beta(\rho)+\div(E\rho);
\end{align*} 
the Keller--Segel equation
\begin{align*}
\partial_t\rho
&=
\Delta\rho
-
\operatorname{div}(\rho\nabla c),
\qquad
-\Delta c= \rho;
\end{align*}
and the 2D vorticity Navier--Stokes equation
\begin{align*}
\partial_t\rho
&=
\Delta\rho
-
\operatorname{div}(\rho\nabla^\perp c),
\qquad
-\Delta c= \rho.
\end{align*}
Such equations possess rich analytic structures and natural probabilistic
interpretations: the associated DDSDE describes the stochastic evolution of a
representative particle, cell, or agent, while the nonlinear Fokker--Planck
equation governs the corresponding   density.

A central problem in the theory of DDSDEs and the associated nFPEs is their well-posedness. Classical results on
McKean--Vlasov equations in the whole space can be found in
\cite{Fun84,Szn84,Sch87,Chi94}. Since then, a substantial theory has been
developed for equations with singular coefficients; see
\cite{RZ21,CRF22,HW23,Wan23,Zha25}. These works extend, in various directions, the classical regularization theory for SDEs with singular drifts developed in\cite{KR05}. 
On the PDE side, Fokker--Planck--Kolmogorov equations with
measure-valued or probability-valued solutions have been studied extensively;
see   \cite{ BR21,BR22,BR23,BSS23} and the references therein. For results on
stationary distributions of DDSDEs, we refer to
\cite{Wan18,LM22,HWY24}. 
 
A related but conceptually different problem concerns the multiplicity of
stationary distributions. The coexistence of several stationary states is a
mathematical signature of multistability and is commonly associated with phase
transitions in mean-field systems. Moreover, at criticality the fluctuations
of mean-field models need not obey the usual Gaussian central-limit scaling;
see, for instance, \cite{EN78a,EN78b,Daw83}.
A classical example is Dawson's model with double-well confinement and
Curie--Weiss interaction on the line. It can be written in the form
\begin{align}
\dif X_t
=
-
\left[
\nabla U(X_t)
+
\theta\bigl(X_t-m_t\bigr)
\right]\dif t
+
\sigma_0\,\dif W_t,
\qquad
m_t
:=
\int_{\mathbb{R}}y\,\mu_t(\dif y),\notag
\end{align}
where 
\begin{align*}
U(x)
=
\frac14x^4-\frac12x^2,
\qquad
\theta>0.
\end{align*}
Dawson \cite{Daw83} proved that there exists a critical noise strength
$\sigma_c>0$ such that the system has a unique stationary distribution when
$\sigma_0>\sigma_c$, whereas three stationary distributions coexist when
$0<\sigma_0<\sigma_c$. This multiplicity of equilibrium states is referred to
as a phase transition. Further quantitative studies of phase transitions and
long-time behavior can be found in
\cite{HT10,Tug10,Tug13,DT16}. We also refer to
\cite{CP10,CGPS20} for related results in the periodic setting.
Most of these phase-transition results concern constant scalar diffusion.  The stationary problem for genuinely
distribution-dependent diffusion coefficients is considerably less
understood; see, in particular, \cite{Zha23}.

Another natural phenomenon is the possible
non-uniqueness of the time evolution starting from a fixed initial law. For
rough drift fields, even if the system is initialized at a stationary
distribution, it is natural to ask whether every weak solution must remain
stationary.
An even more striking question is whether infinitely many distinct evolutions
can start from a common stationary state.
In recent work, the authors \cite{GG25,LRZ25,RZZ25,LR25} established weak non-uniqueness
for SDEs with constant diffusion
$\sigma=\operatorname{Id}$ and supercritical drifts.   In that setting,
the associated Fokker--Planck equation is linear. Consequently, the existence of two solutions automatically yields infinitely
many solutions. For general distribution-dependent coefficients, however, the
associated Fokker--Planck equation is genuinely nonlinear in the density, and
convex combinations of solutions are no longer solutions in general.
The construction of infinitely many solutions requires additional
ideas and cannot be obtained by a direct application of the methods developed
for a single linear Fokker--Planck equation.

The main purpose of the present paper is to establish non-uniqueness in law
for a broad class of DDSDEs with nonlinear or nonlocal diffusion and
interaction mechanisms. We show that an arbitrarily small but sufficiently
rough external drift can destroy both uniqueness of the evolution starting
from a stationary initial law and uniqueness of stationary states.

We treat both the periodic and the whole-space settings. The periodic setting
is technically simpler and allows us to cover a broader class of nonlinear and
singular interaction mechanisms, since the   volume measure provides
a natural stationary reference state for many equations. Then in the latter whole-space case, we work within the confining framework studied by
Barbu and R\"ockner \cite{BR23}, in which the   equation enjoys
well-posedness under suitable assumptions. In   contrast, our results show
that these conclusions may fail after adding an arbitrarily small but
sufficiently rough, compactly supported divergence-free perturbation to the
drift.

\subsection{Main results on the torus}

From now on, we consider the following distribution-dependent SDE on
$\mathbb{T}^d$:
\begin{align}
    \dif X_t
    =
    v(t,X_t)\,\dif t
    +(V*\rho_t)(X_t)\,\dif t
    +\sqrt{2}\,\sigma(t,X_t,\rho_t)\,\dif W_t ,
    \label{eq:sdei}
\end{align}
where $\mu_t=\operatorname{Law}(X_t)$ and   we denote by $\rho_t$
its density with respect to the Lebesgue measure on $\mathbb{T}^d$, namely $\mu_t(\dif x)=\rho_t(x)\,\dif x$.  Here
$v:[0,T]\times\mathbb{T}^d\to\mathbb{R}^d$,
$V:\mathbb{T}^d\to\mathbb{R}^d$, and
$\sigma:[0,T]\times\mathbb{T}^d\times\mathcal P(\mathbb{T}^d)
\to \mathbb{R}^{d\times d}$ are measurable functions.  
The process $W$ is a standard $\mathbb{T}^d$-valued Brownian motion.
The displayed SDE should be understood  
through the corresponding martingale problem on the manifold $\mathbb{T}^d$.
We adopt the latter interpretation throughout this paper. We refer to
\cite{Hsu02} for background on Brownian motion and stochastic analysis on
manifolds. With a slight abuse of notation, we shall still refer to
\eqref{eq:sdei} as a DDSDE.

\begin{definition}[Martingale solution]
Let $C([0,T];\mathbb{T}^d)$ be the space of continuous paths on the torus,
equipped with its Borel $\sigma$-algebra. Denote by
\begin{align*}
\Pi_t(\omega):=\omega(t),
\qquad
\omega\in C([0,T];\mathbb{T}^d),
\end{align*}
the canonical process, and let $(\mathcal F_t)_{t\in[0,T]}$ be the natural
filtration generated by $(\Pi_t)_{t\in[0,T]}$.

The corresponding (linear)  diffusion operator is defined by
\begin{align*}
L_{\rho_t}:=  \sigma\sigma^{T}(t,x,\rho_t):\nabla^2+(v(t,x)+V*\rho_t)\cdot \nabla.    
\end{align*} 

We say that a probability measure $\mathbf Q$ on
$C([0,T];\mathbb{T}^d)$ is a martingale solution associated with
$L_{\rho_t}$ if
\begin{align*}
\dif\mathbf Q\circ \Pi_t^{-1} =\rho_t \,\dif x,
\qquad
t\in[0,T],
\end{align*}
and, for every $f\in C^2(\mathbb{T}^d)$, the process
\begin{align*}
     f( \Pi_t) -f(\Pi_0)-\int^t_0 L_{\rho_s}f( \Pi_s)\dif s
\end{align*}
is a $\mathbf Q$-martingale with respect to $(\mathcal F_t)_{t\in[0,T]}$.
\end{definition}

In this paper, we mainly focus on the associated nonlinear Fokker--Planck
equation. In distributional form, it reads
\begin{align}
\partial_t \rho-\div\div((\sigma\sigma^{T})(\rho) \rho)+\div(v\rho)+\div((V*\rho)\rho)&=0
\qquad
\text{on } (0,T)\times\mathbb{T}^d. 
\label{eq:fpe:i}
\end{align} 
For a matrix-valued function $A=(A_{ij})_{1\leq i,j\leq d}$, we define
$\div\div A=\sum_{i,j}\partial_i\partial_jA_{ij}$. 
By the superposition principle, any probability solution $\rho$ to the nFPE
  satisfying the standard integrability conditions on the
coefficients gives rise to a martingale solution of the DDSDE \eqref{eq:sdei}.

We  impose the following assumptions on the coefficients.

\begin{assumption}
\label{def:ass}
Assume that the interaction kernel satisfies
$V\in L^{d,\infty}(\mathbb{T}^d;\mathbb{R}^d)$.

We assume that the diffusion coefficient
$\sigma:[0,T]\times\mathbb{T}^d\times\mathcal P(\mathbb{T}^d)
\to\mathbb{R}^{d\times d}$ satisfies one of the following conditions.

\begin{enumerate}

\item[(1)] (Nemytskii-type)
$\sigma(t,x,\rho)=\sigma_1(t,x,\rho_t(x))$,
where
$\sigma_1:[0,T]\times\mathbb{T}^d\times  \mathbb{R}\to \mathbb{R}^{d\times d}$
is a measurable function satisfying
$\sigma_1\in C_b^2([0,T]\times  \mathbb{T}^d\times \mathbb{R})$, and $\|\partial_\rho[\sigma_1(t,x,\rho)\sigma_1^T(t,x,\rho)\rho]\|_{C_{t,x}^1C_\rho^0}<\infty$.

\item[(2)] (Convolution-type)
$\sigma(t,x,\rho)=\int_{\mathbb{T}^d}\sigma_2(t,x-y)\rho_t(y)\,\dif y$, or $\sigma(t,x,\rho)=\sqrt{\int_{\mathbb{T}^d}\sigma_2(t,x-y)\rho_t(y)\,\dif y}$,
where
$\sigma_2:[0,T]\times \mathbb{T}^d\to \mathbb{R}^{d\times d}$
is a measurable function satisfying
$\sigma_2\in C^2([0,T];C( \mathbb{T}^d))$.  In the second case, we assume that the matrix inside the square root is
symmetric and non-negative. The notation $\sqrt{B}$ denotes a matrix square
root of $B$, namely a matrix $A$ satisfying $AA^T=B$.

\item[(3)] $\sigma$ admits the decomposition
$\sigma\sigma^T
=
\tilde{\sigma}_1(\tilde{\sigma}_1)^T
+
\tilde{\sigma}_2(\tilde{\sigma}_2)^T$,
where $\tilde{\sigma}_1$ satisfies condition \textnormal{(1)} above and $\tilde{\sigma}_2$ satisfies condition \textnormal{(2)} above.

\end{enumerate}
\end{assumption}

In particular, it is straightforward to verify that, for every diffusion
coefficient $\sigma$ satisfying Assumption~\ref{def:ass}\textnormal{(2)}, or
Assumption~\ref{def:ass}\textnormal{(1)} in the space-independent case $\sigma(t,x,\mu)=\sigma_1(t,\rho(x)),$
the following property holds:
\begin{align}
\operatorname{div}\operatorname{div}
\bigl(
\sigma\sigma^T(t,x,1)
\bigr)
\equiv 0 .
\label{eq:divdiv}
\end{align}
Let us also point out that condition \eqref{eq:divdiv} covers many distribution-independent cases. For example, 
\begin{align*}
\sigma\sigma^T(t,x)
=
\operatorname{diag}\bigl(1+\varepsilon h(t,x),1,\ldots,1\bigr),
\end{align*}
where $h(t,x)$ is any smooth periodic function independent of $x_1$, and $|\varepsilon|$ is sufficiently
small. Or \begin{align*}
    \sigma\sigma^T(t,x)=\Id+\varepsilon
\operatorname{diag}(B,0,..,0), \ B
=
\begin{pmatrix}
\partial_{22}\psi(x_1,x_2), & -\partial_{12}\psi(x_1,x_2)\\
-\partial_{12}\psi(x_1,x_2), & \partial_{11}\psi(x_1,x_2)
\end{pmatrix}.
\end{align*}
where   $\psi\in C^\infty(\mathbb T^2)$ and $|\varepsilon|$ is sufficiently
small. 

 Consequently, under \eqref{eq:divdiv}, if $v$ is divergence-free, then the
constant density $\rho\equiv 1$ is a stationary probability solution of the
  nonlinear Fokker--Planck equation \eqref{eq:fpe:i}.

\subsubsection{Non-uniqueness of weak solutions}

Our first main result shows that the DDSDE may admit \emph{infinitely many}
distinct weak solutions, even when starting from this stationary state.

 \bt\label{thm:4converge}
Let $T>0$, $\epsilon_0\in(0,1)$. Let $d\geq 2$, $1<s<d$, and let
$p,r\in[1,\infty]$ satisfy
$\frac{d}{p}+\frac{1}{r}>1 .$
Assume that $V$ and $\sigma$ satisfy Assumption~\ref{def:ass}, and  
  also \eqref{eq:divdiv}. Let
$\overline v\in C^\infty([0,T]\times\mathbb{T}^d;\mathbb{R}^d)$ be
divergence-free. Then there exists a
distribution-independent, divergence-free vector field
$v\in L^r(0,T;L^p(\mathbb{T}^d;\mathbb{R}^d))
\cap C([0,T];L^s(\mathbb{T}^d;\mathbb{R}^d))$
satisfying
\begin{align}
\|v-\overline v\|_{L^r_tL^p(\mathbb{T}^d)}
+
\|v-\overline v\|_{C_tL^s(\mathbb{T}^d)}
\leq \epsilon_0 ,
\label{bd:v-barv}
\end{align}
such that the following assertions hold.

\begin{enumerate}
\item The nonlinear Fokker--Planck equation \eqref{eq:fpe:i} admits infinitely many
distinct probability density solutions $\rho^i$, $i\in\mathbb{N}$, with
initial condition $\rho_0\equiv 1$. More precisely, there exists
$\epsilon>0$ such that, for every $i\in\mathbb{N}$,
$\rho^i
\in
C\bigl([0,T];L^{1+\epsilon}(\mathbb{T}^d)\bigr)
\cap
L^1\bigl((0,T);W^{1,1+\epsilon}(\mathbb{T}^d)\bigr),
$
and 
$
|v|^{1+\epsilon}\rho^i
\in L^1([0,T];L^1(\mathbb{T}^d)).
$

\item
The DDSDE \eqref{eq:sdei} admits infinitely many distinct martingale solutions
$\mathbf Q^i$, $i\in\mathbb{N}$, starting from the  volume measure on
$\mathbb{T}^d$. Moreover, these solutions satisfy
$
\mathbf E^{\mathbf Q^i}
\left[
\int_0^T |v(s,\Pi_s)|^{1+\epsilon}\,\dif s
\right]
<\infty,
  i\in\mathbb{N}.
$
\end{enumerate}
\et
For sufficiently regular diffusion coefficients satisfying suitable uniform
ellipticity  assumptions, the corresponding initial value problem
is known to be well posed for drifts $v\in L^r_tL^p$ with
$\frac{d}{p}+\frac{2}{r}<1$; see \cite{KR05,RZ21,Zha25} and the references
therein. 
In contrast, our result shows that, for any  divergence-free drift field
$\overline v\in C_tL^{d+}$, every sufficiently small neighborhood of $\overline v$ in the
low-regularity topology
$  C_tL^{d-}$ contains ``bad'' drift fields for which the
corresponding DDSDE admits non-unique weak solutions. Our result is essentially sharp in view of the expected well-posedness for drifts.  Several examples covered by our results are presented in
Section~\ref{sec:example}.

We emphasize that our construction yields \emph{infinitely many} weak
solutions. Since the equation is nonlinear, this is
fundamentally different from merely producing finitely many distinct solutions:
convex combinations of solutions are no longer solutions in general. As a
consequence, the construction of infinitely many solutions requires additional
ideas and constitutes one of the main technical difficulties of the present
work. The proof of Theorem~\ref{thm:4converge} is given in
Section~\ref{cogpss3}- Section~\ref{sec:convex}.

\medskip

The previous results apply to a large class of interaction kernels. However, in many physical models, the interaction kernel may be more
singular. For instance, if $G$ denotes the Green function, then
\begin{align*}
\nabla G(x)\lesssim  |x|^{1-d}
\in L^{\frac{d}{d-1},\infty},
\qquad d\geq 2 .
\end{align*}
The Green function on the torus exhibits the same type of singularity near
the origin; see Section~\ref{sec:example:kse} for a detailed discussion.
Therefore, when $d\geq 3$, such kernels are not covered by
Theorem~\ref{thm:4converge}. The following result shows that our
non-uniqueness mechanism remains valid even for more singular interactions.

\bt\label{thm:6converge}
Let $T>0$, $\epsilon_0\in(0,1)$, and $N\in\mathbb{N}$. Let $d\geq 3$ and
$1<d_0<d$. Assume that the coefficients $\sigma$ and $V$ are given by
$\sigma(\rho)=\sqrt{a*\rho},
V=\operatorname{div} b,$
where $a,b\in L^{\frac{d}{d-\gamma},\infty}
(\mathbb{T}^d;\mathbb{R}^{d\times d})$ and
$
1<\frac{d}{\gamma}<d_0':=\frac{d_0}{d_0-1}.
$  
Assume moreover that $a*\rho$ admits a measurable matrix square root for all
densities $\rho$ considered below.

Then, for any smooth divergence-free vector field
$\overline v(t,x)$, there exists a   divergence-free
vector field
$
v\in C ([0,T];L^{d_0}(\mathbb{T}^d;\mathbb{R}^d) )
$
satisfying $\|v-\overline v\|_{C_tL^1(\mathbb{T}^d)}
\leq \epsilon_0, $ 
such that the following assertions hold.

\begin{enumerate}
\item
The nonlinear Fokker--Planck equation \eqref{eq:fpe:i} admits at least $N$
distinct nontrivial probability density solutions $\rho^i$, $1\leq i\leq N$,
with initial condition $\rho^i_0\equiv 1$. More precisely, for some
small constant $\epsilon>0$, one has
$
\rho^i
\in
C\bigl([0,T];L^{d_0'}\cap W^{1,1+\epsilon}(\mathbb{T}^d)\bigr),
1\leq i\leq N .
$

\item
The DDSDE \eqref{eq:sdei} admits at least $N$ distinct martingale solutions
starting from the  volume measure on $\mathbb{T}^d$.
\end{enumerate}
\et

Compared with Theorem~\ref{thm:4converge}, the time regularity estimates here
are upgraded to $C_t$ bounds. Consequently, we obtain stronger regularity
estimates for the densities $\rho^i$, which are sufficient to make sense of
the nonlinear interaction terms involving the singular kernels above. At the same time, we still keep the drift in the 
regime $C_tL^{d-}$. The
proof of Theorem~\ref{thm:6converge} is deferred to
Section~\ref{cogpss6}- Section~\ref{sec:6convex}.

\subsubsection{Non-uniqueness of stationary solutions}

Our second main result concerns the existence of non-unique stationary
solutions to the DDSDE \eqref{eq:sdei}. More precisely, by a stationary
solution we mean a stationary martingale solution $\mathbf Q$ whose  time
marginals are time-independent, namely
\begin{align*}
\dif\mathbf Q\circ \Pi_t^{-1} =\rho\,\dif x,
\qquad
t\geq 0,
\end{align*}
for some probability density $\rho$ on $\mathbb{T}^d$.
The non-uniqueness of stationary solutions means that the system admits several
different equilibrium states. From the physical viewpoint, this is reminiscent
of multistability and phase-transition phenomena.
In this setting, we study the following nonlinear stationary Fokker--Planck
equation:
\begin{align}
-\operatorname{div}\operatorname{div}
\bigl(
(\sigma\sigma^T)(\rho)\rho
\bigr)
+
\operatorname{div}(v\rho)
+
\operatorname{div}\bigl((V*\rho)\rho\bigr)
=
0.
\label{eq:fpe:sta}
\end{align}

Since the problem is posed on the compact torus, strongly coercive diffusion
operators often impose strong rigidity on stationary probability solutions.
For example, for the aforementioned case
\begin{align*}
-\Delta\beta(\rho)=0,\ 
\beta'>0, \ \ {\rm or}\  \ 
-\Delta(
(a*\rho) \rho
)
=
0,\ a(x)>0,a\in L^1,
\end{align*}
one immediately obtains that  the unique
probability solution is $\rho\equiv 1$.  
In contrast, in this part, we show that, for many systems, small but rough perturbations of
the drift can produce multiple stationary states.

\begin{theorem}\label{thm:sta}
Let $\epsilon_0\in(0,1)$ and $N\in\mathbb{N}$. Let $d\geq 3$ and
$1<d_0<d-1$. Assume that the time-independent coefficients $\sigma$ and $V$
either satisfy Assumption~\ref{def:ass}, or belong to the following class of
singular interaction kernels:
$
\sigma(\rho)=\sqrt{a*\rho},
V=2\,\div b,
$
where $
a,b\in
L^{\frac{d}{d-\gamma},\infty}
(\mathbb{T}^d;\mathbb{R}^{d\times d}),
1<\frac{d}{\gamma}<d_0':=\frac{d_0}{d_0-1}.
$ 
Then, for any smooth divergence-free vector field
$\overline v\in C^\infty(\mathbb{T}^d;\mathbb{R}^d)$, there exists a
divergence-free vector field
$
v\in L^{d_0}(\mathbb{T}^d;\mathbb{R}^d)
$
satisfying $\|v-\overline v\|_{L^1(\mathbb{T}^d)}
\leq
\epsilon_0,$ 
such that the following assertions hold.

\begin{enumerate}
\item
The stationary equation \eqref{eq:fpe:sta} admits at least $N$ distinct
non-constant probability density solutions $\rho^i$, $1\leq i\leq N$.
Moreover, for some $\epsilon>0$ and every $1\leq i\leq N$, one has
$
\rho^i
\in
L^{d_0'}(\mathbb{T}^d)
\cap
W^{1,1+\epsilon}(\mathbb{T}^d).
$

\item
The DDSDE \eqref{eq:sdei} admits at least $N$ distinct nontrivial stationary
martingale solutions.  
\end{enumerate}
\end{theorem}
The proof of Theorem~\ref{thm:sta} is given in Section~\ref{cogpss4}- Section~\ref{sec:convexsta}.

Moreover, in the linear stationary case
\begin{align*}
-\Delta\rho+\operatorname{div}(v\rho)=0,
\end{align*}
with $v\in L^p $, several uniqueness and
non-uniqueness results are known. For $p>d$ and $d\geq 2$, Bogachev,
R\"ockner, and Stannat \cite{BRS02} proved  that equation admits a unique
probability solution $\rho$ in the class $v\rho-\nabla\rho\in L^1 (\mathbb{T}^d;\mathbb{R}^d)$.
In this paper, we prove that, for every $d\geq 3$, there
exists a divergence-free drift
$
v\in L^{d-1-} 
$
such that the above linear stationary equation admits non-unique probability
solutions, which moreover satisfy the conditions above.   We also recall that, for divergence-free drifts
$v\in L^p $, (non)-uniqueness in the class
$\rho\in H^1 $ was studied in \cite{Zhi04,CO22}.

\subsection{Main results in the whole space}\label{sec:rcase}
The results above can also be adapted to the whole-space setting. The
periodic case is technically simpler, since the  volume measure
provides a natural stationary probability density for many equations under
consideration. In contrast, there is no analogous canonical probability
density on $\mathbb{R}^d$. As a consequence, the construction in the
whole-space setting requires a more delicate analysis of the behavior at
infinity.

We consider the following nonlinear Fokker--Planck equation on
$\mathbb{R}^d$, $d\geq 2$:
\begin{align}
\partial_t\rho
-
\Delta 
\beta(\rho)  
+
\div(v\rho)
+
\div(Eb(\rho)\rho)
=
0 .
\label{R:eq:fpe}
\end{align}
We assume that the coefficients satisfy the following conditions:
\begin{enumerate}
    \item[(a)] $\beta\in C^2(\mathbb{R})$, $\beta(0)=0$, and there exist
    constants $0<\gamma<\gamma_1<\infty$ and $0<\gamma_2<\infty$ such that
   $
    \gamma\leq \beta'(r)\leq \gamma_1,
    |\beta''(r)|\leq \gamma_2,
    r\in\mathbb{R}.
 $

    \item[(b)] $b\in C_b^1(\mathbb{R})$ and
   $
    b(r)\geq b_0>0,
    r\geq 0 .
   $

    \item[(c)] $E=-\nabla\Phi$, where
    $\Phi\in C(\mathbb{R}^d)\cap W^{2,\infty}_{\mathrm{loc}}(\mathbb{R}^d)$
    satisfies $\Phi\geq 1$,
    $
    \lim_{|x|\to\infty}\Phi(x)=+\infty,
     $
    and there exists $m\in[2,\infty)$ such that
    $\Phi^{-m}\in L^1(\mathbb{R}^d)$ and
    $
    |\nabla\Phi|+|\nabla^2\Phi|
    \leq
    \Phi^m .
    $
\end{enumerate}

Such equations arise in statistical physics; see, for instance,
\cite{Fra05,FD01,SNC07,Tsa09}. They also play an important role in
nonequilibrium statistical mechanics, where they describe the evolution of
particle densities in disordered media.

The related DDSDE is formally given by
\begin{align}
\dif X_t
=
v(t,X_t)\,\dif t
+
E(X_t)b(\rho_t(X_t))\,\dif t
+
\sqrt{\frac{2\beta(\rho_t(X_t))}{\rho_t(X_t)}}\,\dif W_t .
\label{R:eq:por}
\end{align} 

In the case $v=0$, the system is conservative and admits a stationary solution of the form
\begin{align}
\rho^{\mathrm{st}}(x)
:=
g^{-1}\bigl(-\Phi(x)+\mu^{\mathrm{st}}\bigr),
\label{def:rhost}
\end{align}
where $\mu^{\mathrm{st}}\in\mathbb{R}$ is uniquely determined by the mass
constraint
\begin{align}
\int_{\mathbb{R}^d}
g^{-1}\bigl(-\Phi(x)+\mu^{\mathrm{st}}\bigr)
\,\dif x
=
1 ,\notag
\end{align}
and
\begin{align}
g(r)
:=
\int_1^r
\frac{\beta'(s)}{s b(s)}
\,\dif s,
\qquad
r>0 .\notag
\end{align}
By assumptions \textnormal{(a)} and \textnormal{(b)}, the function $g$ is
strictly increasing and has logarithmic growth at both $0$ and $\infty$.
Equivalently, $g^{-1}$ has exponential-type growth. Together with the
confining assumption on $\Phi$,  $\mu^{\mathrm{st}}$ exists and is unique.  

Recently, the well-posedness theory for this class of nonlinear
Fokker--Planck equations has been further developed in a number of works. We
refer the reader to \cite{BR21,BR22,BR23,BR23b,Gru25} and to the monograph
\cite{BKRS22} for a comprehensive account of the subject. For the reader's
convenience, we recall below several results from \cite{BR23} that are
particularly relevant to the present work. In addition to
\textnormal{(a)}--\textnormal{(c)}, some of these results require the following
conditions:
\begin{enumerate}
    \item[(d)] $E\in L^\infty(\mathbb{R}^d;\mathbb{R}^d)
    \cap W^{1,1}_{\mathrm{loc}}(\mathbb{R}^d;\mathbb{R}^d)$ and
    $ 
    \operatorname{div}E
    \in    L^2(\mathbb{R}^d)+L^\infty(\mathbb{R}^d).
   $ 

    \item[(e)] For a.e. $x\in\mathbb{R}^d$, $
    \gamma_1\Delta\Phi(x)-b_0|\nabla\Phi(x)|^2
    \leq 0 .$ 
\end{enumerate}
The following properties are known when $v\equiv0$:
\begin{itemize}
    \item Under conditions \textnormal{(a)}--\textnormal{(d)}, existence and
    uniqueness of mild solutions are established via a nonlinear contraction
    semigroup $S(t)$, $t>0$, acting on $L^1(\mathbb{R}^d)$.

    \item \textnormal{($H$-Theorem)} Under assumptions
    \textnormal{(a)}--\textnormal{(e)},  the solution $u(t)$ converges, along
    any sequence $t_n\to\infty$, to the stationary state $\rho^{\mathrm{st}}$ in
    $L^1(\mathbb{R}^d)$.

    \item Under conditions \textnormal{(a)}--\textnormal{(e)}, the stationary
    equation associated with \eqref{R:eq:fpe} admits a unique stationary
    probability solution, given by \eqref{def:rhost}.
\end{itemize}

Our main result is that the above properties may fail under suitable small but
rough perturbations of the drift field.

First, we show that there exists a drift field $v$ such that the equation
admits infinitely many solutions starting from the stationary state
$\rho^{\mathrm{st}}$. Moreover, the perturbation $v$ is compactly supported, and
therefore it does not affect the behavior of $E$ at infinity. Due to technical
limitations, we additionally require the potential $\Phi$ to be constant on some domain.

\bt\label{R:thm:4converge}
Let $\epsilon_0\in(0,1)$. Let $d\geq 2$, $1<s<d$, and let
$p,r\in[1,\infty]$ satisfy
$
\frac{d}{p}+\frac{1}{r}>1 .
$
Assume that conditions \textnormal{(a)}--\textnormal{(c)} hold, and that
$\Phi$ is constant on $[-\frac12,\frac12]^d$. Then there exists a
divergence-free vector field $v(t,x)$ such that
$
\operatorname{supp}_{t,x}v
\subset
 [\frac1{12},\frac{11}{12} ]
\times
 [-\frac12,\frac12 ]^d,
$
and $\|v\|_{L^r_tL^p(\mathbb{R}^d)}
+
\|v\|_{C_tL^s(\mathbb{R}^d)}
\leq
\epsilon_0 ,$ 
such that the following assertions hold.

\begin{enumerate}
    \item
    The nonlinear Fokker--Planck equation \eqref{R:eq:fpe} admits infinitely
    many distinct probability density solutions
    $\rho^i\in C([0,1];L^1(\mathbb{R}^d))$, $i\in\mathbb{N}$, satisfying
 $
    \rho^i(t,\cdot)=\rho^{\mathrm{st}}
    \text{for }
    t\in
     [0,\frac1{12} ]
    \cup
     [\frac{11}{12},1].
   $

    \item
    The DDSDE \eqref{R:eq:por} admits infinitely many distinct weak
    solutions starting from the stationary probability measure
    $\rho^{\mathrm{st}}(x)\,\dif x$.
\end{enumerate}
\et

We believe that the regularity of the drift field obtained here,
$v\in C_tL^{d-}$, is essentially optimal. In the linear case
$\beta'(r)=b(r)=1$, uniqueness is known to hold for drift fields in
$C_tL^{d+}$; see for example, \cite{KR05,Zha11,RZ21}. Thus, at least in this particular setting, our
result is sharp with respect to the integrability of the drift. 

We next show that the full convergence-to-equilibrium conclusion associated
with the $H$-theorem is not stable. 
Indeed, since
$
  v=0 
$
and $\rho^i=\rho^{\mathrm{st}}$ near the endpoints $t=0$ and $t=1$, we
may extend the pair $(v,\rho^i)$ periodically in time by setting
$
v(t+k)=v(t),
\rho^i(t+k)=\rho^i(t),
k\in\mathbb{N},
t\in[0,1).
$
As a consequence, any non-stationary solution obtained in this way cannot
converge to $\rho^{\mathrm{st}}$ or any other stationary solution as $t\to\infty$.

Finally, we show that there exists a drift field $v$ such
that the  equation admits multiple stationary solutions.

\begin{theorem}\label{R:thm:sta}
Let $\epsilon_0\in(0,1)$ and $N\in\mathbb{N}$. Let $d\geq 3$ and
$1<d_0<d-1$. Assume that conditions \textnormal{(a)}--\textnormal{(c)} hold. Then there exists a
compactly supported divergence-free vector field
$
v\in L^{d_0}(\mathbb{R}^d;\mathbb{R}^d),
\operatorname{supp}v
\subset
 [-\frac12,\frac12 ]^d,
$
satisfying $
\|v\|_{L^1(\mathbb{R}^d)}
\leq
\epsilon_0 ,$
such that the following assertions hold.

\begin{enumerate}
\item
Equation \eqref{R:eq:fpe} 
admits at least $N$ distinct stationary probability density solutions $\rho^i$,
$1\leq i\leq N$. Moreover, for some $\epsilon>0$ and every $1\leq i\leq N$,
$
\rho^i
\in
L^{d_0'}(\mathbb{R}^d)
\cap
W^{1,1+\epsilon}(\mathbb{R}^d).
$

\item
The DDSDE \eqref{R:eq:por} admits at least $N$ distinct stationary  weak
solutions. 
\end{enumerate}
\end{theorem}

The proofs of the above two theorems follow essentially the same strategy as
in the torus setting. The only new issue arises from the treatment of the
inverse divergence operator in the whole space. This difficulty is overcome by
introducing the Bogovskii operator in Section~\ref{R:tamr}. We shall omit a number of routine
computations and highlight only the necessary changes. The complete proofs are
deferred to Section~\ref{R:cogpss4}.
 
\subsubsection{Further discussion on the double-well model} 

We further discuss a possible connection with the classical double-well model. Consider, for instance, an even double-well potential
$
\Phi(x)=\frac14 |x|^4- |x|^2 .
$
Then the stationary density $\rho^{\mathrm{st}}$ defined in
\eqref{def:rhost} is also even. In particular,   one has
$
\int_{\mathbb{R}^d} x_i \rho^{\mathrm{st}}(x)\,\dif x=0,
1\leq i\leq d .
$
Here $x=(x_1,\ldots,x_d)$.
This observation is relevant to the mean-field model considered by Dawson
\cite{Daw83}:
\begin{align}
-\sigma_0\Delta\rho
+\div\bigl((-\nabla\Phi+m_\rho)\rho\bigr)
=
0,
\qquad
m_\rho:=\int_{\mathbb{R}^d}x\rho(x)\,\dif x ,\notag
\end{align}which also admits
$\rho^{\mathrm{st}}$ as a stationary solution. Dawson \cite{Daw83} showed that in 1D case, depending on the
parameters of the model, the system may exhibit either a unique stationary
state or 3 stationary states.  

If, in Theorem~\ref{R:thm:sta}, one could additionally ensure that the convex
integration construction preserves the even symmetry of the density, then the
theorem would also apply to the  double-well
model. A natural strategy would be to impose the symmetry conditions that
$v$ is odd and $\rho^i$ is even. Accordingly, throughout the convex integration
scheme, one would require
$
v_q$  to be odd,
$\rho_q^i$  to be even, and
$M_q^i$ to be odd.
Preserving these symmetries at each iteration would be an interesting problem,
but it requires a more delicate analysis of the construction, in particular of
the choice of the coefficients in the building blocks.

\subsection{Related literature and main new ideas}

The main objective of this article is to establish the first statement in Theorem \ref{thm:4converge} (and in Theorem \ref{thm:sta}), which will be proved via the convex integration method.  
This technique was first introduced to fluid dynamics by De Lellis and Sz\'ekelyhidi Jr. \cite{DLS09, DLS10, DLS13} and  has  led to numerous groundbreaking results for determined and stochastic fluid dynamics on the torus. For the incompressible Euler equations, the famous Onsager conjecture was proved in \cite{Ise18,BDLSV19}. We refer to \cite{TCP,DSJ17,BFH20,NV23,GKN23,GR24} for further literature on the Euler equation. For the Navier-Stokes equations, the sharp non-uniqueness of weak solutions has been shown in \cite{BVb,BCV18,CL22,CL23}.  We refer to \cite{BV,HZZ21b,HZZ19,HZZ22,LZ23b,LZ24} for further results on the Navier-Stokes equations.   
For the ODE or transport equations with Sobolev vector fields, we refer to \cite{CGSW15,MS18,MS19,MS20,BCDL21,CL21,PS23}. 
For the SDE or the Fokker-Planck equation, we refer to \cite{LRZ25,LR25} for non-uniqueness results.  Regarding the non-unique stationary solution to NS equation, we refer to \cite{Luo19}.   Recently, this method has been applied to fluid dynamics on the whole space, see \cite{MNY24a,MNY24b,LR25}.

 Existing convex integration results are largely restricted to systems involving
only finitely many equations, and  one cannot directly apply it 
simultaneously to infinitely many equations. Indeed, doing so would require
constructing the total perturbation as an infinite sum of sub-perturbations,
each of which is designed to eliminate the corresponding stress term. The
convergence of such a total perturbation would then become highly problematic.

Our approach is to introduce an increasing sequence of integers
$\{N_q\}_{q\in\mathbb{N}}$ such that
$2^{N_q}\sim \lambda_q^\alpha$ for some sufficiently small $\alpha>0$.
At step $q+1$, we apply the standard convex integration scheme only to the first $N_{q+1}$ equations. To ensure convergence
of the drift perturbation, we impose an additional decay factor $2^{-i}$ in
the estimate of the $i$-th stress term. A key observation is that this extra
decay does not interfere with the convex integration construction. Indeed, by
the definition of $N_{q+1}$, for every $i\leq N_{q+1}$ one has
$2^i\lesssim \lambda_{q+1}^{\alpha}$, which can be absorbed
as a low-frequency loss.
For the remaining equations, namely those with $i>N_{q+1}$, the stress term
$M_q^i$ is not treated by the convex integration perturbation at step $q+1$.
Nevertheless, it must still satisfy the more restrictive estimates required at
the next level. The key point is that $M_q^i$ can be rewritten in a form that
enjoys the stronger decay bound $4^{-i}$. Since $i>N_{q+1}$, we have
$2^{-i}\lesssim 2^{-N_{q+1}}\lesssim \lambda_{q+1}^{-\alpha}$, and hence
$
4^{-i}
\lesssim
2^{-i}\lambda_{q+1}^{-\alpha}.
$
This provides precisely the additional smallness needed to close the
iteration.

In the construction of stationary solutions, we can no longer exploit time
intermittency to improve the estimates. Moreover, the intermittent flows used
in the time-dependent construction, which have intermittency dimension $D=0$,
are no longer applicable. Instead, we employ Mikado-type flows with
intermittency dimension $D=1$. Consequently, the construction yields drift
regularity only up to
$
L^{d-D-}=L^{d-1-}.
$

\section{Examples covered by the main results}\label{sec:example}

In this section, we present several examples that fall within the scope of our
framework. We restrict our discussion to the periodic setting, corresponding
to Theorem~\ref{thm:4converge}--Theorem \ref{thm:sta}. These examples illustrate
that the non-uniqueness mechanism developed in this paper applies to a broad
class of nonlinear and nonlocal Fokker--Planck-type equations arising in
probability theory, statistical physics, and continuum mechanics.

\subsection{The distribution-independent case}

We first consider a distribution-independent SDE on $\mathbb{T}^d$:
\begin{align}
    \dif X_t
    =
    v(t,X_t)\,\dif t
    +
    \sqrt{2}\,\sigma(t,X_t)\,\dif W_t ,
    \notag
\end{align}
where
$v:[0,T]\times\mathbb{T}^d\to\mathbb{R}^d$ and
$\sigma:[0,T]\times\mathbb{T}^d\to\mathbb{R}^{d\times d}$ are measurable
functions. The corresponding Fokker--Planck equation reads as
\begin{align}
\partial_t\rho
-
\operatorname{div}\operatorname{div}\bigl(\sigma\sigma^T\rho\bigr)
+
\operatorname{div}(v\rho)
=
0 .
\notag
\end{align}
For simplicity, assume that $\sigma$ is uniformly elliptic and sufficiently
regular. We also assume that
$
\operatorname{div}\operatorname{div}\bigl(\sigma\sigma^T\bigr)
\equiv 0.
\notag
$ Then the constant density $\rho\equiv1$ is a
stationary probability solution when $v$ is divergence-free. In this case, the
model is covered by Assumption~\ref{def:ass}\textnormal{(1)} with $V=0$.
Therefore, by Theorem~\ref{thm:4converge}, there exists a
distribution-independent, divergence-free drift field
$
v\in C_tL^{d-}(\mathbb{T}^d)
$
such that the Fokker--Planck equation and the corresponding SDE admit
infinitely many weak solutions starting from the stationary density
$\rho_0\equiv1$.

Moreover, this result is essentially optimal with respect to the integrability
of the drift. Indeed, for every
$
v\in C_tL^p,
p>d,$
the SDE is  well-posed and the associated Fokker--Planck equation
is also well posed in the corresponding probability solution class. We refer
to Zhang~\cite{Zha11} for the corresponding result in the whole-space setting;
the same argument applies to the torus case by periodic lifting.

\subsection{The convolution-type equation}

We next consider the convolution-type coefficients covered by
Assumption~\ref{def:ass}\textnormal{(2)}. Let
$V\in L^p(\mathbb{T}^d;\mathbb{R}^d)$ for some $p>d$, and let
$\Sigma\in C^2([0,T]\times\mathbb{T}^d;\mathbb{R}^{d\times d})$. This gives
rise to the following McKean--Vlasov SDE:
\begin{align*}
\dif X_t
=
v(t,X_t)\,\dif t
+
(V*\rho_t)(X_t)\,\dif t
+
\sqrt{2}\,(\Sigma*\rho_t)(X_t)\,\dif W_t .
\end{align*}  We also assume that
the covariance matrix
$
(\Sigma*\rho)(\Sigma*\rho)^T
$
is uniformly elliptic for the densities under consideration.
This equation can be interpreted as the mean-field limit of a weakly
interacting particle system, as discussed in the Introduction. The term $v$
represents an external transport field acting on the particles, while
$V*\rho_t$ and $\Sigma*\rho_t$ describe the mean-field interaction in the
drift and diffusion coefficients, respectively. In this interpretation,
$X_t$ denotes the trajectory of a typical particle in the large-population
limit.

The above equation falls within the scope of
Theorem~\ref{thm:4converge}. Consequently, the introduction of an arbitrarily
small but sufficiently irregular external divergence-free drift field
$v\in C_tL^{d-}(\mathbb{T}^d)$ may lead to the existence of infinitely many
weak solutions starting from the stationary density $\rho_0\equiv 1$.
This regularity threshold is essentially sharp. Indeed, in the whole-space setting, with the regime
$v\in C_tL^p$ with $p>d$, one expects well-posedness under the usual
ellipticity and stability assumptions. In the additive-noise case
$\Sigma\equiv \operatorname{Id}$, the well-posedness theory   was developed in \cite[Theorem~4.3]{RZ21}. For general
distribution-dependent diffusion coefficient, we refer to  \cite[Theorem~4.1]{Zha25}.

Finally, applying Theorem~\ref{thm:sta}, one can also construct a divergence-free drift field
$
v\in L^{d-1-} 
$
for which the corresponding stationary nonlinear Fokker--Planck equation
admits non-unique stationary probability solutions.
 
\subsection{The porous-medium-type equation}

We next consider a nonlinear diffusion of porous-medium type. In   Assumption~\ref{def:ass}\textnormal{(1)}, let
$\beta:\mathbb{R}\to\mathbb{R}$ satisfy $\beta(0)=0$ and
$
0<\gamma\leq \beta'(r)\leq \gamma_1<\infty,
r\in\mathbb{R}.
$ We then choose
$
\sigma(\rho)
=
\sqrt{\frac{\beta(\rho)}{\rho}}\,\operatorname{Id},
V\equiv 0 .$
With this choice, equation \eqref{eq:fpe:i} reduces to the nonlinear
Fokker--Planck equation on $\mathbb{T}^d$
\begin{align*}
\partial_t\rho
-
\Delta\beta(\rho)
+
\operatorname{div}(v\rho)
=
0 .
\end{align*}
Here $\rho$ denotes the probability density associated with the corresponding
DDSDE
\begin{align*}
\dif X_t
=
v(t,X_t)\,\dif t
+
\sqrt{\frac{2\beta(\rho_t(X_t))}{\rho_t(X_t)}}\,\dif W_t .
\end{align*} 

The whole-space analogue of this model has been discussed in
Section~\ref{sec:rcase}. In the periodic setting, Theorems~\ref{thm:4converge}
and~\ref{thm:sta} apply to this equation. Consequently, one can construct
divergence-free drift fields $v\in C_tL^{d-}$ for which the equation admits infinitely many
weak solutions starting from the stationary density $\rho_0\equiv 1$, as well
as divergence-free drift fields for which the corresponding stationary problem
admits non-unique stationary probability solutions.

 \subsection{Keller--Segel equation}\label{sec:example:kse}

We consider aggregation equations $(\nu=0)$ and Keller--Segel-type models
$(\nu>0)$ of the form
\begin{align*}
\partial_t\rho
-
\nu\Delta\rho
+
\operatorname{div}\bigl((\nabla G*\rho)\rho\bigr)
+
\operatorname{div}(v\rho)
=
0 .
\end{align*}
The corresponding DDSDE is formally given by
\begin{align}
\dif X_t
=
v(t,X_t)\,\dif t
+
(\nabla G*\rho_t)(X_t)\,\dif t
+
\sqrt{2\nu}\,\dif W_t,
\notag
\end{align}
where $X_t$ represents the position of a moving biological cell at time $t$.
Here $G$ denotes the Green function of the Laplacian on the torus. Since the
Laplacian is not invertible on constants on $\mathbb{T}^d$, the inverse
Laplacian is understood on mean-zero functions. More precisely, $G$ is defined
by
\begin{align*}
-\Delta G=\delta_0-1,
\qquad
\int_{\mathbb{T}^d}G(x)\,\dif x=0 .
\end{align*}
Near the origin, $G$ has the same singular behavior as the classical Euclidean
potential; see, for instance, \cite{Jos19}. In particular,
\begin{align*}
|\nabla G(x)|
\lesssim
|x|^{1-d},
\qquad
d\geq 2 .
\end{align*} 

The Keller--Segel model \cite{KS70,KS71} is a fundamental equation in
chemotaxis and collective behavior, describing the motion of biological cells
or microorganisms attracted by chemical signals produced collectively by the
population itself. In this context, probability solutions are particularly
natural, since the density $\rho$ represents the distribution of particles or
cells and must therefore remain nonnegative with conserved total mass.

Keller--Segel equations with an additional incompressible transport field $v$ have
been extensively studied, mainly in connection with enhanced dissipation,
suppression of blow-up, and coupled fluid models, see for example \cite{BH17,IXZ21,HKY25}. In these works, the external
drift is typically smooth or has a special mixing structure and therefore plays
a stabilizing role. In contrast, our result focuses on the opposite mechanism:
we show that an arbitrarily small but sufficiently rough divergence-free drift
can destroy uniqueness of probability solutions.  

In the case $d=2$, one has
$
\nabla G\in L^{2,\infty}(\mathbb{T}^2;\mathbb{R}^2).
$
Therefore, by Theorem~\ref{thm:4converge}, there exists a divergence-free
drift field
$
v\in C_tL^{2-} 
$
such that the above equation admits infinitely many probability solutions
starting from the stationary density $\rho_0\equiv1$.

In the case $d\geq 3$,   since
$
G\in L^{\frac{d}{d-2},\infty},
$   the assumptions of
Theorem~\ref{thm:6converge} and Theorem~\ref{thm:sta} are satisfied with
$\gamma=2$, provided
$
1<\frac{d}{2}<d_0':=\frac{d_0}{d_0-1}$.
Consequently, for  $d_0<\frac{d}{d-2}$, there exists a
divergence-free drift field
$
v\in C_tL^{d_0} 
$
such that the corresponding equation with initial datum $\rho_0\equiv1$
admits non-unique probability solutions. Similarly, by
Theorem~\ref{thm:sta}, one can construct a time-independent divergence-free
drift field for which the stationary equation admits non-unique stationary
probability solutions.

In the Keller--Segel case with $\nu>0$, one expects well-posedness for
sufficiently regular external drifts. For instance, when
$v\in C_tL^r$ with $r>d$, the drift is
subcritical with respect to the parabolic scaling. In this regime, local
well-posedness and uniqueness of mild solutions in subcritical spaces of the
form $
C_tL^q,
\frac1q=\frac1d+\frac1r<\frac2d,
$
are expected; see, for instance, \cite{Kar99,BHN94}.
In dimension $d=3$, our construction reaches the sharp threshold suggested by
this scaling. On the one hand, for $v\in C_tL^r$ with $r>3$, one expects
uniqueness in the class
$
C_tL^q,
\frac1q=\frac13+\frac1r<\frac23 .
$
On the other hand, for every $1<r<3$, our result provides a divergence-free
drift $
v\in C_tL^r$ 
such that the equation admits multiple probability solutions in the class
$
C_tL^q,
\frac1q=1-\frac1r<\frac23 .
$
Thus the point
$
\left(\frac1r,\frac1q\right)
=
\left(\frac13,\frac23\right)
$
marks the critical threshold separating the expected uniqueness regime from
the non-uniqueness regime; see Figure~\ref{tu:sde} for an illustration.

\begin{figure}
    \centering
   
\begin{tikzpicture}[scale=4]
    \draw[->] (0,0) -- (1.1,0) node[right] {$\frac1r$};
    \draw[->] (0,0) -- (0,1.1) node[above] {$\frac1q$};

    \draw[red,thick] (0,1/3) -- (1/3,2/3);
     \draw[green,thick] (1/3,2/3) -- (1,0);

    \fill[red,opacity=0.3] (0,1/3)--(1/3,2/3) -- (1/3,0) -- (0,0) -- cycle;
    \fill[green,opacity=0.3] (1/3,1) --(1/3,2/3) --(1,0)-- (1,1) -- cycle;

    \node at  (0.7,0.6) {¬!};
    \node at (0.15,0.2) {!};

    \draw[dotted] (1/3,1) --  (1/3,0);

     \draw (1/3,2/3) node[right] {$(\frac13,\frac23)$} ;
    \draw (1,0) node[below] {$1$} -- +(0,0.02);
    
     \draw (0,1/3) node[left] {$\frac13$} -- +(0.02,0);
    \draw (0,1) node[left] {$1$} -- +(0.02,0);
    \draw (1/3,0) node[below] {$\frac13$} -- +(0,0.02);
    
    \filldraw[black] (1/3,2/3) circle (0.03mm);
\end{tikzpicture}
\caption{
Uniqueness and non-uniqueness regimes for probability solutions in
$C_tL^q$ to the three-dimensional Keller--Segel equation with external drift
$v\in C_tL^r$. \\
The red region corresponds to the expected well-posedness regime: for $r>3$
and $\frac1q\leq \frac13+\frac1r$, uniqueness is expected in the corresponding
subcritical class.  \\
The green region corresponds to the non-uniqueness regime
obtained in this paper: for $1<r<3$ and
$\frac1q\geq 1-\frac1r$, there exists a divergence-free drift
$v\in C_tL^r$ for which the equation admits multiple probability solutions.\\
The point $(\frac13,\frac23)$ marks the sharp threshold separating these
two regimes.
}
\label{tu:sde}
\end{figure}
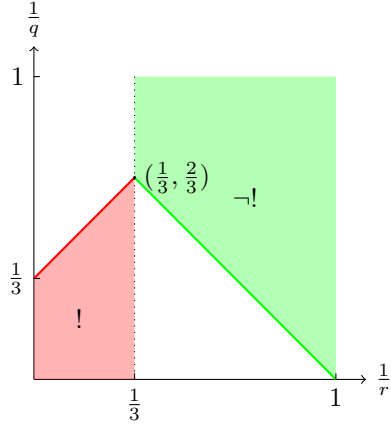

 \subsection{The 2D vorticity Navier--Stokes equation}

We next consider the two-dimensional case $d=2$. Let
$V\in L^{2,\infty}(\mathbb{T}^2;\mathbb{R}^2)$ and take
$\sigma=\operatorname{Id}$. Then \eqref{eq:fpe:i} becomes
\begin{align*}
\partial_t\rho
-
\Delta\rho
+
\operatorname{div}\bigl((V*\rho)\rho\bigr)
+
\operatorname{div}(v\rho)
=
0 .
\end{align*}
The corresponding DDSDE is
\begin{align*}
\dif X_t
=
v(t,X_t)\,\dif t
+
(V*\rho_t)(X_t)\,\dif t
+
\sqrt{2}\,\dif W_t .
\end{align*}
In the case $v\equiv0$, the existence and conditional uniqueness of the
corresponding SDE and Fokker--Planck equation have been studied in
\cite{Kry20,Kry20b,RZZ25}.

In particular, if we choose the Biot--Savart kernel
\begin{align*}
V=\nabla^\perp G\in L^{2,\infty}(\mathbb{T}^2;\mathbb{R}^2),
\end{align*}
where  $G$ is the Green function on the torus, then the equation reduces to the
two-dimensional vorticity Navier--Stokes equation with an additional drift:
\begin{align*}
\partial_t\rho
-
\Delta\rho
+
\operatorname{div}\bigl((\nabla^\perp G*\rho)\rho\bigr)
+
\operatorname{div}(v\rho)
=
0 .
\end{align*} 
The two-dimensional vorticity Navier--Stokes equation is a fundamental model in
fluid mechanics and turbulence theory, describing the evolution of the
vorticity of an incompressible fluid. The well-posedness theory for the
two-dimensional vorticity Navier--Stokes equation in the class $C_tL^1$ is now
classical; see, for instance, \cite{MB03,BA94,GG05}.  The probabilistic formulation of the 2D vorticity
Navier--Stokes equation in terms of nonlinear McKean--Vlasov SDEs has been
studied, for instance, in \cite{BRZ25, RZZ25}.

On the other hand, by Theorem~\ref{thm:4converge}, there exists a
divergence-free drift field
$
v\in C_tL^{2-}
$
such that the above vorticity equation with the additional drift admits
infinitely many probability solutions in the class $C_tL^1$, all starting from
the stationary density $\rho_0\equiv1$.

  \subsection{Landau-type equation on $\mathbb{T}^d$, $d\geq 3$}
\label{sec:example:landau}

We next consider a Landau-type equation on the torus. Let
$
\sigma(\rho)=\sqrt{a*\rho},
V=2\div a,
$
where the matrix-valued kernel $a$ has the Coulomb-type singularity
\begin{align*}
a(x)
=
\left(
\operatorname{Id}
-
\frac{x\otimes x}{|x|^2}
\right)
|x|^{2-d}
\end{align*}
near the origin. More precisely, $a$ is understood as a periodic
matrix-valued function on $\mathbb{T}^d$ which coincides with the above
kernel in a neighborhood of the origin.  

With this choice, equation \eqref{eq:fpe:i} becomes the following
Landau-type equation:
\begin{align*}
\partial_t\rho
-
\operatorname{div}
\left[
(a*\rho)\nabla\rho
-
(a*\nabla\rho)\rho
\right]
+
\operatorname{div}(v\rho)
=
0 .
\end{align*}
The corresponding DDSDE is formally given by
\begin{align*}
\dif X_t
=
v(t,X_t)\,\dif t
+
2(a*\nabla\rho_t)(X_t)\,\dif t
+
\sqrt{2 (a*\rho_t)(X_t)}\,\dif W_t .
\end{align*} 

The Landau equation is a fundamental kinetic model describing the evolution of
dilute plasmas under Coulomb-type interactions, where the dynamics are
dominated by the accumulation of grazing collisions. Classical references
include \cite{Lan36,Vil98,Vil02}. In the genuine kinetic setting, the variable
$x$ here should be interpreted as the velocity variable. From this viewpoint,
it is not physically natural to impose periodicity in the velocity variable.
However, due to the limitations of our current method, we restrict ourselves
to the periodic setting on the torus.

Since  $
a\in
L^{\frac{d}{d-2},\infty} $,  the assumptions of Theorem~\ref{thm:6converge} and
Theorem~\ref{thm:sta} are satisfied with $\gamma=2$, provided
$
1<\frac{d}{2}<d_0',
$ i.e. $d_0<\frac{d}{d-2}$.
Consequently, there exists a divergence-free drift field
$
v\in C_tL^{d_0} 
$
 in the above range such that the corresponding Landau-type
equation admits non-unique weak solutions. Moreover, by applying
Theorem~\ref{thm:sta}, one can also construct a time-independent
divergence-free drift
for which the  equation admits non-unique
stationary probability solutions. 
 
\subsection{Further discussion}

Our approach is flexible and is expected to be adaptable to a much broader
class of diffusion mechanisms than those treated in the present paper. With
suitable modifications, one may hope to extend the construction to nonlinear
diffusion operators such as the $p$-Laplacian
$
\operatorname{div}\bigl(|\nabla\rho|^{p-2}\nabla\rho\bigr),
$
nonlocal diffusion operators such as the fractional Laplacian
$(-\Delta)^\alpha$, and genuinely degenerate porous-medium diffusions such as
$
\Delta \rho^m .
$
These operators arise naturally in a variety of physical models, including
non-Newtonian fluids, transport in heterogeneous media, and anomalous
diffusion. We do not pursue these extensions here, since they would require
additional estimates adapted to the specific structure of each diffusion
operator.

In this paper, we restrict our attention to the torus and whole-space settings.
Nevertheless, the method is expected to be adaptable to more general domains
equipped with suitable boundary conditions. A main additional ingredient would
be the use of Bogovskii operators \cite{Bog79}, which provide a convenient
tool for solving divergence equations while taking boundary conditions into
account. The presence of boundaries, however, also introduces new technical
difficulties.

\section{Notations and Preliminaries}\label{sec:notation}
 Let $T>0$, $\mN_{0}:=\mN\cup \{0\}$. Throughout the manuscript, we write  $\mT^d = \mR^d/\mZ^d$ for the $d$-dimensional flat torus.  
We employ the notation $a\lesssim b$ if there exists a constant $c>0$ such that $a\leq cb$. 
  Given a Banach space $E$ with a norm $\|\cdot\|_E$, we write $C_tE=C([0,T];E)$ for the space of continuous functions from $[0,T]$ to $E$, equipped with the supremum norm. For $p\in [1,\infty]$ we write $L^p_tE=L^p([0,T];E)$ for the space of $L^p$-integrable functions from $[0,T]$ to $E$, equipped with the usual $L^p$-norm.
For $\alpha\in(0,1)$ we  define $C^\alpha_tE$ as the space of $\alpha$-H\"{o}lder continuous functions from $[0,T]$ to $E$, endowed with the norm $\|f\|_{C^\alpha_tE}=\sup_{s,t\in[0,T],s\neq t}\frac{\|f(s)-f(t)\|_E}{|t-s|^\alpha}+\sup_{t\in[0,T]}\|f(t)\|_{E},$   and write  $C_t^\alpha$ in the case when $E=\mathbb{R}$. 
   We use $L^p$ to denote the set of  standard $L^p$-integrable functions on $\mathbb{T}^d$.
For $s>0$, $p>1$ we set $W^{s,p}:=\{f\in L^p; \|(I-\Delta)^{\frac{s}{2}}f\|_{L^p}<\infty\}$ with the norm  $\|f\|_{W^{s,p}}=\|(I-\Delta)^{\frac{s}{2}}f\|_{L^p}$.   We use $L^{p,\infty}$ to denote the Lorentz space
\begin{align*}
\|f\|_{L^{p,\infty}}
:=
\sup_{\lambda>0}
\lambda
\,|\{x\in\mathbb{T}^d:\ |f(x)|>\lambda\}|^{1/p}.
\end{align*}
For $N\in \N_0 $, $C_x^N$ denotes the space of $N$-times differentiable functions equipped with the norm
	$$
	\|f\|_{C_x^N}:=\sum_{\substack{|\alpha|\leq N, \alpha\in\N^{d}_{0} }}\| D^\alpha f\|_{ L^\infty_x}.$$ Similarly, if the norm is taken in space-time, we use $C^{N}_{t,x}$. 
  
   For a matrix-valued function
$
A=(A^{ij})_{1\leq i,j\leq d},
$
we define
$
(\div A)_i:=\sum_{j=1}^d\partial_jA^{ij}.
$

We define  the projection  $\mathbb{P}_{\neq0} f := f -\int_{\mathbb{T}^d} f\dif x$.
We denote by $\dif x$ the Lebesgue measure on the torus.
 We denote by $\mathcal{L}$ the law of  random variable.

\subsection{Some technical tools}\label{tamr}
We collect some technical tools used in the construction of convex integration schemes. 

 We define $\mathcal{R}_1:=\nabla\Delta^{-1}$ as a right
inverse of the $\div$ operator, i.e. $\div(\mathcal{R}_1v) = v$ for scalars $v$ with $\int_{\mathbb{T}^d} v\dif x = 0$.   Here we use the notation  $\mathcal{R}_1v:=\mathcal{R}_1(v-\int v\dif x)$ for a general scalar function $v$. Then  since $\mathcal{R}_1$ is a Calder\'on-Zygmund operator, we have 
\bl\label{bb_1}
 Let $1\leq p \leq\infty$. For any vector field 
$f \in C_0^\infty(\mathbb{T}^d; \mathbb{R})$, $\sigma\in\mathbb{N}$,
$$\|\mathcal{R}_1f(\sigma\cdot)\|_{ L^p}\lesssim\sigma^{-1} \| f\|_{ L^p}.$$
\el
We introduce the  bilinear version $\mathcal{B}_1: C^\infty(\mathbb{T}^d; \mathbb{R}) \times C_0^\infty(\mathbb{T}^d; \mathbb{R})\to C^\infty(\mathbb{T}^d; \mathbb{R}^{d})$  by $$\mathcal{B}_1(v,f)=v\mathcal{R}_1f-\mathcal{R}_1\(\nabla v\cdot\mathcal{R}_{1}f+\int vf\dif x\).$$

\bl$($\cite[Lemma 3.3]{BCDL21}$)$\label{bb1_1}
Let $1 \leq p \leq\infty$. For any $v \in C^\infty(\mathbb{T}^d; \mathbb{R})$ and $f\in C_0^\infty(\mathbb{T}^d; \mathbb{R})$, we have $\div(\mathcal{B}_1(v,f))=vf-\int_{\mathbb{T}^d}vf\dif x,$ and  for $\sigma\in\mathbb{N}$,
$$\|\mathcal{B}_1(v,f(\sigma\cdot))\|_{W^{k,p}}\lesssim\sigma^{k-1}\|v\|_{C^{k+1}}\|f\|_{W^{k,p}}.$$
\el

 Then we introduce the improved H\"older’s inequality  by using the additional decorrelation between frequencies.
\bl$($\cite[Theorem B.1]{CL22}$)$\label{ihiot}
Let $d\geq2,p \in [1, \infty]$ and $a, f : \mathbb{T}^d \to \mathbb{R}$ be smooth functions. Then for any $\sigma\in\mathbb{N}$,
$$| \|a f(\sigma\cdot)\|_{L^p}-\|a\|_{L^p}\|f\|_{L^p} |\lesssim\sigma^{-1/p}\|a\|_{C^1}\|f\|_{L^p}.$$
\el

Finally, we record the following weak Young inequality.

\begin{lemma} \label{lem:weak-young}
Let $d\geq 2$ and let $1<p,q,r<\infty$ satisfy
$
\frac1q=\frac1p+\frac1r-1>0 .
$
Then, for all $f\in L^p(\mathbb{T}^d)$ and
$g\in L^{r,\infty}(\mathbb{T}^d)$, one has
\begin{align}
\|f*g\|_{L^q(\mathbb{T}^d)}
\lesssim
\|f\|_{L^p(\mathbb{T}^d)}
\|g\|_{L^{r,\infty}(\mathbb{T}^d)} .
\label{eq:hls}
\end{align}
\end{lemma}

\section{Infinitely many evolutions from a stationary state}\label{cogpss3}
In this section, we prove our main result, Theorem~\ref{thm:4converge}. 
We first establish the first statement at the PDE level. 
For any fixed diffusion coefficient $\sigma$ satisfying Assumption~\ref{def:ass} and \eqref{eq:divdiv}, we construct a distribution-independent divergence-free drift $v \in L^r_tL^p \cap C_tL^s$ 
such that the corresponding nonlinear Fokker--Planck equation \eqref{eq:fpe:i} admits infinitely many distinct positive solutions $\rho^i \in C_tL^1$, satisfying moreover $v\rho^i \in L^1_tL^1.$

To achieve this, we apply a convex integration scheme to the Fokker--Planck equation \eqref{eq:fpe:i}. 
The overall iterative procedure follows the framework developed in previous works such as \cite{LRZ25,LR25}. 
However, the present setting requires handling general diffusion operators together with the construction of infinitely many distinct solutions. 
Accordingly, in the sequel we mainly focus on the new ideas and techniques specific to this setting, while only briefly discussing those parts of the argument that are analogous to earlier works.  

\subsection{Preparations} 
From now on, we write
$
A(\rho):=\sigma\sigma^T(\rho).
$
The following lemma collects several estimates that will be used later.

\begin{lemma}\label{lem:A}
Under Assumption~\ref{def:ass}, it holds that
\begin{align}
\|A(\rho)&-A(\tilde{\rho})\|_{C_{t,x}^1}
\lesssim
   \|\rho- \tilde{\rho} \|_{C_{t,x}^1}+\|\rho- \tilde{\rho} \|_{C_{t,x}^0}(\|\rho \|_{C_{t,x}^1}+\|\tilde\rho \|_{C_{t,x}^1}),
\notag\\
\|A(\rho)\|_{C_{t,x}^0}
&\lesssim 1,
\notag\ \
\|A(\rho)\|_{C_{t,x}^1}
\lesssim 1+\|\rho\|_{C_{t,x}^1},\ \ \|A(\rho)\|_{C_{t,x}^2}
\lesssim 1+\|\rho\|_{C_{t,x}^1}^2+\|\rho\|_{C_{t,x}^2}.
\notag
\end{align}
\end{lemma}

\begin{proof}
We first consider case \textnormal{(1)}. Since
$
\sigma(t,x,\rho)=\sigma_1(t,x,\rho_t(x))
$
and
$
\sigma_1\in C_b^2([0,T]\times\mathbb{T}^d\times\mathbb{R}),
$
we have for $0\leq i\leq d$ ($\partial_0$ is regarded as the derivative  to  time  $t$), 
\begin{align*}
|\partial_i\sigma(t,x,\rho)&-\partial_i\sigma(t,x,\tilde{\rho})|\\
&\lesssim|(\partial_i\sigma_1)(t,x,\rho)-(\partial_i\sigma_1)(t,x,\tilde{\rho})|+|(\partial_\rho\sigma_1)(t,x,\rho)\partial_i\rho-(\partial_\rho\sigma_1)(t,x,\tilde{\rho})\partial_i\tilde\rho|\\
&\lesssim 
\|\rho- \tilde{\rho} \|_{C_{t,x}^1}+\|\rho- \tilde{\rho} \|_{C_{t,x}^0}\|\rho \|_{C_{t,x}^1},
\\
\|\sigma(t,x,\rho)\|_{C_{t,x}^1}
&\lesssim
1+\|\rho\|_{C_{t,x}^1},\\
\|\sigma(t,x,\rho)\|_{C_{t,x}^2}
&\lesssim
1+\|\rho\|_{C_{t,x}^1}^2+\|\rho\|_{C_{t,x}^2}.
\end{align*}

Next, in case \textnormal{(2)}, it suffices to consider the case that $\sigma(\rho)=\sigma_2*\rho$. Young's inequality yields
\begin{align*}
\|\sigma(t,x,\rho)-\sigma(t,x,\tilde{\rho})\|_{C_{t,x}^1}
&\lesssim
\|\sigma_2\|_{C_{t}^1C_{x}^0}
\|\rho_t-\tilde{\rho}_t\|_{C_{t,x}^1},
\\
\|\sigma(t,x,\rho)\|_{C_{t,x}^0}
&\lesssim
\|\sigma_2\|_{C_{t,x}^0}\|\rho_t\|_{L^1}
\lesssim 1,
\\
\|\sigma(t,x,\rho)\|_{C_{t,x}^i}
&\lesssim
1+\|\rho\|_{C_{t,x}^i},\ i=1,2.
\end{align*}

Case \textnormal{(3)} follows directly by combining the estimates from the previous two cases.
The corresponding estimates for $A$ then follow immediately.
\end{proof}

\subsection{The iteration procedure}
Let $0<\epsilon_0\leq 1$ be given. Without loss of generality, we may replace $\epsilon_0$ by
$
\epsilon_0\wedge \frac14.
$ The iteration is indexed by a parameter
$
q\in\mathbb{N}_0.
$
We consider an increasing sequence
$
\{\lambda_q\}_{q\in\mathbb{N}_0}\subset\mathbb{N}
$
and a sequence
$
\{\delta_q\}_{q\in\mathbb{N}_0}\subset(0,1]
$
defined by
\begin{align*}
\lambda_q
=
a^{b^q},
\ 
q\geq0,
\ 
\delta_q
=
(\frac{\epsilon_0}{2})^{d+1}
\lambda_1^{2\beta}
\lambda_q^{-2\beta},
\ 
q\geq1,
\ 
\delta_0=1,
\end{align*}
where $\beta>0$ will be chosen sufficiently small, while $a,b$ will be chosen sufficiently large.
Moreover, by imposing
\begin{align}
a^{2(b-1)\beta/(d+1)}>2,
\label{ieq:4ab2}
\end{align}
we obtain
\begin{align*}
\sum_{q\geq1}\delta_q^{1/(d+1)}
&\leq
\frac{\epsilon_0}{2}
\sum_{q\geq1}
a^{(1-q)2(b-1)\beta/(d+1)}
\leq
\frac{\epsilon_0}{2}
\cdot
\frac1{1-a^{-2(b-1)\beta/(d+1)}}
<
\epsilon_0.
\end{align*}

Without loss of generality, we assume $T=1$ from now on. 
At each step $q$, we construct a family
$(v_q,\rho_q^i,M_q^i)_{i\in\mathbb{N}}$
solving the system
\begin{align}\label{eq:4qth}
\partial_t\rho_q^i
-
\div\div\bigl(A(\rho_q^i)\rho_q^i\bigr)
+
\div(v_q\rho_q^i)+\div((V*\rho_q^i)\rho_q^i)
&=
-\div M_q^i,
\\
\div v_q&=0.
\notag
\end{align}
The sequence $\{\rho_q^i\}_{q\in\mathbb{N}_0}$ converges to a limit $\rho^i$, while the stress fields $\{M_q^i\}_{q\in\mathbb{N}_0}$ converge to $0$ in a suitable topology as $q\to\infty$.

To incorporate the initial condition, we require that
$\rho_q^i=1$ on $[0,T_q]$, where
$
T_q:=\frac13-\sum_{1\leq r\leq q}\delta_r^{1/2}.
$
By \eqref{ieq:4ab2}, we have
$
0<T_q\leq \frac13.
$
Here and throughout the paper, we use the convention
$
\sum_{1\leq r\leq0}:=0.
$

Let $\alpha\in(0,1)$ be fixed later. 
For each $q\in\mathbb{N}_0$, we define
$
N_q:=\min\bigl\{i\in\mathbb{N}:2^i\geq\lambda_q^\alpha\bigr\},
$
which satisfies
$
\lambda_q^\alpha
\leq
2^{N_q}<2\lambda_q^\alpha.$
At the iteration step $q+1$, we construct new perturbations to eliminate the first $N_{q+1}$ stress terms
$M_q^i$, $1\leq i\leq N_{q+1}$,
while ensuring that the remaining stress terms remain sufficiently small.

We initialize the iteration by setting
\begin{align*}
\rho_0^i(t,x)
&=
1+\frac{\sin2\pi x_1}{4^i}\chi_0(t),
 \ 
v_0=\overline v,\\
M_0^i(t,x)
&=
\partial_t\chi_0(t)\frac{\cos2\pi x_1}{4^i\cdot 2\pi}(1,0,\dots,0)
-\overline v(\rho_0^i-1)\\
&\ +
\div(A(t,x,\rho_0^i) \rho_0^i-A(t,x,1))-(V*\rho_0^i)\rho_0^i+V*1,
\end{align*}
where $x=(x_1,\dots,x_d)$ and $\chi_0$ is a smooth function satisfying
$
\chi_0(t)=0 $ on $
 [0,\frac13 ], \chi_0(t)=1$ on $[\frac23,1 ].
$   It is easy to verify that $(v_0,\rho^{i}_0,  M^{i}_0)_{i\in\mN}$ is a solution to \eqref{eq:4qth} by using the facts that $\div \overline v=0$, $\div\div(A(t,x,1))=0$ and $\div(V*1)=0$. By definition, we know that $\rho_0^i=1$ on $[0,T_0]$, and then $M_0^i=0$ on $[0,T_0]$.

Then   we have by weak Young's inequality
\begin{align*}
  \|  \div(A(t,x,\rho_0^i) \rho_0^i-A(t,x,1))\|_{C_{t,x}^0}&\lesssim \|   A(t,x,\rho_0^i) -A(t,x,1)\|_{C_{t,x}^1}+\|  A(t,x,\rho_0^i) (\rho_0^i-1) \|_{C_{t,x}^1}\\
  &\lesssim \|\rho_0^i-1 \|_{C_{t,x}^1}\lesssim 4^{-i},\\
\|(V*\rho_0^i)\rho_0^i-V*1\|_{L_{t}^1L^1}&\lesssim \|(V*\rho_0^i-V*1)\rho_0^i\|_{L_{t}^1L^1}+\|( V*1)(\rho_0^i-1)\|_{L_{t}^1L^1}\\
&\lesssim \|\rho_0^i-1 \|_{C_{t,x}^0}\lesssim 4^{-i}.
\end{align*} By choosing $C_0>0$ sufficiently large (depending on $\epsilon_0$), we obtain by Lemma \ref{lem:A}
\begin{align}
\| \rho_0^i-1\|_{C_{t,x}^2}
& 
\leq
4^{-i}C_0^{1/d_0'},\ \ \|v_0\|_{C_{t,x}^1}\leq C_0^{1/d_0},
\ \
\|M_0^i\|_{L_t^1L^1}
\lesssim
4^{-i}
\leq
(\frac{\epsilon_0}{2})^{d+1}4^{-i}C_0.
\label{bd:4rho_0,M_0}
\end{align}

With the above assumptions in hand, our main iteration relies on the first step of iteration and reads as follows:

\bp\label{prop:case4} 
Under the assumption of Theorem \ref{thm:4converge}, there exist $d+1>d_0>2>d_0'>1$ with $\frac{1}{d_0}+\frac{1}{d_0'}=1$ and a choice of parameters $a,b,\alpha,\beta$  such
that the following holds: Let $(v_q,\rho^{i}_q,  M^{i}_q)_{i\in\mN}$ be a solution to the system \eqref{eq:4qth} satisfying  $\int\rho^{i}_q\dif x=1$,
\begin{align}\label{bd:4vql2}
\|v_q\|_{L^{d_0}_tL^{d_0}}\leq C_vC_0^{1/d_0}\sum_{m=0}^q\delta_{m}^{1/d_0},
\end{align}
for some universal constant $C_v\geq1$, and
\begin{align}
\|v_q\|_{C_{t,x}^1}\leq C_0^{1/d_0}\lambda_q^{d+4},\
  \|\rho^{i}_{q}-1\|_{C_{t,x}^1}+\lambda_q^{-2}\|\rho^{i}_{q}-1\|_{C_{t,x}^2}\leq 4^{-i}C_0^{1/d_0'}\lambda_q^{d+4},
\label{bd:4rhoqc1}\\
 \|M^{i}_q\|_{L^1_tL^1}\leq C_02^{-i}\delta_{q+1},&\label{bd:4rql1}\\
 M^{i}_{q}=M^{i}_{0}-(v_{q}-v_0)(\rho^{i}_0-1),\ \rho_q^{i}=\rho_0^{i}, \ {\rm for\ }  i>N_{q},\label{bd:4mqi>nq}\\
 \rho^{i}_q(t)=1, M^{i}_q(t)=0\ {\rm on}\ [0,T_q].&\label{bd:4mq=rhoq=0}
\end{align}
Then there exists 
$(v_{q+1},\rho^{i}_{q+1},  M^{i}_{q+1})_{i\in\mN}$ which solves \eqref{eq:4qth} and satisfies \eqref{bd:4vql2}-\eqref{bd:4mq=rhoq=0} at the level $q+1$ and
\begin{align}
\|v_{q+1}-v_{q}\|_{L^{d_0}_tL^{d_0}}\leq
C_vC_0^{1/d_0}\delta_{q+1}^{1/d_0},\label{bd:4vq+1-vqldd}\ \
\|\rho^{i}_{q+1}-\rho^{i}_{q}\|_{L^{d_0'}_tL^{d_0'}}\leq
C_v C_0^{1/{d_0'}}2^{-i/d_0'}\delta_{q+1}^{1/{d_0'}}.
\end{align}
Moreover, we have for some $\epsilon>0$ small enough
\begin{align}
\|v_{q+1}-v_{q}\|_{L^{r}_tL^{p}}+\|v_{q+1}-v_q\|_{C_tL^s}\leq
\delta_{q+1}^{1/d_0},\label{bd:4vq+1-vqlpr}\\
\|\rho^{i}_{q+1}-\rho^{i}_{q}\|_{C_tL^{1+\epsilon}}+ \|\rho^{i}_{q+1}-\rho^{i}_{q}\|_{L_t^1W^{1,1+\epsilon}}\leq
4^{-i}\delta_{q+1}^{1/d_0'},\ \  \inf_{t\in [0,1]}(\rho^{i}_{q+1} - \rho^{i}_q)\geq - \delta_{q+1}^{1/d_0'} .\label{bd:4rhoq+1-rhoql1}
\end{align}
\ep

 The proof of our main iteration Proposition \ref{prop:case4} will be stated in the following section.

\subsection{Proof of   Theorem \ref{thm:4converge}}\label{sec:proofmaireuslt}
  
We intend to start the iteration from 
$(v_0, \rho^{i}_0,  M^{i}_0)_{i\in\mN}$ which are defined as above. 
By \eqref{bd:4rho_0,M_0},  \eqref{bd:4vql2}-\eqref{bd:4mq=rhoq=0} are satisfied as $\delta_0=1,\delta_1=(\frac{\epsilon_0}{2})^{d+1}$.

Next, we use Proposition \ref{prop:case4} to build inductively $(v_q,\rho^{i}_q, M^{i}_q)_{i\in\mN}$ for every $q \geq 1$. By \eqref{ieq:4ab2} and \eqref{bd:4vq+1-vqldd}-\eqref{bd:4rhoq+1-rhoql1}, the sequence $\{(v_q,\rho^{i}_q)_{i\in\mN}\}_{q\in \N}$ is
Cauchy in 
\begin{align*}
    \(L^r([0,1];L^p)\cap L^{d_0}([0,1]\times \mT^d)\cap C([0,1];L^s)\)\times \(L^{d_0'}([0,1]\times \mT^d)\cap C([0,1];L^{1+\epsilon})\cap L^1([0,1];W^{1,1+\epsilon})\)^{\mN}\end{align*}
     and we denote by $(v,\rho^{i})$ its limit.    
By \eqref{bd:4rql1}, \eqref{bd:4mq=rhoq=0} and a similar calculation as \eqref{arhorho-arhorho}, it is easy to verify $(\rho^{i},v)$  solves   \eqref{eq:fpe:i}.  Then we know that $\int \rho^{i}\dif x=1$.

 By \eqref{ieq:4ab2}, \eqref{bd:4vq+1-vqldd} and \eqref{bd:4rhoq+1-rhoql1} we have
\begin{align} 
 \left|   \|\rho^i-1\|_{C_tL^1}- \| \rho^{i}_0-1\|_{C_tL^1}\right|\leq \sum_{q=0}^\infty\|\rho^{i}_{q+1}-\rho^{i}_{q}\|_{C_tL^1}\leq
4^{-i}\sum_{q=0}^\infty\delta_{q+1}^{1/{d_0'}}\leq 4^{-i-1},\label{rhoi-rhoi0}\\
\inf_{t\in [0,1]}\rho^i \geq \inf_{t\in [0,1]} \rho_0^i +
\sum_{q=0}^\infty \inf_{t\in [0,1]}(\rho_{q+1}^i- \rho_q^i) \geq1-\frac14-\sum_{q=0}^\infty \delta_{q+1}^{1/d_0'}\geq \frac12,\notag
\end{align}
at which point, together with the fact that $\| \rho^{i}_0-1\|_{C_tL^1}=\frac{2}{\pi4^i}$, we obtain $\rho^i\to 1$ in $ C_tL^1$ as $i\to\infty$, and  
 $\rho^{i}$ is nonnegative on $\mathbb{T}^d$. Moreover, for $i>j$, it holds that 
 \begin{align*}
     \| \rho^{i}_0-\rho^{j}_0\|_{C_tL^1}\geq \frac{3}{2\pi4^j},
 \end{align*}
 which together with \eqref{rhoi-rhoi0} implies that $\rho^i$
 do not coincide with each other. 

 By \eqref{ieq:4ab2} and \eqref{bd:4vq+1-vqlpr} we obtain that  \begin{align*}
    \|v -\overline v \|_{L^{r}_tL^{p}}+\|v -\overline v \|_{C_tL^s}\leq  \sum_{q\geq0}( \|v_{q+1}-v_{q}\|_{L^{r}_tL^{p}}+\|v_{q+1}-v_q\|_{C_tL^s})\leq
\sum_{q\geq0}\delta_{q+1}^{1/d_0}\leq\epsilon_0.
\end{align*}
 We finish the proof of the first statement. 

 For the second statement,
by \eqref{bd:4rhoqc1}, \eqref{bd:4vq+1-vqldd} and interpolation, we have $|v| \in L_{t}^{d_0(1+\epsilon)}L^{d_0(1+\epsilon)}$ for some $\epsilon > 0$ sufficiently small. Since $\rho^i \in L_{t}^{d_0'}L^{d_0'}$, we deduce that $|v|^{1+\epsilon} \rho^i \in L_{t}^{1}L^{1}$.
Moreover, since
$
t\mapsto \rho^i(t)
$
is continuous on $[0,1]$, we may apply the superposition principle to the linear diffusion operator
\begin{align*}
L_{\rho^i}
:=
 \sigma\sigma^T(\rho^i):\nabla^2 
+
(v+V*\rho^i)\cdot\nabla,
\end{align*}
where the density $\rho^i$ is regarded as fixed.

More precisely, let $C([0,1];\mathbb{T}^d)$ denote the space of continuous paths, equipped with its Borel $\sigma$-algebra and the natural filtration generated by the canonical process
$
\Pi_t$, $t\in[0,1],
$
defined by
$
\Pi_t(\omega):=\omega(t),\ 
\omega\in C([0,1];\mathbb{T}^d).
$
By the superposition principle \cite[Section~7.2]{Tre14}, there exists a probability measure $\mathbf{Q}^i$ on $C([0,1];\mathbb{T}^d)$ which is a martingale solution associated with $L_{\rho^i}$ in the sense that, for every smooth function $f$ on $\mathbb{T}^d$, the process
\begin{align*}
f(\Pi_t)
-
f(\Pi_0)
-
\int_0^t
L_{\rho^i}f(\Pi_s)\,\dif s
\end{align*}
is a $\mathbf{Q}^i$-martingale. Moreover,
$
\dif\mathbf{Q}^i\circ \Pi_t^{-1}
=
\rho^i(t)\,\dif x,
\ 
t\in[0,1].
$

Since $\{\rho^i\}_{i\in\mathbb{N}}$ is a family of distinct solutions to \eqref{eq:fpe:i}, it follows that $\{\mathbf{Q}^i\}_{i\in\mathbb{N}}$ is a family of distinct martingale solutions.

\section{Proof of Proposition \ref{prop:case4}}\label{sec:convex}
In this section, we extend the convex integration scheme to a system of infinitely many equations in order to prove Proposition~\ref{prop:case4}. 
At step $q+1$, we focus on the first $N_{q+1}$ Fokker--Planck equations in \eqref{eq:fpe:i}, and construct new perturbations to eliminate the corresponding stress terms. 
At the same time, the remaining equations with indices larger than $N_{q+1}$ are left unchanged and shown to remain sufficiently small so as to satisfy the required estimates at level $q+1$.

\subsection{Mollification}\label{sec:mpll4}
For a sufficiently small parameter $\alpha\in(0,1)$ to be chosen later, we define
$
l:=\lambda_{q+1}^{-7\alpha/4}\lambda_q^{-3d/2-6}.
$
Then
\begin{align}
l^{-1}
\leq
\lambda_{q+1}^{2\alpha},
 \ \
l\lambda_q^{3d+12}
\leq
\lambda_{q+1}^{-3\alpha/2},\ \ \lambda_{q+1}^{-\alpha/2}
\ll
\delta_{q+2},
\label{para42}
\end{align}
provided that
\begin{align}
\alpha b>6d+24,
\ \
\alpha>4\beta b.
\label{alphab>}
\end{align}

To avoid the loss of derivatives, we first mollify the first $N_{q+1}$ equations.
Let
$
\phi_l:=\frac1{l^d}\phi\bigl(\tfrac{\cdot}{l}\bigr)
$
be a family of standard radial mollifiers on $\mathbb{R}^d$, and let
$
\varphi_l:=\frac1l\varphi\bigl(\tfrac{\cdot}{l}\bigr)
$
be a family of standard mollifiers supported in $(0,1)$.
We define  
\begin{align}
v_l=(v_{q }*_x\phi_l)*_t\varphi_l,\ \ 
 \rho_l^i=(\rho_{q}^i*_x\phi_l)*_t\varphi_l,\ \ M_l^i=(M_{q }^i*_x\phi_l)*_t\varphi_l.\notag
\end{align}
  For the mollification around $t = 0$,  since $ \rho_q^i$ and $M_q$ are constants around $t = 0$, we  can directly extend these definitions  to $t\leq 0$ by their values at $t=0$. For $v_q$, we also extend the definition to $t\leq0$ by their values at $t=0$.
Since $\rho_q^i$ and $M_q^i$ are constant near $t=0$, this causes no problem.
Moreover,    we know that $ \rho_l^i=1 $ and $ M_l^i=0$ on $[0, T_{q+1}]$.
  
By straightforward calculations  we obtain 
\begin{align}
\partial_t \rho_l^i-\div\div(A(\rho_l^i)\rho_l^i)+\div(v_l \rho_l^i)+\div((V*\rho_l^i)\rho_l^i)&=-\div (M^i_l+M^i_{com}),\ \ \div v_l=0,\label{eq:4v_l}
\end{align} where
 \begin{align*}
M_{com}^i&:=-v_l \rho^i_l+(v_{q } \rho^i_{q})*_x\phi_l*_t\varphi_l+\div(A(\rho_l^i) \rho_l^i)-(\div(A(\rho_q^i) \rho_q^i))*_x\phi_l*_t\varphi_l\\
&\ \ \ \   -(V*\rho_l^i)\rho_l^i+((V*\rho_q^i)\rho_q^i)*_x\phi_l*_t\varphi_l.
\end{align*} 
 Moreover,    we know that   $ M_{com}^i=0$ on $[0, T_{q+1}]$.

Finally, by the standard mollification estimates, the space-time Sobolev embedding
$
W^{d+\frac43,1}\hookrightarrow L^\infty,
$
and \eqref{bd:4rql1}, we obtain for every $N\geq0$,
\begin{align}
\|M_l^i\|_{C_{t,x}^N}
&\lesssim
l^{-d-\frac43-N}
\|M_q^i\|_{L_t^1L^1}
\lesssim
C_02^{-i}l^{-d-\frac43-N}.
\label{bd:4mlcn}
\end{align}

\subsection{Construction of $v_{q+1}$}\label{sec:4defq+1}
As discussed above, we apply the convex integration scheme only to the first $N_{q+1}$ equations. 
To construct the perturbations for $v_q$ and $\rho_q^i$, $1\leq i\leq N_{q+1}$, we employ the $L^{d_0}$-based building blocks introduced in Appendix~\ref{gij}, for some exponent $d_0>1$ to be specified later. 
We also recall the parameters
$
\lambda,\ r_\perp,\ r_\parallel,\ \eta,\ \mu,\ \sigma
$
introduced in Appendix~\ref{gij}. Their precise choice will be specified subsequently.

Let
$
\chi\in C_c^\infty(-\frac34,\frac34)
$
be a nonnegative function such that
$
\sum_{n\in\mathbb{Z}}\chi(t-n)=1
$
for every $t\in\mathbb{R}$.
Let
$
\tilde{\chi}\in C_c^\infty(-\frac45,\frac45)
$
be a nonnegative function satisfying
$
\tilde{\chi}=1
$
on
$
[-\frac34,\frac34]
$
and
$
\sum_{n\in\mathbb{Z}}\tilde{\chi}(t-n)\leq2.
$

For $1\leq i\leq N_{q+1}$, we fix the parameters
$
\zeta^i:=20\cdot2^i\delta_{q+2}^{-1},
$
and consider the two disjoint sets
$
\Lambda^1,\Lambda^2
$
defined in Appendix~\ref{gij}. 
We abuse the notation
$
\Lambda^i=\Lambda^1
$
for odd $i$, and
$
\Lambda^i=\Lambda^2
$
for even $i$.

We take
$
K=N_{q+1}
$
in the construction of Appendix~\ref{gij}, so that
\begin{align}
\eta N_{q+1}\ll1.
\label{etan}
\end{align}
We then introduce a family of pairwise disjoint functions
$
g_{(\xi,i,d_0)}
$
and
$
g_{(\xi,i,d_0')}
$
for
$
\xi\in\Lambda
$
and
$
1\leq i\leq N_{q+1}.
$

With these preparations, we define the rescaled building blocks by
\begin{align*}
W_{(\xi,n,i)}(x,t)
&:=
W_{(\xi,d_0)}
\Bigl(
x,
\Bigl(\frac{n}{\zeta^i}\Bigr)^{1/d_0}
H_{(\xi,i,d_0)}(t)
\Bigr),
\\
\Theta_{(\xi,n,i)}(x,t)
&:=
\Theta_{(\xi,d_0')}
\Bigl(
x,
\Bigl(\frac{n}{\zeta^i}\Bigr)^{1/d_0}
H_{(\xi,i,d_0)}(t)
\Bigr).
\end{align*}

Similarly, we define
$
V_{(\xi,n,i)},
\Phi_{(\xi,n,i)},
$
and all the other quantities appearing in Appendix~\ref{gij}. 
By \eqref{eq:ptthe+} and \eqref{eq:parth}, we have
\begin{align}
\partial_t\Theta_{(\xi,n,i)}
+
\Bigl(\frac{n}{\zeta^i}\Bigr)^{1/d_0}
g_{(\xi,i,d_0)}
\div\bigl(
W_{(\xi,n,i)}\Theta_{(\xi,n,i)}
\bigr)
=
0.
\label{eq:4ptthe+n}
\end{align}

We next define the perturbations for the drift term. 
For $1\leq i\leq N_{q+1}$, let
\begin{align*}
w_{q+1}^{(p,i)}
:=
\sum_{n\geq3}
\tilde{\chi}(\zeta^i|M_l^i|-n)
\Bigl(\frac{n}{\zeta^i}\Bigr)^{1/d_0}
\sum_{\xi\in\Lambda^n}
W_{(\xi,n,i)}g_{(\xi,i,d_0)},
\end{align*}
and
\begin{align*}
w_{q+1}^{(c,i)}
:=
\sum_{n\geq3}
\sum_{\xi\in\Lambda^n}
\Bigl(
&
-\tilde{\chi}(\zeta^i|M_l^i|-n)
\Bigl(\frac{n}{\zeta^i}\Bigr)^{1/d_0}
\frac1{(n_*\lambda_{q+1})^2}
\nabla\Phi_{(\xi,n,i)}
\,\xi\cdot\nabla\psi_{(\xi,n,i)}
\\
&\quad
+
\nabla\bigl(
\tilde{\chi}(\zeta^i|M_l^i|-n)
\bigr)
\Bigl(\frac{n}{\zeta^i}\Bigr)^{1/d_0}
:V_{(\xi,n,i)}
\Bigr)
g_{(\xi,i,d_0)}.
\end{align*}

By \eqref{divOmega}, we obtain
\begin{align}
w_{q+1}^{(p,i)}
+
w_{q+1}^{(c,i)}
=
\sum_{n\geq3}
\sum_{\xi\in\Lambda^n}
\div\Bigl(
\tilde{\chi}(\zeta^i|M_l^i|-n)
\Bigl(\frac{n}{\zeta^i}\Bigr)^{1/d_0}
V_{(\xi,n,i)}
\Bigr)
g_{(\xi,i,d_0)}.
\label{eq:4wq+1p+wq+1c}
\end{align}
Since $V_{(\xi,n,i)}$ is skew-symmetric, it follows that
$
\div\bigl(
w_{q+1}^{(p,i)}
+
w_{q+1}^{(c,i)}
\bigr)=0.
$

Finally, we define the total perturbation and the new velocity field by
\begin{align}
w_{q+1}
:=
\sum_{i=1}^{N_{q+1}}
\bigl(
w_{q+1}^{(p,i)}
+
w_{q+1}^{(c,i)}
\bigr),
\qquad
v_{q+1}
:=
v_l+w_{q+1}.
\notag
\end{align}
Then $v_{q+1}$ is mean-zero and divergence-free.
Moreover, since $M_l^i(t)=0$ on $[0,T_{q+1}],$ the perturbation $w_{q+1} $  vanishes on $[0,T_{q+1}]$.

\subsection{Construction of $\rho^{i}_{q+1}$}\label{sec:con:the}
We next define the perturbations for the density functions. 
For $1\leq i\leq N_{q+1}$, we set
\begin{align}
\theta_{q+1}^{(p,i)}
:=
&\sum_{n\geq3}
\chi(\zeta^i|M_l^i|-n)
\Bigl(\frac{n}{\zeta^i}\Bigr)^{1/d_0'}
\sum_{\xi\in\Lambda^n}
\Gamma_\xi\Bigl(\frac{M_l^i}{|M_l^i|}\Bigr)
\Theta_{(\xi,n,i)}
g_{(\xi,i,d_0')},
\notag\\
\theta_{q+1}^{(c,i)}
:=
&-
\int_{\mathbb{T}^d}
\theta_{q+1}^{(p,i)}\,\dif x,
\notag\\
\theta_{q+1}^{(o,i)}
:=
&-
\sigma^{-1}
\sum_{n\geq3}
\sum_{\xi\in\Lambda^n}
h_{(\xi,i,d_0)}
\div\Bigl(
\chi(\zeta^i|M_l^i|-n)
\frac{n}{\zeta^i}
\Gamma_\xi\Bigl(\frac{M_l^i}{|M_l^i|}\Bigr)
\xi
\Bigr).
\notag
\end{align}

By an argument analogous to that in \cite[(5.5)]{LRZ25}, we obtain
\begin{align}
w_{q+1}^{(p,i)}
\theta_{q+1}^{(p,i)}
=
&
\sum_{n\geq3}
\sum_{\xi\in\Lambda^n}
\chi(\zeta^i|M_l^i|-n)
\frac{n}{\zeta^i}
\Gamma_\xi\Bigl(\frac{M_l^i}{|M_l^i|}\Bigr)
\mathbb{P}_{\neq0}
\bigl(
W_{(\xi,n,i)}
\Theta_{(\xi,n,i)}
\bigr)
g_{(\xi,i,d_0)}
g_{(\xi,i,d_0')}
\notag\\
&+
\sum_{n\geq3}
\sum_{\xi\in\Lambda^n}
\chi(\zeta^i|M_l^i|-n)
\frac{n}{\zeta^i}
\Gamma_\xi\Bigl(\frac{M_l^i}{|M_l^i|}\Bigr)
\xi
\bigl(
g_{(\xi,i,d_0)}
g_{(\xi,i,d_0')}
-1
\bigr)
\notag\\
&+
\chi(\zeta^i|M_l^i|-n)
\frac{n}{\zeta^i}
\frac{M_l^i}{|M_l^i|}.
\notag
\end{align}

Combining this identity with \eqref{eq:4ptthe+n}, we deduce that
\begin{align}
&
\partial_t\theta_{q+1}^{(p,i)}
+
\div\bigl(
w_{q+1}^{(p,i)}
\theta_{q+1}^{(p,i)}
-
M_l^i
\bigr)
+
\partial_t\theta_{q+1}^{(o,i)}
\notag\\
=
&
\sum_{n\geq3}
\sum_{\xi\in\Lambda^n}
\partial_t
\Bigl[
\chi(\zeta^i|M_l^i|-n)
\Bigl(\frac{n}{\zeta^i}\Bigr)^{1/d_0'}
\Gamma_\xi\Bigl(\frac{M_l^i}{|M_l^i|}\Bigr)
g_{(\xi,i,d_0')}
\Bigr]
\Theta_{(\xi,n,i)}
\notag\\
&+
\sum_{n\geq3}
\sum_{\xi\in\Lambda^n}
\nabla
\Bigl[
\chi(\zeta^i|M_l^i|-n)
\frac{n}{\zeta^i}
\Gamma_\xi\Bigl(\frac{M_l^i}{|M_l^i|}\Bigr)
\Bigr]
g_{(\xi,i,d_0)}
g_{(\xi,i,d_0')}
\mathbb{P}_{\neq0}
\bigl(
W_{(\xi,n,i)}
\Theta_{(\xi,n,i)}
\bigr)
\notag\\
&+
\div
\Bigl(
\sum_{n\geq3}
\chi(\zeta^i|M_l^i|-n)
\frac{n}{\zeta^i}
\frac{M_l^i}{|M_l^i|}
-
M_l^i
\Bigr)
\notag\\
&-
\sigma^{-1}
\sum_{n\geq3}
\sum_{\xi\in\Lambda^n}
h_{(\xi,i,d_0)}
\partial_t
\div
\Bigl(
\chi(\zeta^i|M_l^i|-n)
\frac{n}{\zeta^i}
\Gamma_\xi\Bigl(\frac{M_l^i}{|M_l^i|}\Bigr)
\xi
\Bigr).
\label{wq+1good}
\end{align}

We now define, for every $1\leq i\leq N_{q+1}$,
\begin{align}
\theta_{q+1}^i
:=
\theta_{q+1}^{(p,i)}
+
\theta_{q+1}^{(c,i)}
+
\theta_{q+1}^{(o,i)},
\qquad
\rho_{q+1}^i
:=
\rho_l^i+\theta_{q+1}^i.
\notag
\end{align}
For $i>N_{q+1}$, we simply set
$
\rho_{q+1}^i:=\rho_q^i.
$

By construction,
$
\int_{\mathbb{T}^d}\rho_{q+1}^i\,\dif x=1
$
for every $i\in\mathbb{N}$.
Since
$
M_l^i(t)=0
$
on
$
[0,T_{q+1}],
$
the perturbation $\theta_{q+1}^i(t)$ vanishes there for all
$
1\leq i\leq N_{q+1}.
$
Hence
$
\rho_{q+1}^i(t)=1
$
on
$
[0,T_{q+1}]
$
for all
$
1\leq i\leq N_{q+1}.
$
For
$
i>N_{q+1},
$
we already have
$
\rho_{q+1}^i(t)
=
\rho_q^i(t)
=
1
$
on
$
[0,T_{q+1}].
$

Moreover, since the functions
$
g_{(\xi,i,d_0)}
$
have disjoint supports for distinct $i$, and since
$
\theta_{q+1}^{(c,i)}
$
depends only on $t$, we obtain for
$
1\leq i\leq N_{q+1}
$
that
\begin{align}
\div\bigl(
w_{q+1}\theta_{q+1}^i
\bigr)
=
\div\Bigl(
(w_{q+1}^{(p,i)}+w_{q+1}^{(c,i)})
\theta_{q+1}^{(p,i)}
+
w_{q+1}\theta_{q+1}^{(o,i)}
\Bigr).
\label{bd:4wq+1rhoq+1i}
\end{align}

\subsection{Construction of the stress terms $M^{i}_{q+1}$}\label{sec:4defmq+1}

\subsubsection{The case $1\leq i\leq N_{q+1}$} 
From \eqref{bd:4wq+1rhoq+1i}, and the definition of the perturbations 
we obtain

\begin{align*}
 -\div M^{i}_{q+1}&=\partial_t\theta_{q+1}^i+\div (w_{q+1}^{(p,i)}\theta_{q+1}^{(p,i)}-M^{i}_l)(:=\div M^{i}_{osc})\\
 &\quad-\div[ \div( A(\rho^{i}_{q+1})\rho^{i}_{q+1}-A(\rho^{i}_{l}) \rho^{i}_{l})-(V*\rho_{q+1}^i)\rho_{q+1}^i+(V*\rho_{l}^i)\rho_{l}^i](:=\div M^{i}_{nonlin})\notag\\ 
 &\quad+\div(v_l\theta^{i}_{q+1}+w_{q+1}(\rho^{i}_l+\theta_{q+1}^{(o,i)})+w_{q+1}^{(c,i)}\theta_{q+1}^{(p,i)})(:=\div M^{i}_{lin})\\
&\quad -\div M_{com}^i,
\end{align*}
where we define the nonlinear error and the linear error  respectively, by
\begin{align*} 
M^{i}_{nonlin}:&= -\div( A(\rho^{i}_{q+1}) \rho^{i}_{q+1}-A(\rho^{i}_{l}) \rho^{i}_{l} )-(V*\rho_{q+1}^i)\rho_{q+1}^i+(V*\rho_{l}^i)\rho_{l}^i,\\
M^{i}_{lin}:&=v_l\theta^{i}_{q+1}+w_{q+1}(\rho^{i}_l+\theta_{q+1}^{(o,i)}) +w_{q+1}^{(c,i)}\theta_{q+1}^{(p,i)}.
\end{align*}

 To define the oscillation error,   using the inverse divergence operators  $\mathcal{R}_1,\mathcal{B}_1$ introduced in Section \ref{tamr} and  \eqref{wq+1good}, we   define $M^{i}_{osc}:=M^{i}_{osc,t}+M^{i}_{osc,x}+M^{i}_{osc,c}+M^{i}_{osc,o}$ as
\begin{align*}
   M^{i}_{osc,t}&:=\sum_{n\geq3}\sum_{\xi\in\Lambda^{n}}\mathcal{R}_{1}\(\partial_t[\chi(\zeta^{i}|M^{i}_l|-n)\(\frac{n}{\zeta^{i}}\)^{1/d_0'}\Gamma_{\xi}\(\frac{M^{i}_l}{|M^{i}_l|}\)g_{(\xi,i,d_0')}]\Theta_{(\xi,n,i)}\),\\
M^{i}_{osc,x}&:=\sum_{n\geq3}\sum_{\xi\in\Lambda^{n}}\mathcal{B}_{1}\(\nabla[\chi(\zeta^{i}|M^{i}_l|-n)\frac{n}{\zeta^{i}}\Gamma_{\xi}\(\frac{M^{i}_l}{|M^{i}_l|}\)],\mathbb{P}_{\neq0}(W_{(\xi,n,i)}\Theta_{(\xi,n,i)})\)g_{(\xi,i,d_0)}g_{(\xi,i,d_0')},\\
   M^{i}_{osc,c}&:=\sum_{n\geq3}\chi(\zeta^{i}|M^{i}_l|-n)\frac{n}{\zeta^{i}}\frac{M^{i}_l}{|M^{i}_l|}- M^{i}_l,\\
    M^{i}_{osc,o}&:=-\sigma^{-1}\sum_{n\geq3}\sum_{\xi\in\Lambda^{n}}h_{(\xi,i,d_0)}\partial_t\(\chi(\zeta|M_l^i|-n)\frac{n}{\zeta^i}\Gamma_{\xi}\(\frac{M_l^i}{|M_l^i|}\)\xi\).
\end{align*}
 
Then we define 
$$-{M}^i_{q+1}:=M^{i}_{osc}+M^{i}_{nonlin}+M^{i}_{lin}-M_{com}^i.$$
Since $M^{i}_l(t)=w_{q+1}(t)=\theta^i_{q+1}(t)=0$ on $[0,T_{q+1}]$,  we have $M_{q+1}^i(t)=0$ on $[0,T_{q+1}]$, which implies  \eqref{bd:4mq=rhoq=0} for $M_{q+1}^i,1\leq i\leq N_{q+1}$. 

\subsubsection{The case $i>N_{q+1}$}
In this case, since
$
\div(v_{q+1}-v_q)=0
$
and
$
\rho_{q+1}^i=\rho_q^i=\rho_0^i,
$
we   define
\begin{align}
M_{q+1}^i
:=
M_q^i-(v_{q+1}-v_q)(\rho_0^i-1)
=
M_0^i-(v_{q+1}-v_0)(\rho_0^i-1).
\notag
\end{align}
It is straightforward to verify that
$
(v_{q+1},\rho_{q+1}^i,M_{q+1}^i)_{i>N_{q+1}}
$
satisfies \eqref{eq:4qth} at level $q+1$.

Moreover, since
$
\rho_0^i(t)=1
$
and
$
M_0^i(t)=0
$
on
$
[0,T_0],
$
it follows that
$
M_{q+1}^i(t)=0
$
on
$
[0,T_{q+1}],
$
which proves \eqref{bd:4mq=rhoq=0} for
$
M_{q+1}^i,
\ 
i>N_{q+1}.
$

\subsection{Proof of Proposition \ref{prop:case4}}  In this section, we verify that the perturbations and the new stress terms constructed above satisfy the properties stated in Proposition~\ref{prop:case4}. 
Most of the estimates are analogous to the basic convex integration estimates, so we mainly focus on the new features arising from the choice of $N_{q+1}$ and on the estimates for $M_{q+1}^i$ when $i>N_{q+1}$.

\subsubsection{Choice of parameters}\label{sec:4choicepara}
Regarding the parameters of the building blocks, we define
  
\begin{align}
  \lambda=\lambda_{q+1},\ r_\perp=\lambda_{q+1}^{-1+\frac1{N}},\ r_\parallel=\lambda_{q+1}^{-1+\frac2{N}},\ \eta=\lambda_{q+1}^{-1},\  \mu=r_\perp^{-\frac{d-1}{d_0}}r_\parallel^{-\frac1{d_0}}, \ \sigma=\lambda_{q+1}^{\frac1{2N}}, \notag
\end{align}
where $N>4d$ is a sufficiently large integer satisfying $\frac{d}{p}+\frac{1}{r}>1+\frac{4d}N,\frac{d}{s}>1+\frac{4d}N$. We further define $d_0:=d+1-\frac{4d}{N}\in(d,d+1), d_0':=\frac{d_0}{d_0-1}\in(1,2)$. It is straightforward to verify \eqref{etan}.
With these choices,  we have
\begin{align}
  r_\perp^{d-1-\frac{d-1}{d_0'}} r_\parallel^{1-\frac{1}{d_0'}}\eta^{-\frac{1}{d_0'}},\    r_\perp^{\frac{d-1}{p}-\frac{d-1}{d_0}} r_\parallel^{\frac{1}{p}-\frac{1}{d_0}}\eta^{\frac1r-\frac1{d_0}}, \  r_\perp^{\frac{d-1}{s}-\frac{d-1}{d_0}} r_\parallel^{\frac{1}{s}-\frac{1}{d_0}}\eta^{-\frac1{d_0}} \leq\lambda^{-\frac1N}.\label{bd:paralam4}
\end{align}
 In fact, by a direct calculation, it holds that
    \begin{align*}
        \lambda r_\perp^{d-1-\frac{d-1}{d_0'}}r_\parallel^{1-\frac{1}{d_0'}}\eta^{1-\frac1{d_0'}}&=r_\perp^{d-1-\frac{d-1}{d_0'}} r_\parallel^{1-\frac{1}{d_0'}}\eta^{-\frac{1}{d_0'}}\leq r_\parallel^{\frac{d}{d_0}}\eta^{\frac{1}{d_0}-1}\\
        &= \lambda^{\frac{2d}{Nd_0}+\frac{d_0-d-1}{d_0}}= \lambda^{-\frac{2d}{Nd_0}}\leq \lambda^{-\frac1N},\\
 r_\perp^{\frac{d-1}{p}-\frac{d-1}{d_0}} r_\parallel^{\frac{1}{p}-\frac{1}{d_0}}\eta^{\frac1r-\frac1{d_0}}&\leq  \lambda^{(\frac{d+1}{N}-d)(\frac{1}{p}-\frac{1}{d_0})-\frac1r+\frac1{d_0}}
 \leq\lambda^{-\frac{d}{p}-\frac1r+\frac{d+1}{d_0}+\frac{d+1}{N}}\\
 &\leq \lambda^{-1-\frac{4d}N+1+\frac{2d}{N}+\frac{d+1}{N}}\leq \lambda^{-\frac1N},\\
  r_\perp^{\frac{d-1}{s}-\frac{d-1}{d_0}} r_\parallel^{\frac{1}{s}-\frac{1}{d_0}}\eta^{-\frac1{d_0}}&\leq  \lambda^{(\frac{d+1}{N}-d)(\frac{1}{s}-\frac{1}{d_0})+\frac1{d_0}}
 \leq\lambda^{-\frac{d}{s}+\frac{d+1}{d_0}+\frac{d+1}{N}}\\
 &\leq \lambda^{-1-\frac{4d}N+1+\frac{2d}{N}+\frac{d+1}{N}}\leq \lambda^{-\frac1N}.
    \end{align*}

In the sequel, we shall also require \eqref{alphab>} together with \begin{align}
   \lambda_{q}^{d+4}\leq\lambda_{q+1}^\alpha,\ \ (12d+44)\alpha<\frac1{2N}.
 \end{align} 
 
These conditions can be achieved as follows:
first choose $\alpha>0$ sufficiently small such that
$
(12d+44)\alpha<\frac1{2N},
$
then choose
$
b\in2N\mathbb{N}
$
sufficiently large so that
$
b>\frac{6d+24}{\alpha},
$
and finally choose $\beta>0$ sufficiently small such that
$
\alpha>4\beta b.
$
At the end, we choose $a$ sufficiently large in order to absorb all implicit and universal constants arising in the subsequent estimates, and to guarantee the validity of \eqref{ieq:4ab2}.

Finally, we record the following estimate for the amplitude functions.

\bl\label{lem:4chi}$($\cite[Proposition 5.2]{LRZ25}$)$.
 For $ M,N\in\mathbb{N}_0,1\leq i\leq N_{q+1}$  we have
\begin{align*}
    \sum_{n\geq3} \(\frac{n}{\zeta^i}\)^M\(\|\chi(\zeta^i|M^{i}_l|-n)\|_{C_{t,x}^N}+ \|\tilde{\chi}(\zeta^i|M^{i}_l|-n)\|_{C_{t,x}^N}\)&\lesssim l^{- (d+4)N-(d+2)(M+1)},\\
     \sum_{n\geq3}\sum_{\xi\in\Lambda^{n}}\(\frac{n}{\zeta^i}\)^M\norm{\chi(\zeta^i|M^{i}_l|-n)\Gamma_\xi\(\frac{M^{i}_l}{|M^{i}_l|}\)}_{C_{t,x}^N}& \lesssim l^{-(2d+8)N-(d+2)(M+1)}.
\end{align*}
\el 

\subsubsection{Proof of  \eqref{bd:4vq+1-vqldd}  for $v_{q+1}-v_q$.}
From this section onward, we establish the desired estimates for the perturbation $w_{q+1}$.

We first estimate the principal perturbations $w_{q+1}^{(p,i)}$ for   $1\leq i\leq N_{q+1}$ in the $L_t^{d_0}L^{d_0}$-norm. By Cauchy's inequality and the fact that $\sum_{n\in\mathbb{Z}} \tilde{\chi}(t - n) \leq2$  we have
\begin{align}
|w_{q+1}^{(p,i)}|^{d_0}&\leq\(\sum_{n\geq3}\tilde{\chi}(\zeta^{i}|M^{i}_l|-n)\)^{d_0-1}\sum_{n\geq3}\tilde{\chi}(\zeta^{i}|M^{i}_l|-n)\frac{n}{\zeta^{i}}\Bigg{|}\sum_{\xi\in\Lambda^{n}}W_{(\xi,n,i)}g_{(\xi,i,d_0)}\bigg{|}^{d_0}\notag\\ 
&\lesssim \sum_{n\geq3}\tilde{\chi}(\zeta^{i}|M^{i}_l|-n)\frac{n}{\zeta^{i}}\sum_{\xi\in\Lambda^{n}}\left|W_{(\xi,n,i)}g_{(\xi,i,d_0)}\right|^{d_0}.\notag
\end{align} 

By applying the generalized H\"older inequality of 
Theorem \ref{ihiot} in spatial direction, together with the estimates for the building blocks in \eqref{int4} and Lemma \ref{lem:4chi} we deduce 
\begin{align}
\|w_{q+1}^{(p,i)}(t)\|^{d_0}_{L^{d_0}}
&\lesssim\sum_{n\geq3}\norm{\tilde{\chi}(\zeta^{i}|M^{i}_l(t)|-n)\frac{n}{\zeta^{i}}}_{L^1}\sum_{\xi\in\Lambda^{n}}\|W_{(\xi,n,i)}\|_{C_tL^{d_0}}^{d_0}g_{(\xi,i,d_0)}^{d_0}(t)\notag\\ 
&\quad+(r_\perp\lambda_{q+1})^{-1}\norm{\tilde{\chi}(\zeta^{i}|M^{i}_l(t)|-n)\frac{n}{\zeta^{i}}}_{C_{t,x}^1}\sum_{\xi\in\Lambda^{n}}\|W_{(\xi,n,i)}\|_{C_tL^{d_0}}^{d_0}g_{(\xi,i,d_0)}^{d_0}(t)\notag\\
&\lesssim\left(\norm{\sum_{n\geq3}\tilde{\chi}(\zeta^{i}|M^{i}_l(t)|-n)(M^{i}_l(t)+|\zeta^{i}|^{-1}
)}_{L^1}+l^{-3d-8}\lambda_{q+1}^{-\frac{1}{N}}\right)\sum_{\xi\in\Lambda}g_{(\xi,i,d_0)}^{d_0}(t)\notag\\
&\lesssim (\|M^{i}_l(t)\|_{L^1}+2^{-i}\delta_{q+1})\sum_{\xi\in\Lambda}g_{(\xi,i,d_0)}^{d_0}(t)
,\notag
\end{align}
where we used   the fact that $\sum_{n\in\mathbb{Z}} \tilde{\chi}(t - n) \leq2$, and used conditions on the parameters to have $(6d+17)\alpha-\frac{1}{N}<-\alpha<-2\beta,2^{i}\leq 2^{N_{q+1}}\leq \lambda_{q+1}^{\alpha}$. Here we recall the notation $\Lambda=\Lambda^{1}\cup\Lambda^{2}$.
Then we apply the generalized H\"older inequality of Theorem \ref{ihiot} in time direction, the bounds \eqref{bd:gwnp} and \eqref{bd:4mlcn} to deduce  for some ${C}_v\geq1$
\begin{align}
\|w_{q+1}^{(p,i)}\|_{L^{d_0}_tL^{d_0}}^{d_0}&\lesssim (\|M^{i}_l\|_{L^1_tL^1}+2^{-i}\delta_{q+1}+\sigma^{-1}\|M^{i}_l\|_{C_{t,x}^1})\sum_{\xi\in\Lambda}\norm{g_{(\xi,i,d_0)}}^{d_0}_{L_t^{d_0}}
\notag\\
    &\lesssim C_02^{-i}(\delta_{q+1}+\lambda_{q+1}^{(2d+6)\alpha-\frac{1}{2N}})\leq (\frac14 C_v(2^{\frac1{d+1}}-1))^{d_0}C_02^{-i}\delta_{q+1},\label{bd:4wq+1pl2}
\end{align}
where we used  conditions on the parameters to have $(2d+6)\alpha-\frac{1}{2N}<-\alpha<-2\beta$.

For the general $L^u_tL^m$-norm with $u,m\in[1,\infty]$, by the estimates for the building blocks in \eqref{int2}-\eqref{int4}, \eqref{bd:gwnp}
and the estimate for the amplitude function in  Lemma \ref{lem:4chi} we obtain
\begin{align}
\|w_{q+1}^{(p,i)}\|_{L^u_tL^m}&\lesssim\sum_{n\geq3}\sum_{\xi\in\Lambda^{n}}\norm{\tilde{\chi}(\zeta^{i}|M^{i}_l|-n)\(\frac{n}{\zeta^{i}}\)^{1/d_0}}_{C_{t,x}^0}\|W_{(\xi,n,i)}\|_{C_tL^m}\|g_{(\xi,i,d_0)}\|_{L_t^u}\notag\\
&\lesssim l^{-2d-4}r_\perp^{\frac{d-1}{m}-\frac{d-1}{d_0}} r_\parallel^{\frac{1}{m}-\frac{1}{d_0}}\eta^{\frac1u-\frac1{d_0}},\label{bd:4wq+1plp}\\
\|w_{q+1}^{(c,i)}\|_{L^u_tL^m}&\lesssim\sum_{n\geq3}\sum_{\xi\in\Lambda^{n}}\norm{\tilde{\chi}(\zeta^{i}|M^{i}_l|-n)\(\frac{n}{\zeta^{i}}\)^{1/d_0}}_{C^1_{t,x}}\notag\\ 
&\quad\quad \times\(\frac{1}{\lambda_{q+1}^2}\|\nabla\Phi_{(\xi,n,i)} \xi\cdot\nabla\psi_{(\xi,n,i)} \|_{L^m}+\|V_{(\xi,n,i)}\|_{L^m}\)\|g_{(\xi,i,d_0)}\|_{L_t^u}\notag\\
&\lesssim l^{-3d-8}r_\perp^{\frac{d-1}{m}-\frac{d-1}{d_0}} r_\parallel^{\frac{1}{m}-\frac{1}{d_0}}\frac{r_\perp}{ r_\parallel}\eta^{\frac1u-\frac1{d_0}}.\label{bd:4wq+1clp}
\end{align}

With these estimates, combining with the choice of parameters in  \eqref{para42} and the bound \eqref{bd:4wq+1pl2} we obtain 
\begin{align}
\|w_{q+1}\|_{L^{d_0}_tL^{d_0}}&\leq \sum_{i=1}^{N_{q+1}}\frac{C_v}42^{-i/d_0}C_0^{1/d_0}(2^{\frac1{d+1}}-1)\delta_{q+1}^{1/d_0}+N_{q+1}Cl^{-3d-8}\frac{r_\perp}{ r_\parallel}\notag \\
&\leq \frac{C_v}4C_0^{1/d_0}\delta_{q+1}^{1/d_0}+C\lambda_{q+1}^{(6d+17)\alpha-\frac1N} \leq \frac{C_v}2C_0^{1/d_0}\delta_{q+1}^{1/d_0},\label{bd:4wq+1l2}\end{align}
where we used  conditions on the parameters to have
${(6d+18)\alpha<\frac{1}{N}},N_{q+1}\lesssim\lambda_{q+1}^{\alpha}$ and chose $a$ large enough to absorb the universal constant. 
The above inequality   yields that \eqref{bd:4vq+1-vqldd} holds for $v_{q+1}-v_q$ and then \eqref{bd:4vql2} holds for $v_{q+1}$:
\begin{align*}
\|v_{q+1}-v_{q}\|_{L^{d_0}_tL^{d_0}}&\leq\|w_{q+1} \|_{L^{d_0}_tL^{d_0}}+\|v_{l}-v_{q}\|_{L^{d_0}_tL^{d_0}}\\
&\leq \frac12
C_vC_0^{1/d_0}\delta_{q+1}^{1/d_0}+lC_0^{1/d_0}\lambda_q^{d+4} \leq
C_vC_0^{1/d_0}\delta_{q+1}^{1/d_0}.
\end{align*}

\subsubsection{Proof of  \eqref{bd:4vq+1-vqlpr}}

Combining  the bounds \eqref{bd:4wq+1plp} and \eqref{bd:4wq+1clp} above we obtain 
\begin{align}
\|&v_{q+1}-v_q\|_{L^{r}_tL^{p}}+\|v_{q+1}-v_q\|_{C_tL^{s}}\lesssim  \|w_{q+1}\|_{C_tL^{s}}+  \|w_{q+1}\|_{L^{r}_tL^{p}}+l\|v_q\|_{C_{t,x}^1}\notag \\
&\lesssim  N_{q+1}l^{-3d-8}( r_\perp^{\frac{d-1}s-\frac{d-1}{d_0}} r_\parallel^{\frac1s-\frac{1}{d_0}}\eta^{-\frac1{d_0}}+r_\perp^{\frac{d-1}{p}-\frac{d-1}{d_0}} r_\parallel^{\frac{1}{p}-\frac{1}{d_0}}\eta^{\frac1r-\frac1{d_0}})+l \lambda_q^{d+4}\notag\\ 
&\lesssim \lambda_{q+1}^{(6d+17)\alpha-\frac1N}+\lambda_{q+1}^{-\alpha}
\leq\delta_{q+1}^{1/d_0}.\label{bd:4wq+1lpr}
\end{align}
  Here we used \eqref{para42}, \eqref{bd:paralam4} and  conditions on the parameters to have 
${(6d+18)\alpha<\frac{1}{N}},N_{q+1}\lesssim\lambda_{q+1}^{\alpha}$. Then we chose $a$ large enough to absorb the universal constant.

\subsubsection{Proof of  \eqref{bd:4rhoqc1} for $v_{q+1}$.}
 By the estimates for the building blocks in \eqref{int4}, \eqref{bd:gwnp},  \eqref{eq:4wq+1p+wq+1c} and the estimates for the amplitude functions in Lemma \ref{lem:4chi} we have for $d_0\geq2$
\begin{align}
\|w_{q+1}\|_{C_{t,x}^1}&\lesssim\sum_{i=1}^{N_{q+1}}\sum_{n\geq3}\sum_{\xi\in\Lambda^{n}}\norm{\tilde{\chi}(\zeta^{i}|M^{i}_l|-n)\(\frac{n}{\zeta^{i}}\)^{1/d_0}}_{C_{t,x}^{2}}\|\nabla V_{(\xi,n,i)}\|_{C_{t,x}^1}\|g_{(\xi,i,d_0)}\|_{C_t^1}\notag\\
&\lesssim N_{q+1}l^{-4d-12}\lambda_{q+1}\mu r_\parallel^{-\frac{1}{d_0}}r_\perp^{-\frac{d-1}{d_0}}\sigma\eta^{-1-\frac{2}{d_0}}\lesssim N_{q+1}\lambda_{q+1}^{(8d+24)\alpha+d+\frac72}.\notag
\end{align}
 Thus, by $N_{q+1}\lesssim\lambda_{q+1}^{\alpha},(8d+25)\alpha<\frac{1}{2}$ we obtain the following
 \begin{align*}
\|v_{q+1}\|_{C_{t,x}^1}\leq\|v_{l}\|_{C_{t,x}^1}+\|w_{q+1}\|_{C_{t,x}^1}\leq C_0^{1/d_0}\lambda_{q}^{d+4}+\frac12\lambda_{q+1}^{d+4}\leq C_0^{1/d_0}\lambda_{q+1}^{d+4},
 \end{align*} where we chose $a$ large enough to absorb the universal constant.

\subsubsection{Proof of  \eqref{bd:4vq+1-vqldd} for $\rho^{i}_{q+1}-\rho_q^i$}\label{sec:4esttheq+1}
Similarly as before, we first estimate the principal perturbations  $\theta^{(p,i)}_{q+1}$ in $L^{d_0'}_tL^{d_0'}$.  By H\"older inequality and the fact that  $\sum_{n\in\mathbb{Z}} {\chi}(t - n) =1$, and $\Gamma_{\xi}$ are uniformly bounded we have
\begin{align}
|\theta_{q+1}^{(p,i)}|^{d_0'}
&\lesssim\(\sum_{n\geq3}\chi(\zeta^{i}|M^{i}_l|-n)\)^{d_0'-1}\sum_{n\geq3}\chi(\zeta^{i}|M^{i}_l|-n)\frac{n}{\zeta^{i}}\sum_{\xi\in\Lambda^{n}}\left|\Gamma_{\xi}\(\frac{M^{i}_l}{|M^{i}_l|}\)\Theta_{(\xi,n,i)}g_{(\xi,i,d_0')}\right|^{d_0'}\notag\\
&\lesssim \sum_{n\geq3}\chi(\zeta^{i}|M^{i}_l|-n)\frac{n}{\zeta^{i}}\sum_{\xi\in\Lambda^{n}}\left|\Theta_{(\xi,n,i)}g_{(\xi,i,d_0')}\right|^{d_0'}.\notag
\end{align} 
By the same argument as in \eqref{bd:4wq+1pl2}, we have 
\begin{align}
    \|\theta_{q+1}^{(p,i)}\|_{L^{d_0'}_tL^{d_0'}}^{d_0'}&\lesssim \(\|M^{i}_l\|_{L^1_tL^1}+2^{-i}\delta_{q+1}+\sigma^{-1}\|M^{i}_l\|_{C_{t,x}^1}\)
    \sum_{\xi\in\Lambda }\norm{g_{(\xi,i,d_0')}}^{d_0'}_{L_t^{d_0'}}
    \notag\\
    &\lesssim 2^{-i}C_0(\delta_{q+1}+\lambda_{q+1}^{(2d+6)\alpha-\frac{1}{2N}})\leq 2^{-i}(\frac12C_v)^{d_0'}C_0\delta_{q+1},\label{bd:4theq+1pl2}
\end{align}
where we used  conditions on the parameters to have $(2d+6)\alpha-\frac{1}{2N}<-\alpha<-2\beta$.

For general $L_t^uL^m$-norm with $u,m\in[1,\infty]$, by the estimates for the building blocks in \eqref{int4theta}, \eqref{bd:gwnp} and Lemma \ref{lem:4chi} we obtain
\begin{align}
\|\theta_{q+1}^{(p,i)} \|_{L^u_tL^{m}}&\lesssim\sum_{n\geq3}\sum_{\xi\in\Lambda^{n}}\norm{\chi(\zeta^{i}|M^{i}_l|-n)\(\frac{n}{\zeta^{i}}\)^{1/d_0'}\Gamma_{\xi}\(\frac{M^{i}_l}{|M^{i}_l|}\)}_{C_{t,x}^0}\|\Theta_{(\xi,n,i)} \|_{C_tL^{m}}\|g_{(\xi,i,d_0')}\|_{L_t^{u}}\notag\\ 
&\lesssim l^{-2d-4}r_\perp^{\frac{d-1}{m}-\frac{d-1}{d_0'}} r_\parallel^{\frac{1}{m}-\frac{1}{d_0'}}\eta^{\frac1u-\frac{1}{d_0'}}.\label{bd:4theq+1plp}
\end{align}
Moreover, by the bounds \eqref{bd:gwnp} and \eqref{bd:paralam4} we have for some $\epsilon>0$ small enough
\begin{align}
\|\theta_{q+1}^{(c,i)} \|_{C_t}&\lesssim\|\theta_{q+1}^{(p,i)} \|_{C_tL^{1+\epsilon}}
\lesssim l^{-2d-4}r_\perp^{\frac{d-1}{1+\epsilon}-\frac{d-1}{d_0'}} r_\parallel^{\frac{1}{1+\epsilon}-\frac{1}{d_0'}}\eta^{-\frac{1}{d_0'}}\lesssim  \lambda_{q+1}^{(4d+8)\alpha-\frac1N+d\epsilon}\lesssim4^{-i} \lambda_{q+1}^{-2\alpha},\label{bd:4theq+1clp}\\
\|\theta_{q+1}^{(o,i)}\|_{C_tC^1}
&\lesssim \sigma^{-1} l^{-6d-20}\lesssim \lambda_{q+1}^{(12d+40)\alpha-\frac1{2N}}\lesssim4^{-i} \lambda_{q+1}^{-2\alpha},\label{bd:4theq+1olp}
\end{align}
where we used  conditions on the parameters to have $(12d+44)\alpha<\frac1{2N},2^i\leq2^{N_{q+1}}\lesssim\lambda_{q+1}^{\alpha}$ and chose $a$ large enough to absorb the universal constant. We also chose $\epsilon>0$ small such that $d\epsilon<\alpha$.
Then, combining the above estimates together, we obtain by   \eqref{bd:4rhoqc1} 
\begin{align}
\|\rho^{i}_{q+1}-\rho^{i}_q\|_{L^{d_0'}_tL^{d_0'}}
&\leq\|\theta^{i}_{q+1}\|_{L^{d_0'}_tL^{d_0'}}+l\|\rho_q^i-1\|_{C_{t,x}^1}\notag\\
&\leq\frac12 C_v C_0^{1/{d_0'}}2^{-i/d_0'}\delta_{q+1}^{1/{d_0'}}+4^{-i} \lambda_{q+1}^{-2\alpha}+l2^{-i}C_0^{1/{d_0'}}  \lambda_q^{d+4}  \leq C_v C_0^{1/{d_0'}}2^{-i/d_0'}\delta_{q+1}^{1/{d_0'}},\notag
\end{align}
which implies \eqref{bd:4vq+1-vqldd} for $\rho^{i}_{q+1}-\rho_q^i,1\leq i\leq N_{q+1}$. Here we chose $a$ large enough to absorb the universal constant. For $i>N_{q+1}$, we have
$
\rho_{q+1}^i-\rho_q^i=0.
$
Hence, the second estimate in \eqref{bd:4vq+1-vqldd} also holds for $i>N_{q+1}$.

\subsubsection{Proof of   \eqref{bd:4rhoq+1-rhoql1}}
We first estimate $\theta_{q+1}^{(p,i)}$ in $W^{1,1+\epsilon}$-norm for some $\epsilon>0$ small enough. By the bounds for the building blocks in \eqref{int4theta}, \eqref{bd:gwnp}  and the bounds for the amplitude functions in Lemma \ref{lem:4chi} we have
\begin{align*}
\|\theta_{q+1}^{(p,i)}\|_{L^1_tW^{1,1+\epsilon}}&\lesssim\sum_{n\geq3}\sum_{\xi\in\Lambda^{n}}\norm{\chi(\zeta^{i}|M^{i}_l|-n)\(\frac{n}{\zeta^{i}}\)^{1/d_0'}\Gamma_{\xi}\(\frac{M^{i}_l}{|M^{i}_l|}\)}_{C_{t,x}^1}\|\Theta_{(\xi,n,i)} \|_{C_tW^{1,1+\epsilon}}\|g_{(\xi,i,d_0')}\|_{L_t^1}\notag\\ 
&\lesssim l^{-4d-12}\lambda_{q+1}r_\parallel^{\frac1{1+\epsilon}-\frac{1}{d_0'}}r_\perp^{\frac{d-1}{1+\epsilon}-\frac{d-1}{d_0'}}\eta^{1-\frac1{d_0'}}
\lesssim\lambda_{q+1}^{(8d+24)\alpha-\frac1N+d\epsilon}\lesssim2^{-i}\lambda_{q+1}^{-2\alpha} ,
\end{align*}
where we used the choice of parameters in \eqref{bd:paralam4} and  chose $\epsilon>0$ small enough such that $d\epsilon<\alpha$. We also used  conditions on the parameters to have $(8d+28)\alpha<\frac1N,2^i\leq2^{N_{q+1}}\lesssim\lambda_{q+1}^{\alpha}$.

Similarly, by the fact that $\theta_{q+1}^{(p,i)}$ is non-negative, \eqref{bd:4rhoqc1} and the choice of parameters in \eqref{para42}, \eqref{bd:paralam4}  we have
\begin{align}
\|\rho^{i}_{q+1}-\rho^{i}_q\|_{C_tL^{1+\epsilon}}&+\|\rho^{i}_{q+1}-\rho^{i}_q\|_{L_t^1W^{1,1+\epsilon}}
\leq\|\theta^{i}_{q+1}\|_{C_tL^{1+\epsilon}}+\|\theta^{i}_{q+1}\|_{L_t^1W^{1,1+\epsilon}}+ l\|\rho_q^i-1\|_{C_{t,x}^2}\notag\\
&\qquad\qquad\qquad\qquad\qquad\lesssim4^{-i}\lambda_{q+1}^{-\alpha}\leq 4^{-i}\delta_{q+1}^{1/d_0'},\notag\\
    \inf_{t\in [0,1]}(\rho^{i}_{q+1} - \rho^{i}_q)&\geq-\|\theta_{q+1}^{(c,i)}+\theta
_{q+1}^{(o,i)}\|_{C_{t,x}^0} -l\|\rho_q^i-1\|_{C_{t,x}^1}\geq -C\lambda_{q+1}^{-\alpha} \geq -\delta_{q+1}^{1/d_0'},\notag
\end{align}
which yields \eqref{bd:4rhoq+1-rhoql1} for $1\leq i\leq N_{q+1}$.  Here we chose $a$ large enough to absorb the universal constant. Since
$
\rho_{q+1}^i-\rho_q^i=0
$
for
$
i>N_{q+1},
$
the estimate \eqref{bd:4rhoq+1-rhoql1} also holds in this case.

\subsubsection{Proof of  \eqref{bd:4rhoqc1} for $\rho^{i}_{q+1}$}
  By \eqref{int4theta}, \eqref{bd:gwnp} and Lemma \ref{lem:4chi} we have
\begin{align*}
\|\theta_{q+1}^{(p,i)} \|_{C_{t,x}^1} 
&\lesssim\sum_{n\geq3}\sum_{\xi\in\Lambda^{n}}\norm{\chi(\zeta^{i}|M^{i}_l|-n)\(\frac{n}{\zeta^{i}}\)^{1/d_0'}\Gamma_{\xi}\(\frac{M^{i}_l}{|M^{i}_l|}\)}_{C_{t,x}^1}\|\Theta_{(\xi,n,i)} \|_{C_{t,x}^1}\|g_{(\xi,i,d_0')}\|_{C_t^1}\notag\\ 
&\lesssim l^{-4d-12}\lambda_{q+1}\mu r_\parallel^{-\frac{1}{d_0'}}r_\perp^{-\frac{d-1}{d_0'}}\sigma\eta^{-1-\frac{1}{d_0}-\frac{1}{d_0'}}\lesssim
\lambda_{q+1}^{(8d+24)\alpha+d+\frac{7}{2}},\notag\\
\|\theta_{q+1}^{(c,i)}\|_{C_{t}^1}&\lesssim\|\theta_{q+1}^{(p,i)} \|_{C_{t,x}^1}\lesssim 
\lambda_{q+1}^{(8d+24)\alpha+d+\frac{7}{2}},\notag\\
     \|   \theta_{q+1}^{(o,i)}\|_{C_{t,x}^1}&\lesssim \sigma^{-1}\sum_{n\geq3}\sum_{\xi\in\Lambda^{n}}\|h_{(\xi,i,d_0)}\|_{C_t^1}\norm{\div\(\chi(\zeta^i|M_l|-n)\frac{n}{\zeta^i}\Gamma_{\xi}\(\frac{M_l^i}{|M_l^i|}\)\xi\)}_{C_{t,x}^1}\\ 
     &\lesssim \sigma^{-1}\eta^{-1} l^{-6d-20}\lesssim \lambda_{q+1}^{(12d+40)\alpha+2}.\notag
\end{align*}
Moreover, 
\begin{align*}
    \|\theta_{q+1}^{(p,i)} \|_{C_{t,x}^2}
&\lesssim\sum_{n\geq3}\sum_{\xi\in\Lambda^{n}}\norm{\chi(\zeta^{i}|M^{i}_l|-n)\(\frac{n}{\zeta^{i}}\)^{1/d_0'}\Gamma_{\xi}\(\frac{M^{i}_l}{|M^{i}_l|}\)}_{C_{t,x}^2}\| \Theta_{(\xi,n,i)} g_{(\xi,i,d_0')}\|_{C_{t,x}^2}\notag\\ 
&\lesssim l^{-6d-20}\lambda_{q+1}\mu r_\parallel^{-\frac{1}{d_0'}}r_\perp^{-\frac{d-1}{d_0'}}\sigma\eta^{-2}(\lambda_{q+1}\mu+\sigma\eta^{-1})\lesssim
\lambda_{q+1}^{(12d+40)\alpha+d+\frac{11}{2}},\notag\\
   \|    \theta_{q+1}^{(o,i)}\|_{C_{t,x}^2}&\lesssim \sigma^{-1}\sum_{n\geq3}\sum_{\xi\in\Lambda^{n}}\|h_{(\xi,i,d_0)}\|_{C_t^2}\norm{\chi(\zeta^i|M_l|-n)\frac{n}{\zeta^i}\Gamma_{\xi}\(\frac{M_l^i}{|M_l^i|}\)\xi}_{C_{t,x}^3}
       \lesssim \lambda_{q+1}^{(16d+56)\alpha+4}.\notag
\end{align*}

By choosing $(12d+42)\alpha<\frac{1}{2}$, $2^i\leq2^{N_{q+1}}\lesssim\lambda_{q+1}^{\alpha}$ we deduce 
\begin{align}
    4^i\|\rho^{i}_{q+1} -1\|_{C_{t,x}^1}+\lambda_{q+1}^{-2}\|\rho^{i}_{q+1}\|_{C_{t,x}^2}
    \leq C_0^{1/d_0'}\lambda_{q}^{d+4}+ \frac12\lambda_{q+1}^{d+4}\leq C_0^{1/d_0'}\lambda_{q+1}^{d+4},\notag
\end{align}
which implies  \eqref{bd:4rhoqc1} for $\rho^{i}_{q+1},1\leq i\leq N_{q+1}$. For $i>N_{q+1}$, the estimate \eqref{bd:4rhoqc1} also holds, since
$
\rho_{q+1}^i-\rho_q^i=0
$
and
$
\lambda_q\leq\lambda_{q+1}.
$

\subsubsection{Proof of  \eqref{bd:4rql1} for $1\leq i\leq N_{q+1}$}\label{sec:4estmq+1} We now estimate each term in the definition of $M_{q+1}^i$ separately.

\textbf{Oscillation error $M^{i}_{osc}$.} By Lemma \ref{bb_1}, the estimates for the amplitude functions and for the building blocks in  Lemma \ref{lem:4chi}, \eqref{int4theta} and \eqref{bd:gwnp} respectively   we obtain
\begin{align}
    \|M^{i}_{osc,t}\|_{L^1_tL^1}
    &\lesssim \sum_{n\geq3}\sum_{\xi\in\Lambda^{n}}\norm{\chi(\zeta^{i}|M^{i}_l|-n)\(\frac{n}{\zeta^{i}}\)^{1/d_0'}\Gamma_{\xi}\(\frac{M^{i}_l}{|M^{i}_l|}\)}_{C_{t,x}^1}\|g_{(\xi,i,d_0')}\|_{W_t^{1,1}} \|\Theta_{(\xi,n,i)}\|_{C_tL^1}\notag\\
    &\lesssim l^{-4d-12}r_\perp^{d-1-\frac{d-1}{d_0'}}r_\parallel^{1-\frac{1}{d_0'}}\sigma \eta^{-\frac{1}{d_0'}}\lesssim \lambda_{q+1}^{(8d+24)\alpha-\frac 1{2N}}\lesssim2^{-i} \lambda_{q+1}^{-\alpha},\notag
\end{align}
where we used the choice of parameters in 
\eqref{para42},  \eqref{bd:paralam4} and  conditions for the parameters to have $(8d+26)\alpha<\frac 1{2N},2^i\leq2^{N_{q+1}}\lesssim \lambda_{q+1}^\alpha$.

We observe that $W_{(\xi,n,i)}\Theta_{(\xi,n,i)}$ is $(\mathbb{T}/r_\perp\lambda_{q+1})^d$-periodic. So by  Theorem \ref{bb1_1},  the estimates for the amplitude functions and for the building blocks  in Lemma \ref{lem:4chi}, \eqref{int4},  \eqref{bd:gwnp} and \eqref{int4theta} respectively    we have
\begin{align}
\|M^{i}_{osc,x} \|_{L^1_tL^1}
&\lesssim\sum_{n\geq3}\sum_{\xi\in\Lambda^{n}}\norm{\chi(\zeta^{i}|M^{i}_l|-n)\frac{n}{\zeta^{i}}\Gamma_{\xi}\(\frac{M^{i}_l}{|M^{i}_l|}\)}_{C_{t,x}^2}\notag\\ 
&\quad\quad\times (r_\perp\lambda_{q+1})^{-1}\|W_{(\xi,n,i)}\Theta_{(\xi,n,i)}\|_{C_tL^{1}}\|g_{(\xi,i,d_0)}g_{(\xi,i,d_0')}\|_{L_t^1}\notag\\
&\lesssim l^{-6d-20}(r_\perp\lambda_{q+1})^{-1}\lesssim \lambda_{q+1}^{(12d+40)\alpha-\frac{1}{N}}
\lesssim2^{-i} \lambda_{q+1}^{-\alpha}, \notag
\end{align}
where we used the choice of parameters in \eqref{para42}, \eqref{bd:paralam4} and  conditions on the parameters to have $(12d+42)\alpha<\frac{1}{N},2^i\leq2^{N_{q+1}}\lesssim \lambda_{q+1}^\alpha$.

For the stress term $M^{i}_{osc,c}$, it holds
\begin{align}
     \left| M_{osc,c}^i\right|&\leq \left|  \sum_{n=-1}^{2}\chi(\zeta^i|M_l^i|-n)M_l^i\right|+\left|\sum_{n\geq3}\chi(\zeta^i|M_l^i|-n)(\frac{n}{\zeta^i}\frac{M_l^i}{|M_l^i|}- M_l^i)\right|\notag\\
     &\leq \frac{3}{\zeta^i}+\sum_{n\geq3}\chi(\zeta^i|M_l^i|-n)\left|\frac{n}{\zeta^i}- |M_l^i|\right|\leq \frac{3}{20}2^{-i}\delta_{q+2}+\frac{1}{20}2^{-i}\delta_{q+2}
     \leq \frac15C_02^{-i}\delta_{q+2}.\notag
\end{align}

 By the bounds  \eqref{bd:gwnp}, \eqref{bd:4rql1} and \eqref{para42}  we have
\begin{align}
   \|  M^{i}_{osc,o}\|_{L^1_tL^1}&\lesssim \sigma^{-1}\sum_{n\geq3}\sum_{\xi\in\Lambda^{n}}\|h_{(\xi,d_0)}\|_{L_t^\infty}\norm{\partial_t[\chi(\zeta|M_l^i|-n)\frac{n}{\zeta^i}\Gamma_{\xi}\(\frac{M_l^i}{|M_l^i|}\)\xi]}_{C_{t,x}^0}\notag\\ 
   &\lesssim  \sigma^{-1}l^{-4d-12}
   \lesssim \lambda_{q+1}^{(8d+24)\alpha-\frac1{2N}}
   \lesssim 2^{-i}\lambda_{q+1}^{-\alpha},\notag
\end{align}
where we used   conditions on the parameters to have $(8d+26)\alpha<\frac{1}{2N},2^i\leq2^{N_{q+1}}\lesssim \lambda_{q+1}^\alpha$. 

In summary, we have
\begin{align}
    \|M^{i}_{osc}\|_{L^1_tL^1} \leq C2^{-i} \lambda_{q+1}^{-\alpha}+ \frac15C_02^{-i}\delta_{q+2}\leq \frac{2}{5}C_02^{-i}\delta_{q+2},\notag
\end{align}
where we chose $a$ large to absorb the universal constant.

\textbf{Nonlinear error $M_{nonlin}^i$.} 
     We first consider case Assumption~\ref{def:ass} \textnormal{(1)}. Using the notation $\overline\sigma(\rho)=\sigma(\rho)\sigma^T(\rho)\rho$, we have 
     \begin{align*}
       |\partial_i[ \overline \sigma(\rho_{q+1}^i)  -   \overline \sigma( \rho_l^i)   ]|&\lesssim  |(\partial_i \overline \sigma)(\rho_{q+1}^i)  -   (\partial_i \overline \sigma)(\rho_{l}^i)   |+|(\partial_\rho \overline \sigma)(\rho_{q+1}^i)\partial_i\rho_{q+1}^i  - (  \partial_\rho\overline \sigma)(\rho_{l}^i)\partial_i\rho_{l}^i  ]|\\
       &\lesssim   |\rho_{q+1}^i- \rho_l^i|(1+| \rho_{l}^i|+|\partial_i\rho_{l}^i|)+|\partial_i\rho_{q+1}^i-\partial_i \rho_l^i| ,
     \end{align*}
     which implies    \begin{align*}
        \|\div[A(\rho_{q+1}^i)\rho_{q+1}^i-A({\rho_l^i}){\rho_l^i}]\|_{L_t^1L^1}&\lesssim\| \theta_{q+1}^i \|_{L^1_tL^1}(1+\|\rho_q^i\|_{C_{t,x}^1})+ \|\theta_{q+1}^i\|_{L^1_tW^{1,1}}.
    \end{align*}
     
Next, in case Assumption~\ref{def:ass} \textnormal{(2)}, Young’s inequality yields
\begin{align*}
      \|\div[A(\rho_{q+1}^i)\rho_{q+1}^i&-A({\rho_l^i}){\rho_l^i}]\|_{L_t^1L^1}\\
      &\lesssim \|A(\rho_{q+1}^i)-A(\rho_l^i)\|_{C_{t,x}^0}\|\nabla\rho_{q+1}^i\|_{L^1_tL^1}+\| A(\rho_l^i)\|_{C_{t,x}^0}\|\nabla\theta_{q+1}^i\|_{L^1_tL^1}\\
     &+\|(\div A(\rho_{q+1}^i - \rho_l^i))\rho_{l}^i\|_{L^1_tL^1} +\|\div A( \rho_{q+1}^i)\|_{L_{t}^1C_{x}^0}\|\rho_{q+1}^i-\rho_l^i\|_{C_tL^1}\\
     &\lesssim \| \theta_{q+1}^i \|_{C_tL^1}(\|\rho_{q}^i\|_{C_{t,x}^1}+\|\nabla\theta_{q+1}^i\|_{L^1_tL^1})+ \|\theta_{q+1}^i \|_{L^1_tW^{1,1}}\|\rho_{q}^i\|_{C_{t,x}^0}.
\end{align*}

Case Assumption~\ref{def:ass} \textnormal{(3)} follows directly by combining the estimates from the previous two cases.

By weak Young's inequality and Sobolev embedding, we obtain 
\begin{align*}
    \|(V*\theta_{q+1}^i)\rho_{l}^i\|_{L^1_tL^1}&+\|(V*\rho_{q+1}^i)\theta_{q+1}^i \|_{L^1_tL^1}\\
    &\lesssim \|V*\theta_{q+1}^i\|_{L_t^1L^{1+1/\epsilon}}\|\rho_{l}^i\|_{C_tL^{1+\epsilon}}+\|V*\rho_{q+1}^i\|_{L_t^1L^{1+1/\epsilon}}\|\theta_{q+1}^i \|_{C_tL^{1+\epsilon}}\\
    &\lesssim \| \theta_{q+1}^i\|_{L^1_tL^{d/d-1}}\|\rho_{l}^i\|_{C_tL^{1+\epsilon}}+\| \rho_{q+1}^i\|_{L^1_tL^{d/d-1}}\|\theta_{q+1}^i \|_{C_tL^{1+\epsilon}}\\
    &\lesssim (\|\theta_{q+1}^i \|_{C_tL^{1+\epsilon}}+\| \theta_{q+1}^i\|_{L^1_tL^{d/d-1}})(\|\rho_{l}^i\|_{C_{t,x}^1}+\| \theta_{q+1}^i\|_{L^1_tL^{d/d-1}} )\\
    &\lesssim (\|\theta_{q+1}^i \|_{C_tL^{1+\epsilon}}+\| \theta_{q+1}^i\|_{L^1_tW^{1,1+\epsilon}})(\|\rho_{l}^i\|_{C_{t,x}^1}+\|  \theta_{q+1}^i\|_{L^1_tW^{1,1+\epsilon}} ).
\end{align*}
 
Then, by     the bounds in  \eqref{bd:4theq+1clp}  and \eqref{bd:4theq+1olp}    we have
\begin{align}
    \|M^{i}_{nonlin}\|_{L^1_tL^1} 
      &\lesssim ( \|\theta^{i}_{q+1}\|_{C_tL^{1+\epsilon}}+ \|\theta^{i}_{q+1}\|_{L^1_tW^{1,1+\epsilon}})  (1+\|\rho^{i}_{q}\|_{C_{t,x}^1}+\| \theta_{q+1}^i\|_{L^1_tW^{1,1+\epsilon}} )  \notag\\
      &\lesssim( \lambda_{q}^{d+4}+1)2^{-i}\lambda_{q+1}^{-2\alpha}\lesssim  2^{-i} C_0 \lambda_{q+1}^{-\alpha}\leq \frac{1}{5}C_02^{-i}\delta_{q+2},\label{arhorho-arhorho}
\end{align}
where we used  \eqref{para42}. We also chose $a$ large to absorb the universal constant.

\textbf{Linear error $M^{i}_{lin}$.} 
By the estimates  
in \eqref{bd:4rhoqc1}, \eqref{bd:4wq+1clp}, \eqref{bd:4wq+1lpr}, \eqref{bd:4theq+1pl2} and  \eqref{bd:4theq+1olp} respectively we obtain
\begin{align}
 \|M^{i}_{lin}\|_{L^1_tL^1}&\lesssim \|v_l\|_{C_{t,x}^0}\|\theta_{q+1}^{i}\|_{C_tL^1}+\|\rho^{i}_l+\theta_{q+1}^{(o,i)}\|_{C_{t,x}^0}
\|w_{q+1}\|_{L^r_tL^p}+   \| \theta_{q+1}^{(p,i)}\|_{L^{d_0'}_tL^{d_0'}}\|w_{q+1}^{(c,i)}\|_{L^{d_0}_tL^{d_0}}\notag\\
&\lesssim C_vC_0(\lambda_{q}^{d+4}+1)\lambda_{q+1}^{(12d+40)\alpha-\frac1{2N}}
\lesssim 2^{-i} C_0 \lambda_{q+1}^{-\alpha}\leq  \frac{1}{5}C_02^{-i}\delta_{q+2},\notag
\end{align}
where we used  \eqref{para42} and   conditions on the parameters to have $(12d+43)\alpha<\frac{1}{2N},2^i\leq2^{N_{q+1}}\lesssim \lambda_{q+1}^\alpha$. 

\textbf{Commutator error $M_{com}^i$.} By  the bounds in   \eqref{bd:4rhoqc1}, \eqref{para42} and Lemma \ref{lem:A}  we obtain
 \begin{align}
  \|M_{ com}^i\|_{L_t^1L^1}&\lesssim l\|v_q\|_{C_{t,x}^1}\|\rho_q^i \|_{C_{t,x}^1}+\|(A(\rho_l^i) -A(\rho_q^i)) \rho_l^i\|_{C_{t,x}^1} \notag\\ 
  &\quad+\|A(\rho_q^i)( \rho_l^i- \rho_q^i)\|_{C_{t,x}^1}+l\|A(\rho_q^i) \rho_q^i\|_{C_{t,x}^2}\notag\\
  &\quad +\|(V*(\rho_l^i-\rho_q^i))\rho_l^i\|_{L_{t}^1L^1} +\|(V*\rho_q^i)(\rho_l^i- \rho_q^i)\|_{L_{t}^1L^1} + l\|(\partial_t+\nabla)(V*\rho_q^i)\rho_q^i\|_{L_{t}^1L^1}\notag\\
 & \lesssim l\|v_q\|_{C_{t,x}^1}\|\rho_q^i\|_{C_{t,x}^1}+l (1+  \| \rho_q^i\|_{C_{t,x}^2}+\| \rho_q^i\|_{C_{t,x}^1}^2)\| \rho_q^i\|_{C_{t,x}^1}\notag\\
 &\lesssim C_0 l\lambda_q^{3d+12}\lesssim  2^{-i}C_0  \lambda_{q+1}^{-\alpha/2}\leq \frac{1}{5}C_02^{-i}\delta_{q+2},\label{bd:4mcom}
   \end{align}  
   where we used $2^i\leq2^{N_{q+1}}\lesssim \lambda_{q+1}^\alpha$, and  chose $a$ large to absorb the universal constant.
   
Summarizing all the estimates above  we obtain   \eqref{bd:4rql1}.

\subsubsection{Proof of  \eqref{bd:4rql1} for $i> N_{q+1}$}\label{sec:estmq+1big}

By the corresponding estimates in  \eqref{bd:4vql2} and \eqref{bd:4rho_0,M_0}  we have
\begin{align*}
    \|M^{i}_{q+1}\|_{L^1_tL^1}&\leq \|M^{i}_{0}\|_{L^1_tL^1}+\|v_{q+1}-\overline v\|_{L^{d_0}_tL^{d_0}}\|\rho^{i}_0-1\|_{L^{d_0'}_tL^{d_0'}}\\
    &\lesssim C_0 C_v 4^{-i}\lesssim C_0 C_v2^{-i}2^{-N_{q+1}}\lesssim C_0 C_v 2^{-i}\lambda_{q+1}^{-\alpha}\leq C_0 2^{-i}\delta_{q+2},
\end{align*}
where we  used \eqref{para42} and the conditions on the parameters  to deduce 
$\lambda_{q+1}^{\alpha}\leq2^{N_{q+1}} $ and chose $a$ large enough to absorb the universal constants.

\section{More singular interaction  case}\label{cogpss6} 
In this section, we follow the strategy developed in Section~\ref{cogpss3} to treat the singular interaction case. 
The main idea is to remove the temporal building blocks in order to obtain stronger regularity estimates for $\rho$, which in turn allows us to justify the well-posedness of the convolutions involving singular interaction kernels. The price to pay is that we only obtain estimates in the supremum-in-time norm, and spatial integrability estimates up to $L^{d-}$.
 
Let $N>0,0<\epsilon_0\leq \frac14$ and $1<d_0<d$ be given. We recall $d_0'=\frac{d_0}{d_0-1}$. The iteration is again indexed by a parameter
$
q\in\mathbb{N}_0.
$
We consider an increasing sequence
$
\{\lambda_q\}_{q\in\mathbb{N}_0}\subset\mathbb{N}
$
and a sequence
$
\{\delta_q\}_{q\in\mathbb{N}_0}\subset(0,1]
$
defined  by \begin{align*}
\lambda_q
=
a^{b^q},
\ 
q\geq0,
\ 
\delta_q
=
(\frac{\epsilon_0}{2})^{d_0+d_0'}
\lambda_1^{2\beta}
\lambda_q^{-2\beta},
\ 
q\geq1,
\ 
\delta_0=1.
\end{align*} By imposing
\begin{align}
a^{2(b-1)\beta/(d_0+d_0')}>2,
\label{ieq:6ab2}
\end{align}
we obtain
\begin{align*}
\sum_{q\geq1}\delta_q^{1/(d_0+d_0')}
&\leq
\frac{\epsilon_0}{2} 
\sum_{q\geq1}
a^{(1-q)2(b-1)\beta/(d_0+d_0')}\leq
\frac{\epsilon_0}{2}
\cdot
\frac1{1-a^{-2(b-1)\beta/(d_0+d_0')}}
<
\epsilon_0.
\end{align*}

Without loss of generality, we assume $T=1$ from now on. 
At each step $q$, we construct a family
$(v_q,\rho_q^i,M_q^i)_{1\leq i\leq {N}}$
solving the system
\begin{align}\label{eq:6qth}
\partial_t\rho_q^i
-
\div\div\bigl(A(\rho_q^i)\rho_q^i\bigr)
+
\div(v_q\rho_q^i)+\div((V*\rho_q^i)\rho_q^i)
&=
-\div M_q^i,
\\
\div v_q&=0.
\notag
\end{align} 

We define 
$
T_q:=\frac13-\sum_{1\leq r\leq q}\delta_r^{1/2}.
$ 
We initialize the iteration by setting
\begin{align*}
\rho_0^i(t,x)
&=
1+\frac{\sin2\pi x_1}{4^i}\chi_0(t),
 \ 
v_0=\overline v,\\
M_0^i(t,x)
&=
\partial_t\chi_0(t)\frac{\cos2\pi x_1}{4^i\cdot 2\pi}(1,0,\dots,0)
-\overline v(\rho_0^i-1)  +
\div( A(t,x,\rho_0^i) \rho_0^i)-(V*\rho_0^i)\rho_0^i,
\end{align*}
where $x=(x_1,\dots,x_d)$ and $\chi_0$ is a smooth function satisfying
$
\chi_0(t)=0 $ on $
 [0,\frac13 ], \chi_0(t)=1$ on $[\frac23,1 ].
$   It is easy to verify that $(v_0,\rho^{i}_0,  M^{i}_0)_{1\leq i\leq N}$ is a solution to \eqref{eq:4qth}.   By definition, we know that $\rho_0^i=1$ on $[0,T_0]$, and then $M_0^i=0$ on $[0,T_0]$  by using the facts that $\div \overline v=0$, $\div(a*1)=0$.

Then by \eqref{eq:hls}  we have 
\begin{align}
    \| M^i_0\|_{C_tL^1}&\lesssim 1+\|(a*\nabla \rho_0^i)\rho_0^i\|_{C_tL^1}+\|(a*
\rho^i_0)\nabla \rho^i_0\|_{C_tL^1}+\|(b*
\nabla\rho^i_0) \rho^i_0\|_{C_tL^1}\notag\\
&\lesssim 1+\|a*\nabla \rho_0^i\|_{C_tL^1}+\|a*
\rho^i_0\|_{C_tL^1}+\|b*\nabla
\rho^i_0\|_{C_tL^1}\notag\\
&\lesssim 1+\|\rho^i_0\|_{C_{t,x}^1}\leq
(\frac{\epsilon_0}{2})^{d_0+d_0'}C_0.\label{bd:6rho_0,M_02}
 \end{align}

With the above assumptions in hand, our main iteration relies on the first step of iteration and reads as follows:

\bp\label{prop:case6} 
Under the assumption of Theorem \ref{thm:6converge}, there exists   a choice of parameters $a,b,\beta$  such
that the following holds: Let $(v_q,\rho^{i}_q,  M^{i}_q)_{1\leq i\leq N}$ be a solution to the system \eqref{eq:6qth} satisfying  $\int\rho^{i}_q\dif x=1$,
and
\begin{align}
\|v_q\|_{C_{t,x}^1}\leq C_0^{1/d_0}\lambda_q^{2d+2},\
\|\rho^{i}_{q}\|_{C_{t,x}^1}+\lambda_q^{-d-1}\|\rho^{i}_{q}\|_{C_{t,x}^2}\leq C_0^{1/d_0'}\lambda_q^{2d+2},
\label{bd:6rhoqc1}\\
 \|M^{i}_q\|_{C_tL^1}\leq C_0 \delta_{q+1},&\label{bd:6rql1} \\
  \rho^{i}_q(t)=1, M^{i}_q(t)=0\ {\rm on}\ [0,T_q].&\label{bd:6mq=rhoq=0}
\end{align}
Then there exists 
$(v_{q+1},\rho^{i}_{q+1},  M^{i}_{q+1})_{1\leq i\leq N}$ which solves \eqref{eq:6qth} and satisfies \eqref{bd:6rhoqc1}-\eqref{bd:6mq=rhoq=0} at the level $q+1$ and for some universal constant $C_v\geq1$
\begin{align}
\|v_{q+1}-v_{q}\|_{C_tL^{d_0}}\leq
C_vC_0^{1/d_0}\delta_{q+1}^{1/d_0},\label{bd:6vq+1-vqldd}\ \
\|\rho^{i}_{q+1}-\rho^{i}_{q}\|_{C_tL^{d_0'}}\leq
C_v C_0^{1/{d_0'}} \delta_{q+1}^{1/{d_0'}}.
\end{align}
Moreover, we have for some $\epsilon>0$
\begin{align}
   \|v_{q+1}-v_q\|_{C_tL^1}\leq
\delta_{q+1}^{1/d_0},\label{bd:6vq+1-vqlpr}\\
\|\rho^{i}_{q+1}-\rho^{i}_{q}\|_{C_tW^{1,1+\epsilon}}
\leq
 4^{-i}\delta_{q+1}^{1/d_0'},\ \  \inf_{t\in [0,1]}(\rho^{i}_{q+1} - \rho^{i}_q)\geq - \delta_{q+1}^{1/d_0'} .\label{bd:6rhoq+1-rhoql1}
\end{align}
\ep

\begin{proof}[Proof of Theorem \ref{thm:6converge}]
We intend to start the iteration from 
$(v_0, \rho^{i}_0,  M^{i}_0)_{1\leq i\leq N}$ which are defined as above. 
By \eqref{bd:6rho_0,M_02},  \eqref{bd:6rhoqc1}-\eqref{bd:6mq=rhoq=0} are satisfied as $\delta_1=(\frac{\epsilon_0}{2})^{d_0+d_0'}$.
Next, we use Proposition \ref{prop:case6} to build inductively $(v_q,\rho^{i}_q, M^{i}_q)_{1\leq i\leq N}$. By \eqref{ieq:6ab2} and \eqref{bd:6vq+1-vqldd}-\eqref{bd:6rhoq+1-rhoql1}, the sequence $\{(v_q,\rho^{i}_q)_{1\leq i\leq N}\}_{q\in \N}$ is
Cauchy in $C([0,1];L^{d_0})  \times (C([0,1];L^{d_0'}\cap W^{1,1+\epsilon}))^{N}$ 
     and we denote by $(v,\rho^{i})$ its limit.

     Now we need to verify $(\rho^{i},v)$  solves   \eqref{eq:fpe:i}.   By \eqref{eq:hls}, we have for  $\epsilon>0$ small enough, $q'=1+\epsilon,q=\frac{q'}{q'-1},\frac1q=\frac1p-\frac\gamma d$, $p_1=1+\epsilon,\frac1{q_1}=\frac1{p_1}-\frac\gamma d,q_1'=\frac{q_1}{q_1-1}$, we have
\begin{align}
     \|(a*\rho)\nabla\rho&-(a*\tilde{\rho})\nabla\tilde{\rho}\|_{L^1}+\|(a*\nabla \rho )\rho -(a*\nabla\tilde{\rho})\tilde{\rho}\|_{L^1}\notag\\
     &\lesssim \|a*(\rho-\tilde{\rho})\|_{L^q}\|\nabla\rho\|_{L^{q'}}+\| a*\tilde{\rho}\|_{L^q}\|\nabla(\rho-\tilde{\rho})\|_{L^{q'}}\notag\\
     &+\|  a*\nabla(\rho-\tilde\rho)\|_{L^{q_1}}\|\rho\|_{L^{q_1'}}+\| a*\nabla\tilde\rho\|_{L^{q_1}}\|\rho-\tilde{\rho}\|_{L^{q_1'}}\notag\\
     &\lesssim \| \rho-\tilde{\rho}\|_{L^p}\|\nabla\rho\|_{L^{q'}}+\| \tilde{\rho}\|_{L^p}\|\nabla(\rho-\tilde{\rho})\|_{L^{q'}}\notag\\
     &+\|  \nabla(\rho-\tilde\rho)\|_{L^{p_1}}\|\rho\|_{L^{q_1'}}+\|  \nabla\tilde\rho\|_{L^{p_1}}\|\rho-\tilde{\rho}\|_{L^{q_1'}}.\label{bd:6arhorho-}
\end{align}
The similar calculation also holds for $b$. Then we know that 
\begin{align}
     \|(a*\rho^i_q)\nabla\rho^i_q&-(a* {\rho^i})\nabla {\rho^i}\|_{L^1}+\|(a*\nabla \rho_q^i )\rho^i_q -(a*\nabla {\rho^i}) {\rho^i}\|_{L^1}\to0\notag
\end{align}
By \eqref{bd:6rql1}, \eqref{bd:6mq=rhoq=0}  we then obtain that $(\rho^{i},v)$  solves   \eqref{eq:fpe:i}.

 By \eqref{ieq:6ab2}, \eqref{bd:6vq+1-vqldd} and \eqref{bd:6rhoq+1-rhoql1} we have
\begin{align} 
 \left|   \|\rho^i-1\|_{C_tL^1}- \| \rho^{i}_0-1\|_{C_tL^1}\right|\leq \sum_{q=0}^\infty\|\rho^{i}_{q+1}-\rho^{i}_{q}\|_{C_tL^1}\leq
4^{-i}\sum_{q=0}^\infty\delta_{q+1}^{1/{d_0'}}\leq 4^{-i-1},\\
\inf_{t\in [0,1]}\rho^i \geq \inf_{t\in [0,1]} \rho_0^i +
\sum_{q=0}^\infty \inf_{t\in [0,1]}(\rho_{q+1}^i- \rho_q^i) \geq\frac34-\sum_{q=0}^\infty \delta_{q+1}^{1/d_0'}\geq \frac12,\notag
\end{align}
at which point, by a similar argument as before,  we obtain 
 $\rho^{i}$ is nonnegative on $\mathbb{T}^d$, and do not coincide with each other. 

By \eqref{ieq:6ab2} and \eqref{bd:6vq+1-vqlpr} we obtain that $ \|v -\overline v \|_{C_tL^1}\leq  \sum_{q\geq0} \|v_{q+1}-v_q\|_{C_tL^1}\leq
\sum_{q\geq0}\delta_{q+1}^{1/d_0}\leq\epsilon_0.$ 

    By \eqref{bd:6rhoqc1}, \eqref{bd:6vq+1-vqldd} and interpolation, we have $|v| \in C_tL^{d_0(1+\epsilon)}$ for some $\epsilon > 0$ sufficiently small. Since $\rho^i \in C_tL^{d_0'}$, we deduce that $|v|^{1+\epsilon} \rho^i \in C_tL^{1}$. We conclude the proof by using the superposition principle.
\end{proof}

\section{Proof of Proposition \ref{prop:case6}}\label{sec:6convex}

The construction and the corresponding estimates are essentially the same as those in Section~\ref{sec:convex}, after formally regarding the temporal building blocks $g_{(\xi)}$ as identically equal to $1$.
 We shall choose different parameters in order to obtain higher regularity for $\rho$.

\subsection{Mollification}
For a sufficiently small parameter $\alpha\in(0,1)$ to be chosen later, we define
$
l:=\lambda_{q+1}^{-3\alpha/2}\lambda_q^{-5d/2-5/2}.
$
Then
\begin{align}
l^{-1}
\leq
\lambda_{q+1}^{2\alpha},
 \ \
l\lambda_q^{5d+5}
\leq
\lambda_{q+1}^{-\alpha} 
\ll
\delta_{q+2},
\label{para62}
\end{align}
provided that
\begin{align}
\alpha b>5d+5,
\ \
\alpha>2\beta b.
\label{alphab>6}
\end{align}

To avoid the loss of derivatives, we first mollify the stress term. As before, 
we define  
\begin{align}
v_l=(v_{q }*_x\phi_l)*_t\varphi_l,\ \ 
 \rho_l^i=(\rho_{q}^i*_x\phi_l)*_t\varphi_l,\ \ M_l^i=(M_{q }^i*_x\phi_l)*_t\varphi_l.\notag
\end{align}
Moreover,    we know that $ \rho_l^i=1 $ and $ M_l^i=0$ on $[0, T_{q+1}]$.
  
By straightforward calculations  we obtain 
\begin{align}
\partial_t \rho_l^i-\div\div(A(\rho_l^i)\rho_l^i)+\div(v_l \rho_l^i)+\div((V*\rho_l^i)\rho_l^i)&=-\div (M^i_l+M^i_{com}),\ \ \div v_l=0,\label{eq:6v_l}
\end{align} where
 \begin{align*}
M_{com}^i&:=-v_l \rho^i_l+(v_{q } \rho^i_{q})*_x\phi_l*_t\varphi_l+\div(A(\rho_l^i) \rho_l^i)-(\div(A(\rho_q^i) \rho_q^i))*_x\phi_l*_t\varphi_l\\
&\ \ \ \   -(V*\rho_l^i)\rho_l^i+((V*\rho_q^i)\rho_q^i)*_x\phi_l*_t\varphi_l.
\end{align*}

Finally, by   Sobolev embedding, we still have for every $N\geq0$,
\begin{align*}
\|M_l^i\|_{C_{t,x}^N}
&\lesssim
l^{-d-\frac13-N}
\|M_q^i\|_{C_tL^1}
\lesssim
C_0l^{-d-\frac43-N}.
\end{align*}
Here we still write $l^{-\frac43}$ so that the estimate remains formally the same as before.

\subsection{Construction of $(v_{q+1},\rho_{q+1})$}\label{sec:6defq+1}
To construct the perturbations of $v_l$ and $\rho^i_l$, we  still consider the  space-building blocks in Appendix~\ref{gij}.   We consider $2N$ disjoint sets $\Lambda^{1,n}, \Lambda^{2,n}$ defined in Appendix \ref{gij} by taking suitable rational rotations of one fixed set. We use the notation   $ \Lambda^{j,n} =  \Lambda^{1,n}$ for $j$ odd, and $ \Lambda^{j,n} =  \Lambda^{2,n}$ for $j$ even. 
For a fixed exponent
$
d_0>1,
$ and $\xi\in\Lambda:=\cup_{i=1}^N(\Lambda^{1,i}\cup \Lambda^{2,i})$
we   recall the notation
$
\Phi_{(\xi,d_0)},
\phi_{(\xi,d_0)},
\phi'_{(\xi,d_0')}
$
introduced in Appendix~\ref{gij}. 
We  recall the parameters
$
\lambda,\ r_\perp,\ r_\parallel,\  \mu,
$
introduced in Appendix~\ref{gij}. Their precise choice will be specified subsequently.

\br
 Here constructing infinitely many   solutions appears to be   nontrivial. 
Indeed, if one attempts to treat $N$ equations simultaneously, many parameters arising in the spatial building blocks (such as $n_*$) necessarily depend on $N$. 
It then becomes necessary to carefully control the growth of these parameters, which should increase at most polynomially in $N$. In the former case, this difficulty can be overcome by introducing suitable time building blocks. 
\er
To ensure that the building blocks have mutually disjoint supports, we introduce the following lemma.
\bl\label{lem:4}
$($\cite[Lemma 3]{BMS}$)$
Let $d\geq3$. Then there exist $\{\alpha_\xi\}_{\xi\in\Lambda}$ and $a>0$ such that
$$(B_a(\alpha_\xi)+\{s\xi\}_{s\in \mathbb{R}}+(\mathbb{Z}/r_\perp\lambda)^d)\cap
(B_a(\alpha_{\xi'})+\{s'\xi'\}_{s'\in \mathbb{R}}+(\mathbb{Z}/r_\perp\lambda)^d)=\varnothing,$$
for all $\xi,\xi'\in\Lambda ,\xi\neq\xi'$.
\el

Let
$
\chi\in C_c^\infty(-\frac34,\frac34)
$
be a nonnegative function such that
$
\sum_{n\in\mathbb{Z}}\chi(t-n)=1
$
for every $t\in\mathbb{R}$.
Let
$
\tilde{\chi}\in C_c^\infty(-\frac45,\frac45)
$
be a nonnegative function satisfying
$
\tilde{\chi}=1
$
on
$
[-\frac34,\frac34]
$
and
$
\sum_{n\in\mathbb{Z}}\tilde{\chi}(t-n)\leq2.
$

For $1\leq i\leq N$, we fix the parameters
$
\zeta:=20\delta_{q+2}^{-1}.
$ 
With these preparations, we define the rescaled building blocks by
\begin{align*}
W_{(\xi,n)}(x,t)
&:=
W_{(\xi,d_0)}
\Bigl(
x-\alpha_\xi,
\Bigl(\frac{n}{\zeta}\Bigr)^{1/d_0}
t
\Bigr),
\\
\Theta_{(\xi,n)}(x,t)
&:=
\Theta_{(\xi,d_0')}
\Bigl(
x-\alpha_\xi,
\Bigl(\frac{n}{\zeta}\Bigr)^{1/d_0}
t
\Bigr).
\end{align*}
By the choice of the shifts $\alpha_\xi$ in Lemma~\ref{lem:4}, the building blocks $W_{(\xi)}$ have mutually disjoint supports.
Similarly, we define
$
V_{(\xi,n)},
\Phi_{(\xi,n)},
$
and all the other quantities appearing in Appendix~\ref{gij}. 
By \eqref{eq:ptthe+} and \eqref{eq:parth}, we have
\begin{align}
\partial_t\Theta_{(\xi,n)}
+
\Bigl(\frac{n}{\zeta}\Bigr)^{1/d_0}
\div\bigl(
W_{(\xi,n)}\Theta_{(\xi,n)}
\bigr)
=
0. \notag
\end{align}

We next define the perturbations for the drift term. 
For $1\leq i\leq N$, let
\begin{align*}
w_{q+1}^{(p,i)}
&:=
\sum_{n\geq3}
\tilde{\chi}(\zeta|M_l^i|-n)
\Bigl(\frac{n}{\zeta}\Bigr)^{1/d_0}
\sum_{\xi\in\Lambda^{n,i}}
W_{(\xi,n)},\notag\\
w_{q+1}^{(c,i)}
&:=
\sum_{n\geq3}
\sum_{\xi\in\Lambda^{n,i}}
-\tilde{\chi}(\zeta|M_l^i|-n)
\Bigl(\frac{n}{\zeta}\Bigr)^{1/d_0}
\frac1{(n_*\lambda_{q+1})^2}
\nabla\Phi_{(\xi,n)}
\,\xi\cdot\nabla\psi_{(\xi,n)}
\\
&\quad
+
\nabla\bigl(
\tilde{\chi}(\zeta|M_l^i|-n)
\bigr)
\Bigl(\frac{n}{\zeta}\Bigr)^{1/d_0}
:V_{(\xi,n)}
.
\end{align*}

By \eqref{divOmega}, we obtain
\begin{align}
w_{q+1}^{(p,i)}
+
w_{q+1}^{(c,i)}
=
\sum_{n\geq3}
\sum_{\xi\in\Lambda^{n,i}}
\div\Bigl(
\tilde{\chi}(\zeta|M_l^i|-n)
\Bigl(\frac{n}{\zeta}\Bigr)^{1/d_0}
V_{(\xi,n)}
\Bigr).
\label{eq:6wq+1p+wq+1c}
\end{align}
Since $V_{(\xi,n)}$ is skew-symmetric, it follows that
$
\div\bigl(
w_{q+1}^{(p,i)}
+
w_{q+1}^{(c,i)}
\bigr)=0.
$

Finally, we define the total perturbation and the new velocity field by
\begin{align}
w_{q+1}
:=
\sum_{i=1}^{N}
\bigl(
w_{q+1}^{(p,i)}
+
w_{q+1}^{(c,i)}
\bigr),
\qquad
v_{q+1}
:=
v_l+w_{q+1}.
\notag
\end{align}
Then $v_{q+1}$ is mean-zero and divergence-free.
Moreover, since $M_l^i(t)=0$ on $[0,T_{q+1}],$ the perturbation $w_{q+1}$  vanishes  on $[0,T_{q+1}]$. 
 
We next define the perturbations for the density functions. 
For $1\leq i\leq N$, we set
\begin{align}
\theta_{q+1}^{(p,i)}
:=
&\sum_{n\geq3}
\chi(\zeta|M_l^i|-n)
\Bigl(\frac{n}{\zeta}\Bigr)^{1/d_0'}
\sum_{\xi\in\Lambda^{n,i}}
\Gamma_\xi\Bigl(\frac{M_l^i}{|M_l^i|}\Bigr)
\Theta_{(\xi,n)}
,
\notag\\
\theta_{q+1}^{(c,i)}
:=
&-
\int_{\mathbb{T}^d}
\theta_{q+1}^{(p,i)}\,\dif x.
\notag
\end{align}
By an argument analogous to \eqref{wq+1good}, we obtain
\begin{align}
&
\partial_t\theta_{q+1}^{(p,i)}
+
\div\bigl(
w_{q+1}^{(p,i)}
\theta_{q+1}^{(p,i)}
-
M_l^i
\bigr)
\notag\\
=
&
\sum_{n\geq3}
\sum_{\xi\in\Lambda^{n,i}}
\partial_t
\Bigl[
\chi(\zeta|M_l^i|-n)
\Bigl(\frac{n}{\zeta}\Bigr)^{1/d_0'}
\Gamma_\xi\Bigl(\frac{M_l^i}{|M_l^i|}\Bigr)
\Bigr]
\Theta_{(\xi,n)}+
\div
\Bigl(
\sum_{n\geq3}
\chi(\zeta|M_l^i|-n)
\frac{n}{\zeta}
\frac{M_l^i}{|M_l^i|}
-
M_l^i
\Bigr)
\notag\\
&+
\sum_{n\geq3}
\sum_{\xi\in\Lambda^{n,i}}
\nabla
\Bigl[
\chi(\zeta|M_l^i|-n)
\frac{n}{\zeta}
\Gamma_\xi\Bigl(\frac{M_l^i}{|M_l^i|}\Bigr)
\Bigr]
\mathbb{P}_{\neq0}
\bigl(
W_{(\xi,n)}
\Theta_{(\xi,n)}
\bigr).
\label{6wq+1good}
\end{align}

We now define, for every $1\leq i\leq N$,
\begin{align}
\theta_{q+1}^i
:=
\theta_{q+1}^{(p,i)}
+
\theta_{q+1}^{(c,i)}
,
\qquad
\rho_{q+1}^i
:=
\rho_l^i+\theta_{q+1}^i.
\notag
\end{align}
By construction,
$
\int_{\mathbb{T}^d}\rho_{q+1}^i\,\dif x=1
$.
Since
$
M_l^i(t)=0
$
on
$
[0,T_{q+1}],
$
we know $\theta_{q+1}^i(t)=0,\rho_{q+1}^i(t)=1
$
on
$
[0,T_{q+1}]
$.

Moreover, by Lemma \ref{lem:4}, we obtain for
$
1\leq i\leq N 
$
that
\begin{align}
\div\bigl(
w_{q+1}\theta_{q+1}^i
\bigr)
=
\div\(
(w_{q+1}^{(p,i)}+w_{q+1}^{(c,i)})
\theta_{q+1}^{(p,i)}
\).\notag
\end{align}

\subsection{Construction of the stress terms $M^{i}_{q+1}$}\label{sec:6defmq+1}

From  the definition of the perturbations, as in Section \ref{sec:4defmq+1},  
we   define 
 $$-{M}^i_{q+1}:=M^{i}_{osc}+M^{i}_{nonlin}+M^{i}_{lin}-M_{com}^i.$$
 Here, we define the nonlinear error and the linear error   by
\begin{align*} 
M^{i}_{nonlin}:&= -\div( (a*\rho^{i}_{q+1}) \rho^{i}_{q+1}-(a*\rho^{i}_{l}) \rho^{i}_{l} )-(b*\nabla\rho_{q+1}^i)\rho_{q+1}^i+(b*\nabla\rho_{l}^i)\rho_{l}^i,\\
M^{i}_{lin}:&=v_l\theta^{i}_{q+1}+w_{q+1}\rho^{i}_l+w_{q+1}^{(c,i)}\theta_{q+1}^{(p,i)}.
\end{align*}
 We define the oscillation error  $M^{i}_{osc}:=M^{i}_{osc,t}+M^{i}_{osc,x}+M^{i}_{osc,c}$ as
\begin{align*}
   M^{i}_{osc,t}&:=\sum_{n\geq3}\sum_{\xi\in\Lambda^{n,i}}\mathcal{R}_{1}\(\partial_t[\chi(\zeta|M^{i}_l|-n)\(\frac{n}{\zeta}\)^{1/d_0'}\Gamma_{\xi}\(\frac{M^{i}_l}{|M^{i}_l|}\)\Theta_{(\xi,n)})\),\\
M^{i}_{osc,x}&:=\sum_{n\geq3}\sum_{\xi\in\Lambda^{n,i}}\mathcal{B}_{1}\(\nabla[\chi(\zeta|M^{i}_l|-n)\frac{n}{\zeta}\Gamma_{\xi}\(\frac{M^{i}_l}{|M^{i}_l|}\)],\mathbb{P}_{\neq0}(W_{(\xi,n)}\Theta_{(\xi,n)})\),\\
   M^{i}_{osc,c}&:=\sum_{n\geq3}\chi(\zeta|M^{i}_l|-n)\frac{n}{\zeta}\frac{M^{i}_l}{|M^{i}_l|}- M^{i}_l.
\end{align*}
 
Since $M^{i}_l(t)=0,w_{q+1}(t)=0,\theta^i_{q+1}(t)=0$ on $[0,T_{q+1}]$,  we have $M_{q+1}^i(t)=0$ on $[0,T_{q+1}]$, which implies  \eqref{bd:6mq=rhoq=0} for $M_{q+1}^i$.

\subsection{ Proof of Proposition \ref{prop:case6}}   
\subsubsection{Choice of parameters} 
Regarding the parameters of the building blocks, we define
  
\begin{align}
  \lambda=\lambda_{q+1},\ r_\perp=\lambda_{q+1}^{-1+\frac1{M}},\ r_\parallel=\lambda_{q+1}^{-1+\frac2{M}},\   \mu=r_\perp^{-\frac{d-1}{d_0}}r_\parallel^{-\frac1{d_0}}, \notag
\end{align}
where $M>0$ is a sufficiently large integer satisfying $(-1+\frac2M)d(\frac{\gamma}{d}-\frac{1}{d_0'})<-\frac1M,(-1+\frac2M)\frac{d}{d_0'}<-\frac1M,1+(-1+\frac2M)\frac{d}{d_0}<-\frac1M$.  
With these choices,  we obtain
\begin{align}
  r_\perp^{\frac{d-1}{d/\gamma}-\frac{d-1}{d_0'}} r_\parallel^{\frac1{d/\gamma}-\frac{1}{d_0'}},\    r_\perp^{ d-1 -\frac{d-1}{d_0}} r_\parallel^{1-\frac{1}{d_0}},\ \lambda r_\perp^{\frac{d-1}{1}-\frac{d-1}{d_0'}} r_\parallel^{\frac{1}{1}-\frac{1}{d_0'}}\leq\lambda^{-\frac1M}.\label{bd:paralam6}
\end{align}

In the sequel, we shall also require \eqref{alphab>6} together with \begin{align}
   \lambda_{q}^{2d+2}\leq\lambda_{q+1}^\alpha,\ \ (12d+41)\alpha<\frac1{M}.\notag
 \end{align} 
 
These conditions can be achieved as follows:
first choose $\alpha>0$ sufficiently small such that
$
(12d+41)\alpha<\frac1{M},
$
then choose
$
b\in2N\mathbb{N}
$
sufficiently large so that
$
b>\frac{5d+5}{\alpha},
$
and finally choose $\beta>0$ sufficiently small such that
$
\alpha>2\beta b.
$
At the end, we choose $a$ sufficiently large in order to absorb all implicit and universal constants arising in the subsequent estimates, and to guarantee the validity of \eqref{ieq:6ab2}.

Finally, we record the following estimate for the amplitude functions.

\bl\label{lem:6chi}$($\cite[Proposition 5.2]{LRZ25}$)$.
 For $ M,K\in\mathbb{N}_0,1\leq i\leq N$  we have
\begin{align*}
    \sum_{n\geq3} \(\frac{n}{\zeta}\)^M\(\|\chi(\zeta|M^{i}_l|-n)\|_{C_{t,x}^K}+ \|\tilde{\chi}(\zeta|M^{i}_l|-n)\|_{C_{t,x}^K}\)&\lesssim l^{- (d+4)K-(d+2)(M+1)},\\
     \sum_{n\geq3}\sum_{\xi\in\Lambda^{n,i}}\(\frac{n}{\zeta}\)^M\norm{\chi(\zeta|M^{i}_l|-n)\Gamma_\xi\(\frac{M^{i}_l}{|M^{i}_l|}\)}_{C_{t,x}^K}& \lesssim l^{-(2d+8)K-(d+2)(M+1)}.
\end{align*}
\el 

\subsubsection{Proof of  \eqref{bd:6vq+1-vqldd}  for $v_{q+1}-v_q$.}
We first estimate the principal perturbations $w_{q+1}^{(p,i)}$ for   $1\leq i\leq N $ in the $C_tL^{d_0}$-norm. By Cauchy's inequality    we have
\begin{align}
|w_{q+1}^{(p,i)}|^{d_0} 
&\lesssim \sum_{n\geq3}\tilde{\chi}(\zeta|M^{i}_l|-n)\frac{n}{\zeta}\sum_{\xi\in\Lambda^{n,i}}\left|W_{(\xi,n,i)} \right|^{d_0}.\notag
\end{align} 

By applying the generalized H\"older inequality of 
Theorem \ref{ihiot} in spatial direction, together with the estimates for the building blocks in \eqref{int4} and Lemma \ref{lem:6chi} we deduce 
\begin{align}
\|w_{q+1}^{(p,i)}(t)\|^{d_0}_{L^{d_0}}
&\lesssim\sum_{n\geq3}\norm{\tilde{\chi}(\zeta|M^{i}_l(t)|-n)\frac{n}{\zeta}}_{L^1}\sum_{\xi\in\Lambda^{n,i}}\|W_{(\xi,n,i)}\|_{C_tL^{d_0}}^{d_0}\notag\\ 
&\quad+(r_\perp\lambda_{q+1})^{-1}\norm{\tilde{\chi}(\zeta|M^{i}_l(t)|-n)\frac{n}{\zeta}}_{C_{t,x}^1}\sum_{\xi\in\Lambda^{n,i}}\|W_{(\xi,n,i)}\|_{C_tL^{d_0}}^{d_0} \notag\\
&\lesssim\norm{ M^{i}_l(t)+|\zeta|^{-1}
}_{L^1}+l^{-3d-8}\lambda_{q+1}^{-\frac{1}{M}} \lesssim \|M^{i}_l(t)\|_{L^1}+C_0 \delta_{q+1} \lesssim C_0\delta_{q+1} 
,\notag
\end{align}
where we used   the fact that $\sum_{n\in\mathbb{Z}} \tilde{\chi}(t - n) \leq2$, and used the conditions on the parameters to have $(6d+16)\alpha-\frac{1}{M}<-\alpha<-2\beta$.  

For the general $C_tL^m$-norm with $m\in[1,\infty]$, by the estimates for the building blocks in \eqref{int2}-\eqref{int4} 
and the estimate for the amplitude function in  Lemma \ref{lem:6chi} we obtain
\begin{align}
\|w_{q+1}^{(p,i)}\|_{C_tL^m}&\lesssim\sum_{n\geq3}\sum_{\xi\in\Lambda^{n,i}}\norm{\tilde{\chi}(\zeta|M^{i}_l|-n)\(\frac{n}{\zeta}\)^{1/d_0}}_{C_{t,x}^0}\|W_{(\xi,n,i)}\|_{C_tL^m}\notag\\
&\lesssim l^{-2d-4}r_\perp^{\frac{d-1}{m}-\frac{d-1}{d_0}} r_\parallel^{\frac{1}{m}-\frac{1}{d_0}},\label{bd:6wq+1plp}\\
\|w_{q+1}^{(c,i)}\|_{C_tL^m}&\lesssim\sum_{n\geq3}\sum_{\xi\in\Lambda^{n,i}}\norm{\tilde{\chi}(\zeta|M^{i}_l|-n)\(\frac{n}{\zeta}\)^{1/d_0}}_{C^1_{t,x}}\notag\\ 
&\quad\quad \times\(\frac{1}{\lambda_{q+1}^2}\|\nabla\Phi_{(\xi,n,i)} \xi\cdot\nabla\psi_{(\xi,n,i)} \|_{C_tL^m}+\|V_{(\xi,n,i)}\|_{C_tL^m}\)\notag\\
&\lesssim l^{-3d-8}r_\perp^{\frac{d-1}{m}-\frac{d-1}{d_0}} r_\parallel^{\frac{1}{m}-\frac{1}{d_0}}\frac{r_\perp}{ r_\parallel}.\label{bd:6wq+1clp}
\end{align}

With these estimates, combining with the choice of parameters in  \eqref{para62}  we obtain 
\begin{align}
\|w_{q+1}\|_{C_tL^{d_0}}
&\leq \frac{C_v}4C_0^{1/d_0}\delta_{q+1}^{1/d_0}+C\lambda_{q+1}^{(6d+17)\alpha-\frac1M} \leq \frac{C_v}2C_0^{1/d_0}\delta_{q+1}^{1/d_0},\label{bd:6wq+1l2}\end{align}
where we used  conditions on the parameters to have
${(6d+18)\alpha<\frac{1}{M}},N \lesssim\lambda_{q+1}^{\alpha}$ and chose $a$ large enough to absorb the universal constant. 
The above inequality   yields that \eqref{bd:6vq+1-vqldd} holds for $v_{q+1}-v_q$:
\begin{align*}
\|v_{q+1}-v_{q}\|_{C_tL^{d_0}}&\leq\|w_{q+1} \|_{C_tL^{d_0}}+\|v_{l}-v_{q}\|_{C_tL^{d_0}}\\
&\leq \frac12
C_vC_0^{1/d_0}\delta_{q+1}^{1/d_0}+lC_0^{1/d_0}\lambda_q^{d+4} \leq
C_vC_0^{1/d_0}\delta_{q+1}^{1/d_0}.
\end{align*}

\subsubsection{Proof of  \eqref{bd:6vq+1-vqlpr}}

Combining  the bounds  \eqref{bd:6wq+1plp} and \eqref{bd:6wq+1clp} above we obtain 
\begin{align*}
 \|v_{q+1}-v_q\|_{C_tL^{1}}&\lesssim  \|w_{q+1}\|_{C_tL^{1}} +l\|v_q\|_{C_{t,x}^1}\notag \\
&\lesssim  N l^{-3d-8}r_\perp^{d-1-\frac{d-1}{d_0}} r_\parallel^{1-\frac{1}{d_0}}+l \lambda_q^{d+4}\lesssim \lambda_{q+1}^{(6d+17)\alpha-\frac1M}+\lambda_{q+1}^{-\alpha}
\leq\delta_{q+1}^{1/d_0}. 
\end{align*}
  Here we used \eqref{para62}, \eqref{bd:paralam6} and  conditions on the parameters to have 
${(6d+18)\alpha<\frac{1}{M}},N \lesssim\lambda_{q+1}^{\alpha}$. Then we chose $a$ large enough to absorb the universal constant.

\subsubsection{Proof of  \eqref{bd:6rhoqc1} for $v_{q+1}$.}
 By the estimates for the building blocks in \eqref{int4},  \eqref{eq:6wq+1p+wq+1c} and the estimates for the amplitude functions in Lemma \ref{lem:6chi} we have  
\begin{align}
\|w_{q+1}\|_{C_{t,x}^1}&\lesssim\sum_{i=1}^{N}\sum_{n\geq3}\sum_{\xi\in\Lambda^{n,i}}\norm{\tilde{\chi}(\zeta|M^{i}_l|-n)\(\frac{n}{\zeta}\)^{1/d_0}}_{C_{t,x}^{2}}\|\nabla V_{(\xi,n,i)}\|_{C_{t,x}^1} \notag\\
&\lesssim Nl^{-4d-12}\lambda_{q+1}\mu r_\parallel^{-\frac{1}{d_0}}r_\perp^{-\frac{d-1}{d_0}} \lesssim N_{q+1}\lambda_{q+1}^{(8d+24)\alpha+2d+1}.\notag
\end{align}
 Thus, by $ (8d+25)\alpha<\frac{1}{2}$ we obtain the following
 \begin{align*}
\|v_{q+1}\|_{C_{t,x}^1}\leq\|v_{l}\|_{C_{t,x}^1}+\|w_{q+1}\|_{C_{t,x}^1}\leq C_0^{1/d_0}\lambda_{q}^{2d+2}+\frac12\lambda_{q+1}^{2d+2}\leq C_0^{1/d_0}\lambda_{q+1}^{2d+2},
 \end{align*} where we chose $a$ large enough to absorb the universal constant.

\subsubsection{Proof of  \eqref{bd:6vq+1-vqldd} for $\rho^{i}_{q+1}-\rho_q^i$} 
Similarly as before, we first estimate the principal perturbations  $\theta^{(p,i)}_{q+1}$ in $C_tL^{d_0'}$. 
By the same argument as in \eqref{bd:4wq+1pl2}, we have 
\begin{align}
    \|\theta_{q+1}^{(p,i)}\|_{C_tL^{d_0'}}^{d_0'}&\lesssim \|M^{i}_l\|_{C_tL^1}+ C_0\delta_{q+1} 
    \lesssim  C_0\delta_{q+1}. \notag
\end{align}

For general $C_tL^m$-norm with $m\in[1,\infty]$, by the estimates for the building blocks in \eqref{int4theta}  and Lemma \ref{lem:6chi} we obtain
\begin{align}
\|\theta_{q+1}^{(p,i)} \|_{C_tL^{m}}&\lesssim\sum_{n\geq3}\sum_{\xi\in\Lambda^{n,i}}\norm{\chi(\zeta|M^{i}_l|-n)\(\frac{n}{\zeta}\)^{1/d_0'}\Gamma_{\xi}\(\frac{M^{i}_l}{|M^{i}_l|}\)}_{C_{t,x}^0}\|\Theta_{(\xi,n,i)} \|_{C_tL^{m}}\notag\\ 
&\lesssim l^{-2d-4}r_\perp^{\frac{d-1}{m}-\frac{d-1}{d_0'}} r_\parallel^{\frac{1}{m}-\frac{1}{d_0'}}.\label{bd:6theq+1plp}
\end{align}
Moreover, by   \eqref{bd:paralam6} we have
\begin{align}
\|\theta_{q+1}^{(c,i)} \|_{C_t}&\lesssim\|\theta_{q+1}^{(p,i)} \|_{C_tL^{d/\gamma}}
\lesssim l^{-2d-4}r_\perp^{\frac{d-1}{d/\gamma}-\frac{d-1}{d_0'}} r_\parallel^{\frac1{d/\gamma}-\frac{1}{d_0'}}\lesssim  \lambda_{q+1}^{(4d+8)\alpha-\frac1M}\lesssim4^{-i} \lambda_{q+1}^{-2\alpha},\label{bd:6theq+1clp}
\end{align}
where we used  conditions on the parameters to have $(4d+11)\alpha<\frac1{M},4^i\leq4^{N}\lesssim\lambda_{q+1}^{\alpha}$ and chose $a$ large enough to absorb the universal constant.
Then, combining the above estimates together, we obtain
\begin{align}
\|\rho^{i}_{q+1}-\rho^{i}_q\|_{C_tL^{d_0'}}
&\leq\|\theta^{i}_{q+1}\|_{C_tL^{d_0'}}+\|\rho^{i}_{l}-\rho^{i}_q\|_{C_tL^{d_0'}}\notag\\
&\leq\frac12 C_v C_0^{1/{d_0'}} \delta_{q+1}^{1/{d_0'}}+C\lambda_{q+1}^{-\alpha}+lC_0^{1/{d_0'}}\lambda_q^{d+4}\leq C_v C_0^{1/{d_0'}}\delta_{q+1}^{1/{d_0'}},\notag
\end{align}
which implies \eqref{bd:6vq+1-vqldd} for $\rho^{i}_{q+1}-\rho_q^i$. Here we chose $a$ large enough to absorb the universal constant.  

\subsubsection{Proof of   \eqref{bd:6rhoq+1-rhoql1}}
We first estimate $\theta_{q+1}^{(p,i)}$ in $W^{1,1+\epsilon}$-norm for some $\epsilon>0$ small enough. By the bounds for the building blocks in \eqref{int4theta}   and the bounds for the amplitude functions in Lemma \ref{lem:6chi} we have
\begin{align}
\|\theta_{q+1}^{(p,i)}\|_{C_tW^{1,1+\epsilon}}&\lesssim\sum_{n\geq3}\sum_{\xi\in\Lambda^{n,i}}\norm{\chi(\zeta|M^{i}_l|-n)\(\frac{n}{\zeta}\)^{1/d_0'}\Gamma_{\xi}\(\frac{M^{i}_l}{|M^{i}_l|}\)}_{C_{t,x}^1}\|\Theta_{(\xi,n,i)} \|_{C_tW^{1,1+\epsilon}} \notag\\ 
&\lesssim l^{-4d-12}\lambda_{q+1}r_\parallel^{\frac1{1+\epsilon}-\frac{1}{d_0'}}r_\perp^{\frac{d-1}{1+\epsilon}-\frac{d-1}{d_0'}}
\lesssim\lambda_{q+1}^{(8d+24)\alpha-\frac1M+d\epsilon}\lesssim \lambda_{q+1}^{-2\alpha},\label{bd:6thteaq+1pw11}
\end{align}
where we  chose $\epsilon>0$ small enough such that $d\epsilon<\alpha$. We also used  conditions on the parameters to have $(8d+27)\alpha<\frac1M$. 

By the fact that $\theta_{q+1}^{(p,i)}$ is non-negative, and the choice of parameters in \eqref{para62}, \eqref{bd:paralam6}  we have
\begin{align}
\|\rho^{i}_{q+1}-\rho^{i}_q\|_{C_tW^{1+\epsilon}}
&\lesssim \|\theta_{q+1}^{(p,i)}\|_{C_tW^{1,1+\epsilon}}+l\|\rho_q^i\|_{C_{t,x}^2}\lesssim4^{-i}\lambda_{q+1}^{-\alpha}\leq 4^{-i}\delta_{q+1}^{1/d_0'},\notag\\
    \inf_{t\in [0,1]}(\rho^{i}_{q+1} - \rho^{i}_q)&\geq-\|\theta_{q+1}^{(c,i)}\|_{C_{t,x}^0} -l\|\rho_q^i\|_{C_{t,x}^1}\geq -C\lambda_{q+1}^{-\alpha} \geq -\delta_{q+1}^{1/d_0'},
\end{align}
which yields \eqref{bd:6rhoq+1-rhoql1}.  Here we chose $a$ large enough to absorb the universal constant, including $4^i\leq 4^N$. 

\subsubsection{Proof of  \eqref{bd:6rhoqc1} for $\rho^{i}_{q+1}$}
  By \eqref{int4theta} and Lemma \ref{lem:6chi} we have for $j=1,2$
\begin{align*}
\|\theta_{q+1}^{(c,i)}\|_{C_{t}^j}\lesssim\|\theta_{q+1}^{(p,i)} \|_{C_{t,x}^j} 
&\lesssim\sum_{n\geq3}\sum_{\xi\in\Lambda^{n,i}}\norm{\chi(\zeta|M^{i}_l|-n)\(\frac{n}{\zeta}\)^{1/d_0'}\Gamma_{\xi}\(\frac{M^{i}_l}{|M^{i}_l|}\)}_{C_{t,x}^j}\|\Theta_{(\xi,n,i)} \|_{C_{t,x}^j} \notag\\ 
&\lesssim l^{-6d-20}(\lambda_{q+1}\mu)^j r_\parallel^{-\frac{1}{d_0'}}r_\perp^{-\frac{d-1}{d_0'}} \lesssim \lambda_{q+1}^{(12d+40)\alpha+d+j(d+1)}.\notag
\end{align*}
By choosing $(12d+42)\alpha<\frac{1}{2}$ we deduce 
\begin{align}
    \|\rho^{i}_{q+1} \|_{C_{t,x}^1}+\lambda_{q+1}^{-d-1}\|\rho^{i}_{q+1}\|_{C_{t,x}^2}
    \leq C_0^{1/d_0'}\lambda_{q}^{2d+2}+ \frac12\lambda_{q+1}^{2d+2}\leq C_0^{1/d_0'}\lambda_{q+1}^{2d+2},\notag
\end{align}
which implies  \eqref{bd:6rhoqc1} for $\rho^{i}_{q+1},1\leq i\leq N$.  

\subsubsection{Proof of  \eqref{bd:6rql1} }  We now estimate each term in the definition of $M_{q+1}^i$ separately.

\textbf{Oscillation error $M^{i}_{osc}$.} By Lemma \ref{bb_1}, the estimates for the amplitude functions and for the building blocks in  Lemma \ref{lem:6chi}, \eqref{int4theta}    we obtain
\begin{align}
    \|M^{i}_{osc,t}\|_{C_tL^1}
    &\lesssim \sum_{n\geq3}\sum_{\xi\in\Lambda^{n,i}}\norm{\chi(\zeta|M^{i}_l|-n)\(\frac{n}{\zeta}\)^{1/d_0'}\Gamma_{\xi}\(\frac{M^{i}_l}{|M^{i}_l|}\)}_{C_{t,x}^1} \|\Theta_{(\xi,n,i)}\|_{C_tL^1}\notag\\
    &\lesssim l^{-4d-12}r_\perp^{d-1-\frac{d-1}{d_0'}}r_\parallel^{1-\frac{1}{d_0'}} \lesssim \lambda_{q+1}^{(8d+24)\alpha-\frac 1{M}}\lesssim  \lambda_{q+1}^{-\alpha},\notag
\end{align}
where we used the choice of parameters in 
\eqref{para62},  \eqref{bd:paralam6} and  conditions for the parameters to have $(8d+25)\alpha<\frac 1{M}$.

We observe that $W_{(\xi,n,i)}\Theta_{(\xi,n,i)}$ is $(\mathbb{T}/r_\perp\lambda_{q+1})^d$-periodic. So by  Theorem \ref{bb1_1},  the estimates for the amplitude functions and for the building blocks  in Lemma \ref{lem:6chi}, \eqref{int4} and \eqref{int4theta} respectively    we have
\begin{align}
\|M^{i}_{osc,x} \|_{C_tL^1}
&\lesssim\sum_{n\geq3}\sum_{\xi\in\Lambda^{n,i}}\norm{\chi(\zeta|M^{i}_l|-n)\frac{n}{\zeta}\Gamma_{\xi}\(\frac{M^{i}_l}{|M^{i}_l|}\)}_{C_{t,x}^2} (r_\perp\lambda_{q+1})^{-1}\|W_{(\xi,n,i)}\Theta_{(\xi,n,i)}\|_{C_tL^{1}} \notag\\
&\lesssim l^{-6d-20}(r_\perp\lambda_{q+1})^{-1}\lesssim \lambda_{q+1}^{(12d+40)\alpha-\frac{1}{M}}
\lesssim  \lambda_{q+1}^{-\alpha}, \notag
\end{align}
where we used the choice of parameters in \eqref{para62}, \eqref{bd:paralam6} and  conditions on the parameters to have $(12d+41)\alpha<\frac{1}{M} $.

For the stress term $M^{i}_{osc,c}$, it holds
\begin{align}
     \left| M_{osc,c}^i\right|
     &\leq \frac{3}{\zeta}+\sum_{n\geq3}\chi(\zeta|M_l^i|-n)\left|\frac{n}{\zeta}- |M_l^i|\right|\leq \frac{3}{20}\delta_{q+2}+\frac{1}{20}\delta_{q+2}
     \leq \frac15C_0 \delta_{q+2}.\notag
\end{align}

In summary, we have
\begin{align}
    \|M^{i}_{osc}\|_{C_tL^1} \leq C \lambda_{q+1}^{-\alpha}+ \frac15C_0 \delta_{q+2}\leq \frac{2}{5}C_0 \delta_{q+2},\notag
\end{align}
where we chose $a$ large to absorb the universal constant.

\textbf{Nonlinear error $M_{nonlin}^i$.} 
 By \eqref{bd:6arhorho-},  
  \eqref{bd:6theq+1clp} and \eqref{bd:6thteaq+1pw11} we have 
\begin{align*}
    \|M^i_{nonlin}\|_{C_tL^1} 
     &\lesssim \| \theta_{q+1}^i \|_{C_tL^{d/\gamma}}\|\nabla\rho_{q+1}^i\|_{C_tL^{1+\epsilon}}+\| \rho_l^i\|_{C_tL^{d/\gamma}}\|\nabla \theta_{q+1}^i\|_{C_tL^{1+\epsilon}}\\
     &+\|  \nabla \theta_{q+1}^i\|_{C_tL^{1+\epsilon}}\|\rho_{q+1}^i\|_{C_tL^{d/\gamma}}+\|  \nabla\rho_l^i\|_{C_tL^{1+\epsilon}}\|\theta_{q+1}^i\|_{C_tL^{d/\gamma}}\\
     &\lesssim (\| \theta_{q+1}^i \|_{C_tL^{d/\gamma}}+\|  \nabla \theta_{q+1}^i\|_{C_tL^{1+\epsilon}})(\| \theta_{q+1}^i \|_{C_tL^{d/\gamma}}+\|  \nabla \theta_{q+1}^i\|_{C_tL^{1+\epsilon}}+\| \rho_q^i\|_{C_{t,x}^1}) \\
     &\lesssim C_0\lambda_{q+1}^{-2\alpha}(\lambda_{q+1}^{-\alpha}+\lambda_q^{2d+2})\lesssim    C_0 \lambda_{q+1}^{-\alpha}\leq \frac{1}{5}C_0 \delta_{q+2}.\notag
\end{align*}

\textbf{Linear error $M^{i}_{lin}$.} 
By the estimates  
in \eqref{bd:6rhoqc1}, \eqref{bd:6wq+1plp}, \eqref{bd:6wq+1clp},  \eqref{bd:6theq+1plp}  respectively we obtain
\begin{align}
 \|M^{i}_{lin}\|_{C_tL^1}&\lesssim \|v_l\|_{C_{t,x}^0}\|\theta_{q+1}^{i}\|_{C_tL^1}+\|\rho^{i}_l \|_{C_{t,x}^0}
\|w_{q+1}\|_{C_tL^1}+   \| \theta_{q+1}^{(p,i)}\|_{C_tL^{d_0'}}\|w_{q+1}^{(c,i)}\|_{C_tL^{d_0}}\notag\\
&\lesssim C_vC_0(\lambda_{q}^{2d+2}+1)\lambda_{q+1}^{(6d+16)\alpha-\frac1{N}}
\lesssim  C_0 \lambda_{q+1}^{-\alpha}\leq  \frac{1}{5}C_0\delta_{q+2},\notag
\end{align}
where we used  \eqref{para62} and   conditions on the parameters to have $(6d+18)\alpha<\frac{1}{N}$. 

\textbf{Commutator error $M_{com}^i$.} By  the bounds in   \eqref{bd:6rhoqc1}, \eqref{bd:6arhorho-} and \eqref{para62}   we obtain
 \begin{align}
  \|M_{ com}^i\|_{C_tL^1}
 & \lesssim l\|v_q\|_{C_{t,x}^1}\|\rho_q^i\|_{C_{t,x}^1}+l  \| \rho_q^i\|_{C_{t,x}^2}\| \rho_q^i\|_{C_{t,x}^1}\lesssim C_0 l\lambda_q^{5d+5}\lesssim C_0  \lambda_{q+1}^{-\alpha}\leq \frac{1}{5}C_0\delta_{q+2}.\label{bd:6mcom}
   \end{align}   
   
Summarizing all the estimates above  we obtain   \eqref{bd:6rql1}.

\section{Non-uniqueness of stationary solutions}\label{cogpss4}
In this section, we prove Theorem~\ref{thm:sta} by constructing non-unique stationary solutions. 
For any time-independent diffusion coefficient $\sigma$ satisfying Assumption~\ref{def:ass}, or the singular case, we apply the convex integration method to the stationary Fokker--Planck equation \eqref{eq:fpe:sta}. 
Compared with previous works, the main novelty of the argument lies in the treatment of the time-independent setting. Accordingly, we mainly emphasize the ideas and estimates specific to the stationary case. 

Let $N>0,0<\epsilon_0\leq \frac14$ and $1<d_0<d-1$ be given. We recall $d_0'=\frac{d_0}{d_0-1}$. The iteration is again indexed by a parameter
$
q\in\mathbb{N}_0.
$
We consider an increasing sequence
$
\{\lambda_q\}_{q\in\mathbb{N}_0}\subset\mathbb{N}
$
and a sequence
$
\{\delta_q\}_{q\in\mathbb{N}_0}\subset(0,1]
$
defined  by \begin{align*}
\lambda_q
=
a^{b^q},
\ 
q\geq0,
\ 
\delta_q
=
(\frac{\epsilon_0}{2})^{d_0+d_0'}
\lambda_1^{2\beta}
\lambda_q^{-2\beta},
\ 
q\geq1,
\ 
\delta_0=1.
\end{align*}  By imposing
\begin{align}
a^{2(b-1)\beta/(d_0+d_0')}>2,
\label{ieq:5ab2}
\end{align}
we obtain
\begin{align*}
\sum_{q\geq1}\delta_q^{1/(d_0+d_0')}
&\leq
\frac{\epsilon_0}{2} 
\sum_{q\geq1}
a^{(1-q)2(b-1)\beta/(d_0+d_0')}\leq
\frac{\epsilon_0}{2}
\cdot
\frac1{1-a^{-2(b-1)\beta/(d_0+d_0')}}
<
\epsilon_0.
\end{align*}

At each step $q$, we construct  
$
(v_q,\rho^i_q,M^i_q)_{1\leq i\leq N}
$
solving the stationary system
\begin{align}
-\div\div(A(\rho^i_q)\rho^i_q)
+
\div(v_q\rho^i_q)+\div((V*\rho_q^i)\rho_q^i)
&=
-\div M^i_q,
\qquad
\div v_q=0,
\label{eq:5qth}
\end{align}
where we recall the notation
$
A(\rho)=(\sigma\sigma^T)(\rho).
$

We initialize the iteration by defining
$
(v_0,\rho^i_0,M^i_0)
$
through
\begin{align*}
\rho^i_0(x)
=
1+\frac{\sin 2\pi x_1}{4^i},
\ 
v_0(x)=\overline v(x),
\ 
M^i_0(x)
=-(\overline v+V*\rho_0^i)\rho^i_0(x)+\div(
A(\rho^i_0)\rho^i_0)(x),
\end{align*}
where
$
x=(x_1,\dots,x_d).
$
Then, by Assumption~\ref{def:ass}, there exists a sufficiently large constant $C_0>0$ such that
\begin{align}
\|\rho^i_0\|_{C_x^2}
\leq
C_0^{1/d_0'},\ \|v_0\|_{C_{x}^1}\leq C_0^{1/d_0},\ 
\|M^i_0\|_{L^1}
\leq
(\frac{\epsilon_0}{2})^{d_0+d_0'}C_0.
\label{bd:5rho_0,M_0}
\end{align}
In the singular interaction  case, by \eqref{eq:hls} it holds that 
 \begin{align}
    \| M^i_0\|_{L^1}&\lesssim 1+\|(a*\nabla \rho_0^i)\rho_0^i\|_{L^1}+\|(a*
\rho^i_0)\nabla \rho^i_0\|_{L^1}+\|(b*\nabla
\rho^i_0)  \rho^i_0\|_{L^1}\notag\\
&\lesssim 1+\|a*\nabla \rho_0^i\|_{L^1}+\|a*
\rho^i_0\|_{L^1}+\|b*\nabla
\rho^i_0\|_{L^1} \lesssim 1+\|\rho^i_0\|_{C_x^1}\leq
(\frac{\epsilon_0}{2})^{d_0+d_0'}C_0.\label{bd:5rho_0,M_02}
 \end{align}

With these preparations, we are ready to state the iteration proposition, which forms the basis of the whole scheme.

\bp\label{prop:sta}
Under the assumption of Theorem \ref{thm:sta}, there exists  a choice of parameters $a,b,\beta$  such
that the following holds: Let $(v_q,  \rho^i_q,  M^i_q)_{1\leq i \leq N}$ be a solution to the system \eqref{eq:5qth} satisfying  $\int  \rho^i_q\dif x=1$,
\begin{align}
\|v_q\|_ {C_{x}^1}\leq C_0^{1/d_0}\lambda_q^{ d+1},\
\|  \rho^i_{q}\|_ {C_{x}^1}+\lambda_q^{-1}\|  \rho^i_{q}\|_ {C_{x}^2}\leq C_0^{1/d_0'} \lambda_q^{ d+1},
\label{bd:5rhoqc1}\\
 \|M^i_q\|_{L^1}\leq C_0 \delta_{q+1}.&\label{bd:5rql1}
\end{align}
Then there exists 
$(v_{q+1},  \rho^i_{q+1},  M^i_{q+1})_{1\leq i \leq N}$ which solves \eqref{eq:5qth} and satisfies \eqref{bd:5rhoqc1}-\eqref{bd:5rql1} at the level $q+1$ and
\begin{align}
\|v_{q+1}-v_{q}\|_{L^{d_0}}\leq
C_vC_0^{1/d_0}\delta_{q+1}^{1/d_0},\label{bd:5vq+1-vqldd}\ \
\|  \rho^i_{q+1}-  \rho^i_{q}\|_{L^{d_0'}}\leq 
C_v C_0^{1/{d_0'}}\delta_{q+1}^{1/{d_0'}},
\end{align}for some universal constant $C_v\geq1$. 
Moreover, we have for some $\epsilon>0$ small enough,
\begin{align}
     \|v_{q+1}-v_{q}\|_{L^{1}}&\leq \delta_{q+1}^{1/d_0},\label{bd:5vq+1-vql1}\\
 \|  \rho^i_{q+1}-  \rho^i_{q}\|_{W^{1,1+\epsilon}}\leq 
 4^{-i}\delta_{q+1}^{1/d_0'},\   \rho^i_{q+1} -   \rho^i_q&\geq - \delta_{q+1}^{1/d_0'}.\label{bd:5rhoq+1-rhoql1}
\end{align}
\ep

The proof of the main iteration, Proposition~\ref{prop:sta}, will be given in the next section. 
Assuming Proposition~\ref{prop:sta}, we now complete the proof of Theorem~\ref{thm:sta}.
 
 \begin{proof}[Proof of Theorem~\ref{thm:sta}]
We initialize the iteration from the triple
$
(v_0,\rho^i_0,M^i_0)
$
defined above.
By \eqref{bd:5rho_0,M_0} and \eqref{bd:5rho_0,M_02}, the required estimates hold at level $q=0$, since
$
\delta_1=(\frac{\epsilon_0}{2})^{d_0+d_0'}.
$ 
Next, applying Proposition~\ref{prop:sta} inductively, we construct
$
(v_q,\rho^i_q,M^i_q)
$
for every
$
q\geq1.
$
By \eqref{ieq:5ab2} together with
\eqref{bd:5vq+1-vqldd}--\eqref{bd:5rhoq+1-rhoql1}, the sequence
$
 (v_q,\rho^i_q) 
$
is Cauchy in
$
L^{d_0}\times (L^{d_0'}\cap W^{1,1+\epsilon})^N,
$
and we denote its limit by
$
(v,\rho^i).
$
Moreover, by \eqref{bd:5rql1}, it is straightforward to verify that
$
(\rho^i,v)
$
solves \eqref{eq:fpe:sta}.

Furthermore, by \eqref{ieq:5ab2},
\eqref{bd:5vq+1-vqldd}, and
\eqref{bd:5rhoq+1-rhoql1}, we obtain
\begin{align}
|\|\rho^i-1\|_{L^1}-\|\rho^i_0-1\|_{L^1}|
&\leq
\sum_{q=0}^\infty
\|\rho^i_{q+1}-\rho^i_q\|_{L^1}
\leq
 4^{-i}\sum_{q=0}^\infty
\delta_{q+1}^{1/d_0'}\leq 4^{-i-1}
,
\notag\\
\rho^i
&\geq
\rho^i_0
+
\sum_{q=0}^\infty
(\rho^i_{q+1}-\rho^i_q)
\geq
\frac34
-
\sum_{q=0}^\infty
\delta_{q+1}^{1/d_0'}
>0.
\notag
\end{align}
Hence, by the same calculation as before, $\rho^i$ is nonnegative, nontrivial, and $\rho^i$
 do not coincide with each other.
Since
$
\int_{\mathbb{T}^d}\rho_q^i\,\dif x=1,
$
it follows that $\rho^i$ is a probability density.

Moreover, by \eqref{bd:5vq+1-vql1}  we have $  \|v -\overline v\|_{L^{1}}\leq\sum_{q=0}  \|v_{q+1}-v_{q}\|_{L^{1}}\leq \sum_{q=0}\delta_{q+1}^{1/d_0}
 \leq \epsilon_0. $ 

By \eqref{bd:5rhoqc1}, \eqref{bd:5vq+1-vqldd}, and interpolation, we obtain
$
v\in L^{d_0(1+\epsilon)}
$
for some sufficiently small
$
\epsilon>0.
$
Since
$
\rho^i\in L^{d_0'},
$
it follows that
$
|v|^{1+\epsilon}\rho^i\in L^1.
$
Therefore, we apply the superposition principle
\cite[Section~7.2]{Tre14}
 and there exists a probability measure
$
\mathbf{Q}^i
$
on
$
C([0,\infty);\mathbb{T}^d)
$
which is a martingale solution associated with $L_{\rho^i}$ and satisfies
\begin{align*}
\dif \mathbf{Q}^i\circ\Pi_t^{-1}
=
\rho^i\,\dif x,
\ 
t\geq0.
\end{align*}
The resulting solution is stationary, since the density $\rho^i$ is independent of time.
\end{proof}

\section{Proof of Proposition \ref{prop:sta}}\label{sec:convexsta}
In this section, we apply the convex integration method to prove Proposition~\ref{prop:sta}. 
The overall strategy is similar to that in Section~\ref{sec:4defmq+1}. 
However, in the present stationary setting, we introduce new building blocks adapted to the time-independent framework, which enable us to implement the convex integration scheme and derive the desired estimates.

\subsection{Mollification}\label{sec:mpll5}
For given sufficiently small $\alpha\in(0,1)$ to be determined, we  take  $ {l}:=\lambda_{q+1}^{-3\alpha/{2}}\lambda_{q}^{-3d/2-2}$ and have 
 \begin{align}
     {l}^{-1}\leq\lambda_{q+1}^{2\alpha},\ \   l\lambda_q^{3d+4} \leq 
     \lambda_{q+1}^{-\alpha}\ll\delta_{q+2}.\label{5:para42}
 \end{align} 

Then,  in this section, we also replace $(v_{q},  \rho^i_{q},M^i_q)$ by a space-direction mollified field  
\begin{align}
v_l=v_{q}*\phi_l,\ \ 
   \rho^i_l=  \rho^i_{q}*\phi_l,\  M^i_l=M^i_{q}* \phi_l,\notag
\end{align}
where we recall that $\phi_l$ is a   standard radial mollifiers on $\mathbb{R}^{d}$.  
By calculation we obtain that 
\begin{align*}
-\div\div (A(\rho^i _l) \rho^i_{l})+\div(v_l   \rho^i_l)+\div((V*\rho_l^i)\rho_l^i)&=-\div (M^i_l+M^i_{com} ),\ \
\div v_l=0,
\end{align*}
where
\begin{align*}
M^i_{com} :&=-v_l   \rho^i_l+(v_{q}   \rho^i_{q})* \phi_l+\div( A(  \rho^i_{l}) \rho^i_{l})-( \div (A(  \rho^i_{q})  \rho^i_{q}))* \phi_l\\
&\ \ \ \   -(V*\rho_l^i)\rho_l^i+((V*\rho_q^i)\rho_q^i)*_x\phi_l.\notag
\end{align*}

By the mollification estimate \eqref{bd:5rql1}
 and \eqref{5:para42}  we have for $N\geq 0$,
\begin{align}
    \|M^i_l \|_{C_{x}^N}&\lesssim l^{-d-\frac13-N} \|M^i_{q } \|_{L^1}\lesssim C_0  l^{-d-\frac13-N}. \notag
 \end{align} 

 \subsection{Construction of the perturbations}\label{sec:5defq+1}
To construct the perturbations of $v_l$ and $\rho^i_l$, we introduce in this section a family of Mikado-type building blocks, which may be regarded as a generalization of those in Appendix~\ref{gij}.   We consider $2N$ disjoint sets $\Lambda^{1,n}, \Lambda^{2,n}$ defined in Appendix \ref{gij} by taking suitable rational rotations of one fixed set. We use the notation   $ \Lambda^{j,n} =  \Lambda^{1,n}$ for $j$ odd, and $ \Lambda^{j,n} =  \Lambda^{2,n}$ for $j$ even. 
For a fixed exponent
$
d_0>1,
$ and $\xi\in\Lambda:=\cup_{i=1}^N(\Lambda^{1,i}\cup \Lambda^{2,i})$
we   recall the notation
$
\Phi_{(\xi,d_0)},
\phi_{(\xi,d_0)},
\phi'_{(\xi,d_0')}
$
introduced in Appendix~\ref{gij}. 

\br
We recall that in \cite{Luo19}, the construction of stationary solutions requires the spatial dimension to be  larger than 3. 
In contrast, in the present work, we are able to employ more general $L^{d_0}$-based building blocks, which allows us to construct stationary solutions  also in dimension three.

\er

We  recall $\alpha_\xi$ from Lemma \ref{lem:4}, and  define 
 \begin{align*}
    W_{(\xi)}(x):=\xi\phi_{(\xi,d_0)}(x-\alpha_\xi),\ 
    \Theta_{ (\xi)}(x):=\phi'_{(\xi,d_0')}(x-\alpha_\xi),
\end{align*} 
 and 
\begin{align}
V_{(\xi)}:=\frac{1}{(n_*\lambda)^2}(\xi\otimes\nabla\Phi_{(\xi,d_0)}-\nabla\Phi_{(\xi,d_0)}\otimes\xi)(x-\alpha_\xi)\notag
\end{align}
with
\begin{align}
\div V_{(\xi)}
=W_{(\xi)} .\label{5:divOmega}
\end{align} 
In this case, by choosing $\lambda$ large enough, by Lemma \ref{lem:4} we have
\begin{align}
  W_{(\xi)}\Theta_{(\xi')}\equiv 0,\ {\rm for}\ \xi\neq\xi'\in\Lambda,\label{bd:wtheta0}\\
\int _{\mathbb{T}^d}W_{(\xi)}\Theta_{(\xi)}\dif x=\xi,\ \ \div (W_{(\xi)}\Theta_{(\xi)})=0.\notag
\end{align}

Finally, we obtain that for $N\geq 0$ and $p\in [1, \infty]$, it holds 
\begin{align}
\|\nabla^NW_{(\xi)}\|_{L^p }+\lambda\|\nabla^NV_{(\xi)}\|_{L^p }
\lesssim r_\perp^{\frac{d-1}{p}-\frac{d-1}{d_0}}\lambda^{n}, \ \ \ \|\nabla^N\Theta_{(\xi)}\|_{L^p }\lesssim r_\perp^{\frac{d-1}{p}-\frac{d-1}{d_0'}}\lambda^{n},& \label{5:int4theta}
\end{align}
where the implicit constants may depend on $p,N$, but are independent of $\lambda,r_\perp$.

With these building blocks in hand, then, we fix a  parameter $\zeta = 20\delta_{q+2}^{-1}$.  Let
$
\chi\in C_c^\infty(-\frac34,\frac34)
$
be a nonnegative function such that
$
\sum_{n\in\mathbb{Z}}\chi(t-n)=1
$
for every $t\in\mathbb{R}$.
Let
$
\tilde{\chi}\in C_c^\infty(-\frac45,\frac45)
$
be a nonnegative function satisfying
$
\tilde{\chi}=1
$
on
$
[-\frac34,\frac34]
$
and
$
\sum_{n\in\mathbb{Z}}\tilde{\chi}(t-n)\leq2.
$

Then we define the perturbations
\begin{align*}
w_{q+1}^{(p,i)}:&=\sum_{n\geq3}\sum_{\xi\in\Lambda^{n,i}}\tilde{\chi}(\zeta |M^i_l|-n)\(\frac{n}{\zeta }\)^{1/d_0}W_{ (\xi)},\\
w_{q+1}^{(c,i)}:&=\sum_{n\geq3}\sum_{\xi\in\Lambda^{n,i}}\nabla( \tilde{\chi}(\zeta |M^i_l|-n))\(\frac{n}{\zeta }\)^{1/{d_0}}:V_{ (\xi)}.
\end{align*}
Finally, the total perturbation and new iteration are defined by
\begin{align}
w_{q+1}:=\sum_{i=1}^N\(w_{q+1}^{(p,i)}+w_{q+1}^{(c,i)}\),\ \ v_{q+1}:=v_l+w_{q+1}.\notag
\end{align} 
Then  $w_{q+1},v_{q+1}$ are both mean-zero and divergence-free functions.

Now,  we define the perturbations for the stationary  densities functions as
\begin{align}
\theta_{q+1}^{(p,i)}:&=
\sum_{n\geq3}\sum_{\xi\in\Lambda^{n,i}}\chi(\zeta |M^i_l|-n)\(\frac{n}{\zeta }\)^{1/{d_0'}}\Gamma_{\xi}\(\frac{M^i_l}{|M^i_l|}\)\Theta_{ (\xi)},\notag\\
\theta_{q+1}^{(c,i)}:&=-\int_{\mathbb{T}^d} \theta_{q+1}^{(p,i)}\dif x.\notag
\end{align}

Then, by \eqref{bd:wtheta0} and  a similar calculation as before, we obtain 
\begin{align}
 \div (w_{q+1}^{(p,i)}\theta_{q+1}^{(p,i)}-M^i_l)
&=\sum_{n\geq3}\sum_{\xi\in\Lambda^{n,i}}\nabla\[\chi(\zeta |M^i_l|-n)\frac{n}{\zeta }\Gamma_{\xi}\(\frac{M^i_l}{|M^i_l|}\)\]\mP_{\neq0}(W_{ (\xi)}\Theta_{ (\xi)})\notag\\
 &\ \ +\div \(\sum_{n\geq3}\chi(\zeta |M^i_l|-n)\frac{n}{\zeta }\frac{M^i_l}{|M^i_l|}-M^i_l\). \label{wthetagood}
\end{align} 

Finally, we define
\begin{align}
\theta _{q+1}^i:=\theta_{q+1}^{(p,i)}+\theta_{q+1}^{(c,i)},
\ \ 
  \rho^i_{q+1}:=  \rho^i_l+\theta^i _{q+1}.\notag
\end{align}
  By  definition, it  holds that $\int  \rho^i_{q+1}\dif x=1$.

Moreover, by Lemma \ref{lem:4}, we obtain for
$
1\leq i\leq N 
$
that
\begin{align}
\div\bigl(
w_{q+1}\theta_{q+1}^i
\bigr)
=
\div\(
(w_{q+1}^{(p,i)}+w_{q+1}^{(c,i)})
\theta_{q+1}^{(p,i)}
\).
\label{bd:5wq+1rhoq+1i}
\end{align}

\subsection{Construction of the stress term $M^i_{q+1}$}\label{sec:5defmq+1}
From  \eqref{bd:5wq+1rhoq+1i}  and the definition of the perturbations 
we obtain

\begin{align*}
 -\div M^i_{q+1}&=\div (w_{q+1}^{(p,i)}\theta_{q+1}^{(p,i)}-M^i_l)(:=\div M^i_{osc})\\
 &\quad-\div[ \div( A(\rho^{i}_{q+1})\rho^{i}_{q+1}-A(\rho^{i}_{l}) \rho^{i}_{l})-(V*\rho_{q+1}^i)\rho_{q+1}^i+(V*\rho_{l}^i)\rho_{l}^i] (:=\div M^i_{nonlin})\notag\\ 
 &\quad+\div(v_l\theta _{q+1}^i+w_{q+1}  \rho^i_l+w_{q+1}^{(c,i)}\theta_{q+1}^{(p,i)})(:=\div M^i_{lin})\\
 &\quad -\div M^i_{com},
\end{align*}
where we  define  the  nonlinear error $ M^i_{nonlin}$ and linear error $M^i_{lin}$ respectively by 
\begin{align*}
M^i_{nonlin}:&=-\div(  A(\rho^i_{q+1} )   \rho^i_{q+1}- A(\rho^i_{l} ) \rho^i_{l}) +(V*\rho_{q+1}^i)\rho_{q+1}^i-(V*\rho_{l}^i)\rho_{l}^i,\\
M^i_{lin}:&=v_l\theta^i _{q+1}+w_{q+1}  \rho^i_l+w_{q+1}^{(c,i)}\theta_{q+1}^{(p,i)}.
\end{align*}  
Using \eqref{wthetagood} and the inverse divergence operator  $\mathcal{B}_1$, we  define  the  oscillation error $M^i_{osc}:=M^i_{osc,x}+M^i_{osc,c}$  by
\begin{align*} M^i_{osc,x}:&=\sum_{n\geq3}\sum_{\xi\in\Lambda^{n,i}}\mathcal{B}_1\(\nabla\[\chi(\zeta |M^i_l|-n)\frac{n}{\zeta }\Gamma_{\xi}\(\frac{M^i_l}{|M^i_l|}\)\],\mathbb{P}_{\neq0}(W_{ (\xi)}\Theta_{ (\xi)})\),\\
   M^i_{osc,c}:&=\sum_{n\geq3}\chi(\zeta |M^i_l|-n)\frac{n}{\zeta }\frac{M^i_l}{|M^i_l|}- M^i_l.
\end{align*}
 
Then we define 
$$-{M}^i _{q+1}:=M^i_{osc}+M^i_{nonlin}+M^i_{lin}-M^i_{com} .$$

\subsection{ Proof of Proposition \ref{prop:sta}}\label{sec:proof1}

\subsubsection{Choice of parameters}\label{sec:5choicepara}

Regarding the parameters of the building blocks, we define 
\begin{align}
  \lambda=\lambda_{q+1},\ r_\perp=\lambda_{q+1}^{-1+\frac1{M}}, \notag
\end{align}
where   we introduce a integer $M>0$ large enough  satisfying $0< \frac1M(\frac{d-1}{d_0}+1)<\frac{d-1}{d_0}-1, (1-\frac{1}{M})\frac{d-1}{d_0'}>\frac{1}{M}$ and $(1-\frac{1}{M})(d-1)(\frac{\gamma}{d}-\frac{1}{d_0'})>\frac{1}{M}$.   Then, it holds that 
\begin{align}
    \lambda_{q+1}r_\perp^{d-1-\frac{d-1}{d_0'}}\lesssim \lambda_{q+1}^{-\frac{1}{M}},\  \ r_\perp^{d-1-\frac{d-1}{d_0}}\lesssim\lambda_{q+1}^{-\frac{1}{M}},\ \   r_\perp^{\frac{d-1}{d/\gamma}-\frac{d-1}{d_0'}}\lesssim\lambda_{q+1}^{-\frac1M}.\label{para5lam}
\end{align}

In the sequel, we also need \eqref{5:para42} and \begin{align}
   \lambda_{q}^{ d+1}\leq\lambda_{q+1}^\alpha,\ \ (12d+41)\alpha<\frac1M. \label{para55}
 \end{align} 

 The above conditions can be achieved as follows. 
First, choose $\alpha>0$ sufficiently small such that $(12d+41)\alpha<\frac1M.$ Next, choose $b\in M\mathbb{N}$ sufficiently large so that
$b>\frac{3d+4}{\alpha}.$ We then select
$\beta>0$ sufficiently small such that
$\alpha>2\beta b.$ Finally, we choose
$a$ sufficiently large in order to absorb all implicit and universal constants arising in the subsequent estimates and to guarantee the validity of \eqref{ieq:5ab2}.

In the end, we establish the estimate of the amplitude functions  as in Lemma \ref{lem:4chi}.
\bl\label{lem:5chi} 
 For $ M,K\in\mathbb{N}_0$ and $1\leq i\leq N$, we have
\begin{align*}
    \sum_{n\geq3} \(\frac{n}{\zeta }\)^M\(\|\chi(\zeta |M^i_l|-n)\|_{C_{x}^K}+ \|\tilde{\chi}(\zeta |M^i_l|-n)\|_{C_{x}^K}\)&\lesssim l^{- (d+4)K-(d+2)(M+1)},\\
     \sum_{n\geq3}\sum_{\xi\in\Lambda^{n,i}}\(\frac{n}{\zeta }\)^M\norm{\chi(\zeta |M^i_l|-n)\Gamma_\xi\(\frac{M^i_l}{|M^i_l|}\)}_{C_{x}^K}& \lesssim l^{-(2d+8)K-(d+2)(M+1)}.
\end{align*}
\el 

\subsubsection{Proof of  \eqref{bd:5vq+1-vqldd}  for $v_{q+1}-v_q$.}
We first estimate the principal perturbations ${w}^{(p,i)}_{q+1}$. By Cauchy's inequality  we have
\begin{align}
|w_{q+1}^{(p,i)}|^{d_0}  
&\lesssim \sum_{n\geq3}\tilde{\chi}(\zeta |M^i_l|-n)\frac{n}{\zeta }\sum_{\xi\in\Lambda^{n,i}}\left|W_{ (\xi)}\right|^{d_0}.\notag
\end{align} 
By applying the generalized H\"older inequality of 
Theorem \ref{ihiot} in spatial direction, together with the estimates for the building blocks in \eqref{5:int4theta} and Lemma \ref{lem:5chi} we deduce 
\begin{align}
\|w_{q+1}^{(p,i)}\|^{d_0}_{L^{d_0}} 
&\lesssim\sum_{n\geq3}\norm{\tilde{\chi}(\zeta |M^i_l |-n)\frac{n}{\zeta }}_{L^1}\sum_{\xi\in\Lambda^{n,i}}\|W_{ (\xi)}\|_{L^{d_0}}^{d_0}\notag\\ 
&\quad+(r_\perp\lambda_{q+1})^{-1}\norm{\tilde{\chi}(\zeta |M^i_l |-n)\frac{n}{\zeta }}_ {C_{x}^1}\sum_{\xi\in\Lambda^{n,i}}\|W_{ (\xi)}\|_{L^{d_0}}^{d_0} \notag\\
&\lesssim\norm{ M^i_l +\zeta ^{-1}}_{L^1}+l^{-3d-8}\lambda_{q+1}^{-\frac{1}{M}} \lesssim \|M^i_l \|_{L^1}+ \delta_{q+1}\lesssim C_0 \delta_{q+1}
,\notag
\end{align}
where we used  the condition on the parameters to have $(6d+16)\alpha-\frac{1}{M}<-\alpha<-2\beta$. 

In the next step, we estimate in  general $L^m$-norm with $m\in[1,\infty]$. By  \eqref{5:int4theta}   and   Lemma \ref{lem:5chi} we obtain
\begin{align}
\|w_{q+1}^{(p,i)}\|_{L^m}&\lesssim\sum_{n\geq3}\sum_{\xi\in\Lambda^{n,i}}\norm{\tilde{\chi}(\zeta |M^i_l|-n)\(\frac{n}{\zeta }\)^{1/d_0}}_ {C_{x}^0}\|W_{ (\xi)}\|_{L^m}\lesssim l^{-2d-4}r_\perp^{\frac{d-1}{m}-\frac{d-1}{d_0}} ,\label{bd:5wq+1plp}\\
\|w_{q+1}^{(c,i)}\|_{L^m}&\lesssim\sum_{n\geq3}\sum_{\xi\in\Lambda^{n,i}}\norm{\tilde{\chi}(\zeta |M^i_l|-n)\(\frac{n}{\zeta }\)^{1/d_0}}_{C^1_{x}}\|V_{ (\xi)}\|_{L^m}\lesssim l^{-3d-8}r_\perp^{\frac{d-1}{m}-\frac{d-1}{d_0}}\lambda_{q+1}^{-1}.\label{bd:5wq+1clp}
\end{align}
 We then have
\begin{align}
\|w_{q+1}\|_{L^{d_0}}
&\lesssim NC_0^{1/d_0} \delta_{q+1}^{1/d_0}+N\lambda_{q+1}^{(6d+16)\alpha-1} \leq \frac{C_v}2C_0^{1/d_0}\delta_{q+1}^{1/d_0},\label{bd:5wq+1l2}\end{align}
where we used  conditions on the parameters to have
$(6d+16)\alpha-1<-\alpha<-2\beta$. 
The above inequality together with  \eqref{bd:5rhoqc1} and \eqref{5:para42} implies \eqref{bd:5vq+1-vqldd} for $v_{q+1}$:
\begin{align*}
\|v_{q+1}-v_{q}\|_{L^{d_0}}&\leq \|w_{q+1}\|_{L^{d_0}}+l\|v_{q}\|_ {C_{x}^1}
\leq \frac{ {C}_v}2C_0^{1/d_0}\delta_{q+1}^{1/d_0}+lC_0^{1/d_0}\lambda_{q}^{ d+1}\leq C_vC_0^{1/d_0}\delta_{q+1}^{1/d_0}.
\end{align*}

\subsubsection{Proof of  \eqref{bd:5rhoqc1} for $v_{q+1}$.}
 By the estimates for the building blocks in \eqref{5:int4theta} and   Lemma \ref{lem:5chi} we have 
\begin{align}
\|w_{q+1}\|_{C_{x}^1}&\lesssim \sum_{i=1}^N\sum_{n\geq3}\sum_{\xi\in\Lambda^{n,i}}\norm{\tilde{\chi}(\zeta |M^i_l|-n)\(\frac{n}{\zeta }\)^{1/d_0}}_{C_{x}^{2}}\|  V_{ (\xi)}\|_{C_{x}^2}\notag\\
&\lesssim  Nl^{-4d-12}\lambda_{q+1} r_\perp^{-\frac{d-1}{d_0}}\lesssim \lambda_{q+1}^{(8d+24)\alpha+d }.\notag
\end{align}
 Thus, by $(8d+24)\alpha<1$ we obtain
 \begin{align*}
\|v_{q+1}\|_ {C_{x}^1}\leq\|v_{l}\|_ {C_{x}^1}+\|w_{q+1}\|_ {C_{x}^1}\leq C_0^{1/d_0}\lambda_{q}^{d+1}+ \lambda_{q+1}^{d+1}\leq C_0^{1/d_0}\lambda_{q+1}^{ d+1},
 \end{align*} where we chose $a$ large enough to absorb the universal constant. 

\subsubsection{Proof of \eqref{bd:5vq+1-vql1}}By \eqref{para5lam}, \eqref{bd:5wq+1plp} and \eqref{bd:5wq+1clp}, it holds that
\begin{align*}
\|w_{q+1}\|_{L^1}&\lesssim  Nl^{-2d-4}r_\perp^{{d-1}-\frac{d-1}{d_0}} \lesssim C_0\lambda_{q+1}^{(4d+8)\alpha-\frac{1}{M}},
\end{align*} 
which together with   \eqref{ieq:5ab2}, \eqref{bd:5rhoqc1} implies 
\begin{align*}
     \|v_{q+1}-v_{q}\|_{L^{1}}\lesssim
 C_0\lambda_{q+1}^{(4d+8)\alpha-\frac{1}{M}}+lC_0\lambda_q^{d+1}\lesssim \lambda_{q+1}^{-\alpha}\leq  \delta_{q+1}^{1/d_0},
\end{align*}
where we used the   conditions on the parameters to have $(4d+9)\alpha<\frac{1}{M} $. We also chose $a$ large to absorb the universal constant.

\subsubsection{Proof of  \eqref{bd:5vq+1-vqldd} for $  \rho^i_{q+1}-\rho^i_q$}\label{sec:5esttheq+1}
We first estimate the principal perturbations  $\theta^{(p,i)}_{q+1}$. By the same argument as in \eqref{bd:4wq+1pl2}, 
\begin{align}
\|\theta_{q+1}^{(p,i)}\|^{d_0'}_{L^{d_0'}} 
&\lesssim \|M^i_l \|_{L^1}+ \delta_{q+1}\lesssim C_0 \delta_{q+1}
.\notag
\end{align}

Then we estimate in  general $L^m$-norm with $m\in[1,\infty]$, by the estimates for the building blocks in \eqref{5:int4theta}  and Lemma \ref{lem:5chi} we obtain
\begin{align}
\|\theta_{q+1}^{(p,i)} \|_{L^{m}}&\lesssim\sum_{n\geq3}\sum_{\xi\in\Lambda^{n,i}}\norm{\chi(\zeta |M^i_l|-n)\(\frac{n}{\zeta }\)^{1/d_0'}\Gamma_{\xi}\(\frac{M^i_l}{|M^i_l|}\)}_{C_{x}^0}\|\Theta_{ (\xi)} \|_{L^{m}}\lesssim l^{-2d-4}r_\perp^{\frac{d-1}{m}-\frac{d-1}{d_0'}} ,\label{bd:5theq+1plp}
\end{align}
which implies that  by \eqref{para5lam}
\begin{align}
|\theta_{q+1}^{(c,i)} |&\lesssim\|\theta_{q+1}^{(p,i)}\|_{L^{d/\gamma}}\lesssim l^{-2d-4}r_\perp^{\frac{d-1}{d/\gamma}-\frac{d-1}{d_0'}}\lesssim\lambda_{q+1}^{(4d+8)\alpha-\frac1M}\lesssim \lambda_{q+1}^{-2\alpha}, \label{bd:5theq+1clp} 
\end{align}
where we used  conditions on the parameters to have $(4d+10)\alpha<\frac1{M}$ and chose $a$ large enough to absorb the universal constant.
Then together with \eqref{5:para42}    we imply 
\begin{align}
\|  \rho^i_{q+1}-  \rho^i_q\|_{  L^{d_0'}}
&\leq\|\theta^i _{q+1}\|_{  L^{d_0'}}+l\|  \rho^i_{q}\|_{ C_{x}^1} \lesssim C_0^{1/{d_0'}}(\delta_{q+1}^{1/{d_0'}}+\lambda_{q+1}^{-\alpha})\leq C_v C_0^{1/{d_0'}}\delta_{q+1}^{1/{d_0'}}.\notag
\end{align}

\subsubsection{Proof of   \eqref{bd:5rhoq+1-rhoql1}} 
 We first estimate $\theta_{q+1}^{(p,i)}$ in $W^{1,1+\epsilon}$-norm for some $\epsilon>0$. By  \eqref{5:int4theta}  and  Lemma \ref{lem:5chi} we have
\begin{align}
\|\theta_{q+1}^{(p,i)}\|_{W^{1,1+\epsilon}}&\lesssim\sum_{n\geq3}\sum_{\xi\in\Lambda^{n,i}}\norm{\chi(\zeta |M^i_l|-n)\(\frac{n}{\zeta }\)^{1/d_0'}\Gamma_{\xi}\(\frac{M^i_l}{|M^i_l|}\)}_ {C_{x}^1}\|\Theta_{ (\xi)} \|_{W^{1,1+\epsilon}}\notag\\ 
&\lesssim l^{-4d-12}\lambda_{q+1}r_\perp^{\frac{d-1}{1+\epsilon}-\frac{d-1}{d_0'}}
\lesssim\lambda_{q+1}^{(8d+24)\alpha-\frac1M+d\epsilon}\lesssim  \lambda_{q+1}^{-2\alpha},\label{bd:5thetaq+1w11}
\end{align}
where we chose $\epsilon>0$ small enough to ensure $d\epsilon<\alpha$, and used \eqref{para5lam} and  \eqref{para55}. We also used  conditions on the parameters to have $(8d+27)\alpha<\frac1M$. 

Since $\theta_{q+1}^{(p,i)}$ is non-negative,  by the choice of parameters in \eqref{5:para42}  and \eqref{bd:5theq+1clp}  we have
\begin{align}\|\rho^{i}_{q+1}-\rho^{i}_q\|_{W^{1,1+\epsilon}} 
&\leq\|\theta^{i}_{q+1}\|_{W^{1,1+\epsilon}} +l\|\rho_q^i\|_{C_{x}^2} \lesssim  \lambda_{q+1}^{-\alpha}+l\lambda_q^{d+1}\leq 4^{-i}\delta_{q+1}^{1/d_0'},\notag\\
     \rho^i_{q+1} -   \rho^i_q&\geq-|\theta_{q+1}^{(c,i)}|-l\|  \rho^i_{q}\|_{ C_{x}^1} \geq -C\lambda_{q+1}^{-\alpha}\geq -\delta_{q+1}^{1/d_0'},\notag
\end{align}
which yields \eqref{bd:5rhoq+1-rhoql1}.  Here we chose $a$ large enough to absorb the universal constant.

\subsubsection{Proof of  \eqref{bd:5rhoqc1} for $  \rho^i_{q+1}$}
  By \eqref{5:int4theta}  and Lemma \ref{lem:5chi} we have for $j=1,2$
\begin{align*}
\|\theta_{q+1}^{(p,i)} \|_ {C_{x}^j} 
&\lesssim\sum_{n\geq3}\sum_{\xi\in\Lambda^{n,i}}\norm{\chi(\zeta |M^i_l|-n)\(\frac{n}{\zeta }\)^{1/d_0'}\Gamma_{\xi}\(\frac{M^i_l}{|M^i_l|}\)}_ {C_{x}^j}\|\Theta_{ (\xi)} \|_ {C_{x}^j}\notag\\ 
&\lesssim l^{-6d-20}\lambda_{q+1}^j r_\perp^{-\frac{d-1}{d_0'}}\lesssim
\lambda_{q+1}^{(12d+40)\alpha+d+j-1}.\notag
\end{align*}
By choosing $(12d+40)\alpha<\frac{1}{2}$  we deduce 
\begin{align}
    \|  \rho^i_{q+1}\|_ {C_{x}^1}+\lambda_{q+1}^{-1}\|  \rho^i_{q+1}\|_ {C_{x}^2}
    \leq C_0^{1/d_0'}\lambda_{q}^{ d+1}+  \lambda_{q+1}^{ d+1}\leq C_0^{1/d_0'} \lambda_{q+1}^{ d+1},\notag
\end{align}
which implies  \eqref{bd:5rhoqc1}. Here we chose $a$ large enough to absorb the universal constant.

\subsubsection{Proof of \eqref{bd:5rql1}}\label{sec:5estmq+1} We estimate each terms in the definition of $M^i_{q+1}$ separately.

\textbf{Oscillation error $M^i_{osc}$.} 
We observe that $W_{ (\xi)}\Theta_{ (\xi)}$ is $(\mathbb{T}/r_\perp\lambda_{q+1})^d$-periodic, so by  Theorem \ref{bb1_1},  Lemma \ref{lem:5chi} and \eqref{5:int4theta}   we have
\begin{align}
\|M^i_{osc,x} \|_{L^1}
&\lesssim\sum_{n\geq3}\sum_{\xi\in\Lambda^{n,i}}\norm{\chi(\zeta |M^i_l|-n)\frac{n}{\zeta }\Gamma_{\xi}\(\frac{M^i_l}{|M^i_l|}\)}_{C_{x}^2}(r_\perp\lambda_{q+1})^{-1}\|W_{ (\xi)}\Theta_{ (\xi)}\|_{L^{1}}\notag\\
&\lesssim l^{-6d-20}(r_\perp\lambda_{q+1})^{-1}\lesssim \lambda_{q+1}^{(12d+40)\alpha-\frac{1}{M}}
\lesssim  \lambda_{q+1}^{-\alpha}\leq  \frac15C_0\delta_{q+2}, \notag
\end{align}
where we used the  conditions on the parameters to have $(12d+41)\alpha<\frac{1}{M}$.

The estimate on  the stress term $M^i_{osc,c}$ is directly from Section \ref{sec:5estmq+1}:
\begin{align}
     \left| M^i_{osc,c} \right|
     &\leq \frac{3}{\zeta }+\sum_{n\geq3}\chi(\zeta |M^i_l |-n)\left|\frac{n}{\zeta }- |M^i_l |\right|\leq  \frac15C_0\delta_{q+2}.\notag
\end{align}

\textbf{Nonlinear error $M^i_{nonlin} $.} Under Assumption \ref{def:ass}, by the calculation in \eqref{arhorho-arhorho},   the bounds in   \eqref{bd:5rhoqc1}, \eqref{5:para42} and \eqref{bd:5thetaq+1w11}  we have
\begin{align}
    \|M^i_{nonlin}\|_{L^1} 
      &\lesssim \|\theta^{i}_{q+1}\|_{W^{1,1+\epsilon}} (1+\|\rho^{i}_{q}\|_{C_{x}^1}+\| \theta_{q+1}^i\|_{W^{1,1+\epsilon}} ) \notag\\
      &\lesssim (\lambda_{q}^{ d+1}+1)\lambda_{q+1}^{-2\alpha} \lesssim    C_0 \lambda_{q+1}^{-\alpha}\leq \frac{1}{5}C_0 \delta_{q+2}.\notag
\end{align} 

For the singular interaction case, by \eqref{bd:6arhorho-},  we have
\begin{align*}
    \|M^i_{nonlin}\|_{L^1} 
     &\lesssim (\| \theta_{q+1}^i \|_{L^{d/\gamma}}+\|  \nabla \theta_{q+1}^i\|_{L^{1+\epsilon}})(\| \theta_{q+1}^i \|_{L^{d/\gamma}}+\|  \nabla \theta_{q+1}^i\|_{L^{1+\epsilon}}+\| \rho_q^i\|_{C_{x}^1}) \\
     &\lesssim C_0\lambda_{q+1}^{-2\alpha}(1+\lambda_q^{d+1})\lesssim    C_0 \lambda_{q+1}^{-\alpha}\leq \frac{1}{5}C_0 \delta_{q+2}.\notag
\end{align*}

\textbf{Linear error $M^i_{lin}$.} 
Together with  the estimates  
in \eqref{bd:5rhoqc1}, \eqref{5:para42},  \eqref{bd:5wq+1plp},  \eqref{bd:5wq+1clp} and   \eqref{bd:5theq+1clp}     we obtain
\begin{align}
 \|M^i_{lin}\|_{L^1}&\lesssim \|v_l\|_ {C_{x}^0}\|\theta^i_{q+1} \|_{L^1}+\|  \rho^i_l\|_ {C_{x}^0}
\|w_{q+1}\|_{L^1}+   \| \theta_{q+1}^{(p,i)}\|_{L^{d_0'}}\|w_{q+1}^{(c,i)}\|_{L^{d_0}}\notag\\
&\lesssim C_vC_0(\lambda_{q}^{ d+1}+1)\lambda_{q+1}^{(4d+8)\alpha-\frac1{M}}
\lesssim   C_0 \lambda_{q+1}^{-\alpha}\leq  \frac{1}{5}C_0 \delta_{q+2},\notag
\end{align}
where we used the   conditions on the parameters to have $(4d+10)\alpha<\frac{1}{M}$. 

\textbf{Commutator error $M^i_{com} $.} By the calculation as in \eqref{bd:4mcom} or \eqref{bd:6mcom}, the bounds in   \eqref{bd:5rhoqc1}, \eqref{5:para42} and Lemma \ref{lem:A}  we obtain
 \begin{align}
  \|M^i_{ com} \|_{L^1}  & \lesssim l\|v_q\|_{C_{x}^1}\|\rho_q^i\|_{C_{x}^1}+l (1+  \| \rho_q^i\|_{C_{x}^2}+\| \rho_q^i\|_{C_{x}^1}^2)\| \rho_q^i\|_{C_{x}^1}\notag\\
 &\lesssim C_0 l\lambda_q^{3d+4}\lesssim  C_0  \lambda_{q+1}^{-\alpha}\leq \frac{1}{5}C_0 \delta_{q+2},\notag 
   \end{align}  
 where we also chose $a$ large to absorb the universal constant.
 
Summarizing all the estimates above  we obtain   \eqref{bd:5rql1} at the level $q+1$.

\section{Application to the whole space case}\label{R:cogpss4}
In this section, we consider the whole-space setting and prove Theorem~\ref{R:thm:sta} and Theorem \ref{R:thm:4converge}. The non-uniqueness solutions to linear Fokker--Planck equations on $\mathbb{R}^d$ was established in \cite{LR25}. A key ingredient in the whole-space construction is the Bogovskii operator, which replaces the inverse divergence operator used in the periodic setting.
As a  inverse of the divergence operator, the Bogovskii operator preserves compact support: when applied to a compactly supported function, the resulting vector field remains compactly supported. This property allows us to localize the entire convex integration construction to a fixed domain, say $[-\frac12,\frac12]^d.$ 

The overall construction and the corresponding estimates are largely analogous to those developed in the torus setting. For this reason, we omit many repetitive computations and focus only on the modifications required in the whole-space case.
Moreover, the solutions constructed here may be viewed as small perturbations of a stationary solution $\rho^{ st}$, which  plays a role analogous to that of the uniform distribution in the periodic setting.

 \subsection{Bogovskii operator}\label{R:tamr}
 Let $K\subset \mathbb{R}^d,d\geq2$ be a bounded and star shaped domain with respect to a ball $K'\subset\subset K$.
We choose a non-negative $\omega \in C_c^\infty(K')$ with $\int_{K'} \omega\dif x = 1$ and define for
$g\in C_c^\infty(K)$
\begin{equation}
    \mathcal{B}_\omega g(x) := \int_{K} g(y)\,\frac{x - y}{\lvert x - y \rvert^{n}}
    \int_{0}^{\infty} \omega\!\left(x + r\,\frac{x - y}{\lvert x - y \rvert}\right)
      (|x-y|+r)^{n-1}\,\dif r\,\dif y.\notag
\end{equation}
By rewriting 
\begin{equation}
    \mathcal{B}_\omega g(x) :=\int_{K} g(y)(x - y) 
    \int_{1}^{\infty} \omega\!\left(y + r(x - y) \right)
      r^{n-1}\,\dif r\,\dif y,\notag
\end{equation}
  the following Lemma is well-known (see c.f. \cite[Theorem 1]{Bog79}).
\bp  Let $k\geq0,1 <p< \infty$. Let $f \in W_0^{k,p}(K)$. Then it holds that $ \mathcal{B}_\omega( C_c^\infty(K)) \subset C_c^\infty(K)^d,$ and 
\begin{align*}
    \div \mathcal{B}_\omega f=f-\omega\int_{K}f\dif x.
\end{align*}
Moreover, we have
\begin{align*}
    \|\mathcal{B}_\omega f\|_{W^{k+1,p}(K)}\lesssim_{\omega} \|f\|_{W^{k,p}(K)}.
\end{align*}

\ep
Then, we introduce the the bilinear version $\mathcal{R}_\omega: C^\infty_c(K; \mathbb{R}) \times C_0^\infty(\mathbb{T}^d; \mathbb{R})\to C_c^\infty(K; \mathbb{R}^{d})$  by $$\mathcal{R}_\omega(v,f):=v\nabla\Delta^{-1}f-\mathcal{B}_\omega\(\nabla v\cdot\nabla\Delta^{-1}f\).$$
Here $C^\infty_0(\mathbb{T}^d;\mathbb{R})$ denotes the space of smooth $\mathbb{T}^d$-periodic functions on $\mathbb{R}^d$ with zero-mean.

\bp \label{R:bb1_1}
Let $k\geq 0,1 < p <\infty$. For any $v \in C^\infty_c(K; \mathbb{R})$ and  $f\in C_0^\infty(\mathbb{T}^d; \mathbb{R})$, we have $$\div(\mathcal{R}_\omega(v,f))=vf-\omega\int_{K}vf\dif x,$$ and for  $\sigma\in\mathbb{N}$,
$$\|\mathcal{R}_\omega(v,f(\sigma\cdot))\|_{W^{k,p}(K)}\lesssim\sigma^{k-1}\|v\|_{C^{k+1}(K)}\|f\|_{W^{k,p}(\mT^d)}.$$
\ep

\subsection{Some estimates on $\rho^{st}$}
We collect below several estimates that will be used throughout the remainder of this section.

\bl
The stationary solution satisfies that 
\begin{align}
    \|\rho^{st}\|_{W^{2,\infty}}<\infty.\label{R:rhostc2}
\end{align}
\el
\begin{proof}
    By rewriting $ g(\rho^{st}(x))=-\Phi(x)+\mu^{st} $, taking derivative and using the monotony of $g^{-1}$, it holds that
    \begin{align*}
        |\rho^{st}(x)|\leq g^{-1}(\mu^{st})<\infty.\\
         |\nabla\rho^{st}(x)|\leq |g'(\rho^{st}(x))|^{-1}|\nabla\Phi|\lesssim|\rho^{st}(x)|\Phi^m\lesssim e^{-c \Phi} \Phi^m<\infty,
    \end{align*}
    where we also used that $\rho^{st}=g^{-1}(-\Phi+\mu^{st} )\lesssim e^{c(-\Phi+\mu^{st} )}$.
    
    By calculation, we know that 
    \begin{align*}
        |g''(x)|\leq \frac{|\beta''(x)xb(x)|+|\beta'(x)(b(x)+xb'(x))|}{(xb(x))^2}\lesssim \frac{1}{x}+\frac{1}{x^2},
    \end{align*}
    which implies that 
    \begin{align*}
         |g''(\rho^{st}(x))(\nabla\rho^{st}(x))^2|\lesssim |g''(\rho^{st}(x))||\rho^{st}(x)|^2\Phi^{2m}\lesssim (1+|\rho^{st}(x)|)\Phi^{2m}\lesssim \Phi^{2m}.
    \end{align*}
    Then 
    \begin{align*}
        |\nabla^2\rho^{st}|\leq |g'(\rho^{st}(x))|^{-1}(|g''(\rho^{st}(x))(\nabla\rho^{st}(x))^2|+|\nabla^2\Phi|)\lesssim \rho^{st}\Phi^{2m} \lesssim e^{-c \Phi} \Phi^{2m}<\infty.
    \end{align*}
\end{proof}
Moreover, by definition, $\rho^{st}$ is a strictly positive function. Since $[-\frac12,\frac12]^d$ is compact and $\rho^{st}$ is continuous, there exists a constant $c_{\mathrm{in}}>0$ such that 
\begin{align}
    \rho^{st}\geq 1_{[-\frac12,\frac12]^d}c_{in}.\label{R:rhost>}
\end{align}

\subsection{The main iteration used in the proof of Theorem~\ref{R:thm:sta}}

Let $N>0,0<\epsilon_0\leq \min\{ \frac18, \frac1 {2^{d+10}} c_{in}\}$ and $1<d_0<d-1$ be given.   The iteration is again indexed by a parameter
$
q\in\mathbb{N}_0.
$
We consider the sequences
 \begin{align*}
\lambda_q
=
a^{b^q},
\ 
q\geq0,
\ 
\delta_q
=
(\frac{\epsilon_0}{2})^{d_0+d_0'}
\lambda_1^{2\beta}
\lambda_q^{-2\beta},
\ 
q\geq1,
\ 
\delta_0=1.
\end{align*}  By \eqref{ieq:5ab2}
we obtain $
\sum_{q\geq1}\delta_q^{1/(d_0+d_0')}
<
\epsilon_0.
$

At each step $q$, we construct  
$
(v_q,\rho^i_q,M^i_q)_{1\leq i\leq N}
$
solving the stationary system on the whole space:
\begin{align}
-\div(\beta'(\rho_q^i+\rho^{st})\nabla(\rho_q^i+\rho^{st}))
+
\div(v_q(\rho_q^i+\rho^{st}))+\div(B(\rho_q^i+\rho^{st}))
&=
-\div M^i_q,
\notag\\
\div v_q&=0,
\label{R:eq:5qth}
\end{align}
where we rewrite the notation
$
B(\rho)=Eb(\rho)\rho.
$
In particular, we also need the family $
(v_q,\rho^i_q,M^i_q)_{1\leq i\leq N}
$ has compact support on $\Omega_q \subset [-\frac12,\frac12]^d\subset \mathbb{R}^d$, where 
\begin{align*}
     \Omega_q := \[-\frac13-\sum_{ 1\leq r\leq q} {\delta}_r^{1/2},\frac13+\sum_{ 1\leq r\leq q} {\delta}_r^{1/2}\]^d.
\end{align*}
In fact, we will apply the convex integration on the $[-\frac12,\frac12]^d$ with zero boundary condition.

We initialize the iteration by defining
$
(v_0,\rho^i_0,M^i_0)_{1\leq i\leq N}
$
through
\begin{align*}
\rho^i_0 
&=
 4^{-i}F,
\ \
v_0 =0,
\\ 
M^i_0 
&=\beta'(\rho_0^i+\rho^{st})\nabla(\rho_0^i+\rho^{st})-\beta'(\rho^{st})\nabla(\rho^{st})
-B(\rho_0^i+\rho^{st})+B( \rho^{st}),
\end{align*}
where
$F(x)$ is a smooth mean-zero bounded   function with support in $\Omega_0$, and satisfying $\|F\|_{C_x^0}=\frac1 2 c_{in}, \|F\|_{L^1}\geq \frac1 {2^{d+1}} c_{in}$ by multiplying by a suitable constant. In particular, since $\rho^{st}$ is a stationary solution, we know that $(v_0,\rho^i_0,M^i_0)$ is a solution to \eqref{R:eq:5qth}, and $\supp\rho_0^i,\supp v_0,\supp M_0^i\subset\Omega_0$. 
Also, there exists a sufficiently large constant $C_0>0$ such that
\begin{align}
\|\rho^i_0\|_{C_x^2(\mathbb{R}^d)}
&\leq
C_0^{1/d_0'},\ \|v_0\|_{C_{x}^1(\mathbb{R}^d)}\leq C_0^{1/d_0},\notag\\ 
\|M^i_0\|_{L^1(\mathbb{R}^d)}
&\lesssim\|\beta' \|_{C^1 }+\|E \|_{C_{x}^1([-\frac12,\frac12]^d)}\|b\|_{C^1}\leq 
(\frac{\epsilon_0}{2})^{d_0+d_0'}C_0.
\label{R:bd:5rho_0,M_0}
\end{align}
 
      Our main iteration procedure reads as follows:
\bp\label{R:prop:sta}
Under the assumption of Theorem \ref{R:thm:sta}, there exists  a choice of parameters $a,b,\beta$  such
that the following holds: Let $(v_q,  \rho^i_q,  M^i_q)_{1\leq i \leq N}$ be a solution to the system \eqref{R:eq:5qth} satisfying \begin{align}
\|v_q\|_ {C_{x}^1(\mathbb{R}^d)}\leq C_0^{1/d_0}\lambda_q^{ d+1},\
\|  \rho^i_{q}\|_ {C_{x}^1(\mathbb{R}^d)}+\lambda_q^{-1}\|  \rho^i_{q}\|_ {C_{x}^2(\mathbb{R}^d)}\leq C_0^{1/d_0'} \lambda_q^{ d+1},
\label{R:bd:5rhoqc1}\\
 \|M^i_q\|_{L^1(\mathbb{R}^d)}\leq C_0 \delta_{q+1},&\label{R:bd:5rql1}\\
 \supp\rho_q^i,\supp v_q,\supp M_q^i\subset\Omega_q.&\label{R:bd:supp}
\end{align}
Then there exists 
$(v_{q+1},  \rho^i_{q+1},  M^i_{q+1})_{1\leq i \leq N}$ which solves \eqref{R:eq:5qth} and satisfies \eqref{R:bd:5rhoqc1}-\eqref{R:bd:supp} at the level $q+1$ and
\begin{align}
\|v_{q+1}-v_{q}\|_{L^{d_0}(\mathbb{R}^d)}\leq
C_vC_0^{1/d_0}\delta_{q+1}^{1/d_0},\label{R:bd:5vq+1-vqldd}\ \
\|  \rho^i_{q+1}-  \rho^i_{q}\|_{L^{d_0'}(\mathbb{R}^d)}\leq 
C_v C_0^{1/{d_0'}}\delta_{q+1}^{1/{d_0'}},&\\
 \|v_{q+1}-v_{q}\|_{L^{1}(\mathbb{R}^d)}+ 4^{i}\|  \rho^i_{q+1}-  \rho^i_{q}\|_{W^{1,1+\epsilon}(\mathbb{R}^d)}\leq 
 \delta_{q+1}^{1/d_0'},\  \ \rho^i_{q+1} -   \rho^i_q\geq - \delta_{q+1}^{1/d_0'}.&\label{R:bd:5rhoq+1-rhoql1}
\end{align}for some universal constant $C_v\geq1$.  
\ep

By \eqref{R:bd:supp}, we know that all the bounds on $\mathbb{R}^d$ are in fact restricted on the domain $[-\frac12,\frac12]^d$, so the estimates are similar to Section \ref{sec:convexsta}. 
The proof of the main iteration  will be given below. 
Assuming Proposition~\ref{R:prop:sta}, we now complete the proof of Theorem~\ref{R:thm:sta}.
 
 \begin{proof}[Proof of Theorem~\ref{R:thm:sta}]
We initialize the iteration from the triple
$
(v_0,\rho^i_0,M^i_0)
$
defined above.
By \eqref{R:bd:5rho_0,M_0}, the required estimates hold at level $q=0$.
Next, applying Proposition~\ref{R:prop:sta} inductively, we construct
$
(v_q,\rho^i_q,M^i_q)
$
for every
$
q\geq1.
$
By  
\eqref{R:bd:5vq+1-vqldd}--\eqref{R:bd:5rhoq+1-rhoql1}, the sequence
$
 (v_q,\rho^i_q) 
$
is Cauchy in
$
L^{d_0}\times (L^{d_0'}\cap W^{1,1+\epsilon})^N,
$
and we denote its limit by
$
(v,\rho^i).
$ We then define $\overline \rho^i:=\rho^i+\rho^{st}$.
By \eqref{R:bd:5rql1}, it is straightforward to verify that
$
(\overline \rho^i,v)
$
solves \eqref{R:eq:fpe}.

Furthermore, by \eqref{R:rhost>},  
\eqref{R:bd:5vq+1-vqldd}, and
\eqref{R:bd:5rhoq+1-rhoql1}, we obtain
\begin{align}
 \|v \|_{L^{1}(\mathbb{R}^d)}&\leq\sum_{q=0}  \|v_{q+1}-v_{q}\|_{L^{1}(\mathbb{R}^d)}\leq \sum_{q=0}\delta_{q+1}^{1/d_0}
 \leq \epsilon_0.\notag \\
\| \rho^i-\rho_0^i\|_{L^1(\mathbb{R}^d)} 
&\leq
\sum_{q=0}^\infty
\|\rho^i_{q+1}-\rho^i_q\|_{L^1(\mathbb{R}^d)}
\leq
 4^{-i}\sum_{q=0}^\infty
\delta_{q+1}^{1/d_0'}\leq 4^{-i} \frac{c_{in}}{2^{d+10}}
,
\notag\\
\overline\rho^i
&\geq \rho^{st}+
\rho^i_0
+
\sum_{q=0}^\infty
(\rho^i_{q+1}-\rho^i_q)1_{[-\frac12,\frac12]^d}
\geq
1_{[-\frac12,\frac12]^d}(c_{in}-\frac1 2 c_{in}
-
\sum_{q=0}^\infty
\delta_{q+1}^{1/d_0'})
>0.
\notag
\end{align}
Hence, by noticing that $\|\rho^i_0-\rho_0^j\|_{L^1}\geq |4^{-i}-4^{-j}|\cdot \frac{c_{in}}{2^{d+1}}$, it holds that $\overline \rho^i$ is nonnegative,  and $\rho^i$
 do not coincide with each other.
Since
$
\int_{\mathbb{R}^d}\rho_q^i\,\dif x=0,
$
it follows that $\overline\rho^i$ is a probability density.

Since
$
v\in L^{d_0},\ 
\overline \rho^i\in L^{d_0'},
$ we apply the superposition principle  \cite{Tre16}
 and there exists a probability measure
$
\mathbf{Q}^i
$
on
$
C([0,\infty);\mathbb{R}^d)
$
which is a stationary solution and satisfies
$
\dif \mathbf{Q}^i\circ\Pi_t^{-1}
=
\overline \rho^i\,\dif x,
\ 
t\geq0.$
Then,  a standard result  implies that there exists a $d$-dimensional Brownian motion $W_t,t\geq0$, on a stochastic basis  and   continuous  measurable maps $X^i_t,1  \leq i\leq N $satisfying the corresponding (DD)SDE.
\end{proof}

\subsection{Proof of Proposition \ref{R:prop:sta}}  
We choose the all same parameter as in Section \ref{sec:5choicepara}.

First,  we also replace $(v_{q},  \rho^i_{q},M^i_q)$ by a space-direction mollified field  
\begin{align}
v_l=v_{q}*\phi_l,\ \ 
   \rho^i_l=  \rho^i_{q}*\phi_l,\  M^i_l=M^i_{q}* \phi_l,\notag
\end{align}
where we recall that $\phi_l$ is a   standard radial mollifiers supported on $[-1,1]^d$.  
By calculation we obtain that 
\begin{align*}
-\div(\beta'(\rho^i _l+\rho^{st}) \nabla(\rho^i _l+\rho^{st}))+\div(v_l   (\rho^i _l+\rho^{st}))+\div(B(\rho^i _l+\rho^{st}))&=-\div (M^i_l+M^i_{com} ),\\
\div v_l&=0,
\end{align*}
where
\begin{align*}
M^i_{com} :=&-v_l  ( \rho^i_l+\rho^{st})+(v_{q}  ( \rho^i_{q}+\rho^{st}))* \phi_l\\
&+\beta'(\rho^i _l+\rho^{st}) \nabla(\rho^i _l+\rho^{st})-(\beta'(\rho^i _q+\rho^{st}) \nabla(\rho^i _q+\rho^{st}))*_x\phi_l\\
&-\beta'(\rho^{st}) \nabla\rho^{st}+(\beta'(\rho^{st}) \nabla\rho^{st})*_x\phi_l\\
 &  -B(\rho_l^i+\rho^{st})+B(\rho_q^i+\rho^{st})*_x\phi_l +B(\rho^{st})-B(\rho^{st})*_x\phi_l.\notag
\end{align*}
 Here we used the fact that $\rho^{st}$ is a stationary solution. Since $  \supp\rho_q^i,\supp v_q,\supp M_q^i\subset\Omega_q $ by \eqref{R:bd:supp}, we know that $  \supp\rho_l^i,\supp v_l,\supp M_l^i\subset\Omega_{q+1} $, and then $\supp M_{com}^i\subset\Omega_{q+1} $. 
 
Then, we define the perturbations
\begin{align*}
w_{q+1}^{(p,i)}:&=\sum_{n\geq3}\sum_{\xi\in\Lambda^{n,i}}\tilde{\chi}(\zeta |M^i_l|-n)\(\frac{n}{\zeta }\)^{1/d_0}W_{ (\xi)},\\
w_{q+1}^{(c,i)}:&=\sum_{n\geq3}\sum_{\xi\in\Lambda^{n,i}}\nabla\( \tilde{\chi}(\zeta |M^i_l|-n) \)\(\frac{n}{\zeta }\)^{1/{d_0}}:V_{ (\xi)},
\end{align*}
 and the total perturbation and new iteration  
\begin{align}
w_{q+1}:=\sum_{i=1}^N\(w_{q+1}^{(p,i)}+w_{q+1}^{(c,i)}\),\ \ v_{q+1}:=v_l+w_{q+1}.\notag
\end{align} 
Here we note that $W_{ (\xi)},V_{ (\xi)}$ are introduced in Section \ref{sec:5defq+1}, and are seen as a periodic function on $\mathbb{R}^d$.
Since $\supp M_l^i\subset\Omega_{q+1}$, we know that $\supp (w_{q+1}^{(p,i)}+w_{q+1}^{(c,i)})\subset\Omega_{q+1}$ and so is $v_{q+1}$. Here we note that the building blocks $W_{(\xi)}$ and $V_{(\xi)}$ are not globally integrable. However, since the amplitude functions are supported in $\Omega_{q+1}$, all estimates involving the building blocks are restricted to the interval $[-\frac12,\frac12]^d$. On this domain, the relevant norms coincide with those on the torus, and therefore the same estimates remain valid.

We define the perturbations for the stationary  densities functions as
\begin{align}
\theta_{q+1}^{(p,i)}:&=
\sum_{n\geq3}\sum_{\xi\in\Lambda^{n,i}}\chi(\zeta |M^i_l|-n)\(\frac{n}{\zeta }\)^{1/{d_0'}}\Gamma_{\xi}\(\frac{M^i_l}{|M^i_l|}\)\Theta_{ (\xi)},\notag\\
\theta_{q+1}^{(c,i)}:&=-\hat{\rho}^i_q\int_{\mathbb{R}^d} \theta_{q+1}^{(p,i)}\dif x.\notag
\end{align}
 and then define
\begin{align}
\theta _{q+1}^i:=\theta_{q+1}^{(p,i)}+\theta_{q+1}^{(c,i)},
\ \ 
  \rho^i_{q+1}:=  \rho^i_l+\theta^i _{q+1}.\notag
\end{align} 
Here we note that $\Theta_{ (\xi)}$ is seen as a periodic function on $\mathbb{R}^d$, $ \hat{\rho}^i_q$ is a probability density with support in $\Omega_{q}$ satisfying $\|\hat{\rho}^i_q\|_{C_{x}^2}\lesssim1.$    
Since $\supp M_l^i\subset\Omega_{q+1}$, we know that $\supp (\theta_{q+1}^{(p,i)}+\theta_{q+1}^{(c,i)})\subset\Omega_{q+1}$ and so is $\rho_{q+1}$. 
 By the definition, it is easy to see that $\int_{\mR^d}(\theta_{q+1}^{(p,i)}+\theta_{q+1}^{(c,i)})\dif x=0$, which implies that $\int_{\mR^d}\rho^{i}_{q+1}\dif x=0$.

Now, we need to  establish the corresponding estimates on the perturbations.  
  The definition of the perturbations are the same as before, except for $\theta_{q+1}^{(c,i)}$. However, it could be bounded as the same manner, since there is only an extra $\hat{\rho}_q$.  For example, 
\begin{align*}\|w_{q+1}^{(p,i)}&\|_{L^m(\mathbb{R}^d)}=\|w_{q+1}^{(p,i)}\|_{L^m([-\frac12,\frac12]^d)}\\
&\lesssim\sum_{n\geq3}\sum_{\xi\in\Lambda^{n,i}}\norm{\tilde{\chi}  (\zeta |M^i_l|-n) \(\frac{n}{\zeta }\)^{1/d_0}}_ {C_{x}^0([-\frac12,\frac12]^d)}\|W_{ (\xi)}\|_{L^m(\mathbb{T}^d)}\lesssim l^{-2d-4}r_\perp^{\frac{d-1}{m}-\frac{d-1}{d_0}},\\
\|w_{q+1}^{(c,i)}&\|_{L^m(\mathbb{R}^d)}=\|w_{q+1}^{(c,i)}\|_{L^m([-\frac12,\frac12]^d)}\\
&\lesssim\sum_{n\geq3}\sum_{\xi\in\Lambda^{n,i}}\norm{\tilde{\chi}(\zeta |M^i_l|-n)  \(\frac{n}{\zeta }\)^{1/d_0}}_{C^1_{x}([-\frac12,\frac12]^d)}\|V_{ (\xi)}\|_{L^m(\mathbb{T}^d)}\lesssim l^{-3d-8}r_\perp^{\frac{d-1}{m}-\frac{d-1}{d_0}}\lambda_{q+1}^{-1},
\end{align*}
which are the same as before. All estimates for the perturbations can be obtained by an almost word-for-word repetition of the arguments in Section~\ref{sec:proof1}: for $A_{i,n,\xi}:=\chi(\zeta |M^i_l|-n)\Gamma_{\xi}\(\frac{M^i_l}{|M^i_l|}\)\(\frac{n}{\zeta }\)^{1/d_0'}$,
\begin{align}
\|w_{q+1}^i\|_{C_{x}^1(\mathbb{R}^d)}&\lesssim  \sum_{n\geq3}\sum_{\xi\in\Lambda^{n,i}}\norm{\tilde{\chi}(\zeta |M^i_l|-n) \(\frac{n}{\zeta }\)^{1/d_0}}_{C_{x}^{2}([-\frac12,\frac12]^d)}\|  V_{ (\xi)}\|_{C_{x}^2(\mathbb{T}^d)} \lesssim  l^{-4d-12}\lambda_{q+1} r_\perp^{-\frac{d-1}{d_0}},\notag\\
\|\theta_{q+1}^{(p,i)} \|_{L^{m}(\mathbb{R}^d)}&\lesssim\sum_{n\geq3}\sum_{\xi\in\Lambda^{n,i}}\norm{A_{i,n,\xi}}_{C_{x}^0([-\frac12,\frac12]^d)}\|\Theta_{ (\xi)} \|_{L^{m}(\mathbb{T}^d)} \lesssim l^{-2d-4}r_\perp^{\frac{d-1}{m}-\frac{d-1}{d_0'}} ,\notag\\
\|\theta_{q+1}^{(p,i)}\|_{W^{1,1+\epsilon}(\mathbb{R}^d)}&\lesssim\sum_{n\geq3}\sum_{\xi\in\Lambda^{n,i}}\norm{A_{i,n,\xi}}_ {C_{x}^1([-\frac12,\frac12]^d)}\|\Theta_{ (\xi)} \|_{W^{1,1+\epsilon}(\mathbb{T}^d)} \lesssim l^{-4d-12}\lambda_{q+1}r_\perp^{\frac{d-1}{1+\epsilon}-\frac{d-1}{d_0'}}
 ,\notag\\
\|\theta_{q+1}^{(p,i)} \|_ {C_{x}^i(\mathbb{R}^d)} 
&\lesssim\sum_{n\geq3}\sum_{\xi\in\Lambda^{n,i}}\norm{A_{i,n,\xi}}_ {C_{x}^i([-\frac12,\frac12]^d)}\|\Theta_{ (\xi)} \|_ {C_{x}^i(\mathbb{T}^d)} \lesssim l^{-4d-12}\lambda_{q+1}^i r_\perp^{-\frac{d-1}{d_0'}},\ i=1,2.\notag
\end{align}
 Then the desired estimates on the perturbations are obtained by the same calculations.

 As for the stress term, by the same calculation as Section \ref{sec:5defmq+1}, we could  define 
$$-{M}^i _{q+1}:=M^i_{osc}+M^i_{nonlin}+M^i_{lin}-M^i_{com},$$
where 
\begin{align*}
M^i_{nonlin}: 
&=-  \beta'(\rho^i_{q+1}+\rho^{st} )  \nabla (\rho^i_{q+1}+\rho^{st} )+ \beta'(\rho^i_{l}+\rho^{st} )\nabla( \rho^i_{l}+\rho^{st}) \\
&\quad +Eb( \rho_{q+1}^i+\rho^{st})( \rho_{q+1}^i+\rho^{st})  -Eb(\rho_{l}^i+\rho^{st})(\rho_{l}^i+\rho^{st}) ,\\
M^i_{lin}:&=v_l\theta^i _{q+1}+w_{q+1}  (\rho^i_l+\rho^{st}+\theta_{q+1}^{(c,i)})+w_{q+1}^{(c,i)}\theta_{q+1}^{(p,i)},
\end{align*}  
and 

\begin{align*}
 \div M^i_{osc}&:=\div (w_{q+1}^{(p,i)}\theta_{q+1}^{(p,i)}-M^i_l)\\
 &=\sum_{n\geq3}\sum_{\xi\in\Lambda^{n,i}}\nabla\[\chi(\zeta |M^i_l|-n)\frac{n}{\zeta }\Gamma_{\xi}\(\frac{M^i_l}{|M^i_l|}\)\]\mP_{\neq0}(W_{ (\xi)}\Theta_{ (\xi)})\notag\\
 &\ \ +\div \(\sum_{n\geq3}\chi(\zeta |M^i_l|-n)\frac{n}{\zeta }\frac{M^i_l}{|M^i_l|}-M^i_l\). 
\end{align*} 
Using  the inverse divergence operator  $\mathcal{B}_{\hat{\rho}_q^i}$ in Proposition \ref{R:bb1_1}, we  define  the  oscillation error $M^i_{osc}:=M^i_{osc,x}+M^i_{osc,c}$  by
\begin{align*} M^i_{osc,x}:&=\sum_{n\geq3}\sum_{\xi\in\Lambda^{n,i}}\mathcal{B}_{\hat{\rho}_q^i}\(\nabla\[\chi(\zeta |M^i_l|-n)\frac{n}{\zeta }\Gamma_{\xi}\(\frac{M^i_l}{|M^i_l|}\)\],\mathbb{P}_{\neq0}(W_{ (\xi)}\Theta_{ (\xi)})\),\\
   M^i_{osc,c}:&=\sum_{n\geq3}\chi(\zeta |M^i_l|-n)\frac{n}{\zeta }\frac{M^i_l}{|M^i_l|}- M^i_l.
\end{align*}
Since $\supp \rho^i_l,\supp \theta_{q+1}^i,\supp M_l^i\subset\Omega_{q+1}$, we know that $\supp M_{q+1}^i\subset\Omega_{q+1}$.

We estimate the terms appearing in the definition of $M^i_{q+1}$ separately. For the oscillation error $M^i_{ {osc}}$, the operator $\mathcal{B}_{\hat{\rho}_q^i}$ satisfies the same estimates as in Theorem~\ref{bb1_1}, so the underlying calculations remain unchanged.

\begin{align}
\|M^i_{osc,x} \|_{L^1(\mathbb{R}^d)}
&\lesssim\sum_{n\geq3}\sum_{\xi\in\Lambda^{n,i}}\norm{\chi(\zeta |M^i_l|-n)\frac{n}{\zeta }\Gamma_{\xi}\(\frac{M^i_l}{|M^i_l|}\)}_{C_{x}^2([-\frac12,\frac12]^d)}(r_\perp\lambda_{q+1})^{-1}\|W_{ (\xi)}\Theta_{ (\xi)}\|_{L^{1}(\mathbb{T}^d)}\notag\\
&\lesssim l^{-6d-20}(r_\perp\lambda_{q+1})^{-1} 
\lesssim  \lambda_{q+1}^{-\alpha}\leq  \frac15C_0\delta_{q+2}.\notag
\end{align} 

The treatment of $M^i_{osc,c}$ and the linear error $M^i_{ {lin}}$ is identical as before.

For the nonlinear error $M^i_{nonlin} $, by the same calculation as in \eqref{arhorho-arhorho},  together with \eqref{R:rhostc2}  we have
\begin{align}
    \|M^i_{nonlin}\|_{L^1} 
      &\lesssim (\|\beta' \|_{C^1 }+\|E \|_{C_{x}^0([-\frac12,\frac12]^d)}\|b\|_{C^1})\|\theta^{i}_{q+1}\|_{W^{1,1+\epsilon}} (1+\|\rho^{i}_{l}+\rho^{st}\|_{C_{x}^1} ) \notag\\
      &\lesssim (\lambda_{q}^{ d+1}+1)\lambda_{q+1}^{-2\alpha} \lesssim    C_0 \lambda_{q+1}^{-\alpha}\leq \frac{1}{5}C_0 \delta_{q+2}.\notag
\end{align}

For the commutator error $M^i_{com} $, By the calculation as in \eqref{bd:4mcom}, together with \eqref{R:rhostc2}   we obtain
 \begin{align}
  \|M^i_{ com} \|_{L^1}  & \lesssim l\|v_q\|_{C_{x}^1}(\|\rho_q^i\|_{C_{x}^1}+\|\rho^{st}\|_{C_{x}^1})\notag\\
  &\quad+l (\|\beta' \|_{C^1}+\|E \|_{C_{x}^1([-\frac12,\frac12]^d)}\|b\|_{C^1})\notag\\
  &\quad\quad\quad\times (1+  \| \rho_q^i \|_{C_{x}^2}+\|\rho^{st}\|_{W^{2,\infty}}^2+\| \rho_q^i\|_{C_{x}^1}^2)(\| \rho_q^i \|_{C_{x}^1}+\|\rho^{st}\|_{W^{1,\infty}})\notag\\
 &\lesssim C_0 l\lambda_q^{3d+4}\lesssim  C_0  \lambda_{q+1}^{-\alpha}\leq \frac{1}{5}C_0 \delta_{q+2}.\notag 
   \end{align}    
Then the proof is finished.

\subsection{The main iteration used in the proof of Theorem~\ref{R:thm:4converge}}
Let $0<\epsilon_0\leq \min\{ \frac18, \frac1 {2^{d+10}} c_{in}\}$ be given.  The iteration is indexed by a parameter
$
q\in\mathbb{N}_0.
$
We consider     sequences of parameters
\begin{align*}
\lambda_q
=
a^{b^q},
\ 
q\geq0,
\ 
\delta_q
=
(\frac{\epsilon_0}{2})^{d+1}
\lambda_1^{2\beta}
\lambda_q^{-2\beta},
\ 
q\geq1,
\ 
\delta_0=1,
\end{align*}  by \eqref{ieq:4ab2} 
we obtain $
\sum_{q\geq1}\delta_q^{1/(d+1)}
<
\epsilon_0. $
 
At each step $q$, we construct a family
$(v_q,\rho_q^i,M_q^i)_{i\in\mathbb{N}}$
solving the system
\begin{align}\label{R:eq:4qth}
\partial_t\rho_q^i
-
\div \bigl(\beta'(\rho_q^i+\rho^{st})\nabla(\rho_q^i+\rho^{st})\bigr)
+
\div(v_q(\rho_q^i+\rho^{st}))+\div(B(\rho_q^i+\rho^{st}))
&=
-\div M_q^i,
\\
\div v_q&=0.
\notag
\end{align}

 We recall the notation $
T_q:=\frac13-\sum_{1\leq r\leq q}\delta_r^{1/2}
$ and initialize the iteration by defining
$
(v_0,\rho^i_0,M^i_0)
$
through
\begin{align*}
\rho^i_0(x)
&=
 4^{-i}\div F,
\ \
v_0(x)=0,
\\ 
M^i_0(x)
&=-4^{-i}\partial_t F+\beta'(\rho_0^i+\rho^{st})\nabla(\rho_0^i+\rho^{st})-\beta'(\rho^{st})\nabla\rho^{st}
-B(\rho_0^i+\rho^{st})+B( \rho^{st}),
\end{align*}
where
$  F(t,x)$ is a smooth non-divergence-free bounded $\mR^d$-valued function with support in $[\frac13,\frac23]\times \Omega_0$, and satisfying $\|\div F\|_{C_{t,x}^0}=\frac1 2 c_{in}, \|\div F\|_{C_tL^1}\geq \frac1 {2^{d+1}} c_{in}.$ In particular, $\supp\rho_0^i,\supp v_0,\supp M_0^i\subset [T_0,1-T_0]\times \Omega_0$. 
 
  By choosing $C_0>0$ sufficiently large we obtain  
\begin{align}
\| \rho_0^i\|_{C_{t,x}^2}
& 
\leq
4^{-i}C_0^{1/d_0'},\ \ \|v_0\|_{C_{t,x}^1}\leq C_0^{1/d_0},
\notag\\
\|M_0^i\|_{L_t^1L^1}
&\lesssim  (\|\beta' \|_{C^1 }+\|E \|_{C_{x}^1([-\frac12,\frac12]^d)}\|b\|_{C^1})\|\rho_0^i\|_{C_{t,x}^1}\lesssim 
4^{-i}
\leq
(\frac{\epsilon_0}{2})^{d+1}4^{-i}C_0.
\label{R:bd:4rho_0,M_0}
\end{align}

With the above assumptions in hand, our main iteration relies on the first step of iteration and reads as follows:

\bp\label{R:prop:case4} 
Under the assumption of Theorem \ref{R:thm:4converge}, there exist $d+1>d_0>2>d_0'>1$ with $\frac{1}{d_0}+\frac{1}{d_0'}=1$ and a choice of parameters $a,b,\alpha,\beta$  such
that the following holds: Let $(v_q,\rho^{i}_q,  M^{i}_q)_{i\in\mN}$ be a solution to the system \eqref{R:eq:4qth} satisfying   
\begin{align}\label{R:bd:4vql2}
\|v_q\|_{L^{d_0}_tL^{d_0}(\mathbb{R}^d)}\leq C_vC_0^{1/d_0}\sum_{m=0}^q\delta_{m}^{1/d_0},
\end{align}
for some universal constant $C_v\geq1$, and
\begin{align}
\|v_q\|_{C_{t,x}^1(\mathbb{R}^d)}\leq C_0^{1/d_0}\lambda_q^{d+4},\
4^{i} \|\rho^{i}_{q}\|_{C_{t,x}^1(\mathbb{R}^d)}+\lambda_q^{-2}\|\rho^{i}_{q}\|_{C_{t,x}^2(\mathbb{R}^d)}\leq C_0^{1/d_0'}\lambda_q^{d+4},
\label{R:bd:4rhoqc1}\\
 \|M^{i}_q\|_{L^1_tL^1(\mathbb{R}^d)}\leq C_02^{-i}\delta_{q+1},&\label{R:bd:4rql1}\\
 M^{i}_{q}=M^{i}_{0}-v_{q}\rho^{i}_0,\ \rho_q^{i}=\rho_0^{i}, \ {\rm for\ }  i>N_{q},\label{R:bd:4mqi>nq}\\
  \supp\rho_q^i,\supp v_q,\supp M_q^i\subset [T_q,1-T_q]\times \Omega_q.&\label{R:bd:supp4}
\end{align}
Then there exists 
$(v_{q+1},\rho^{i}_{q+1},  M^{i}_{q+1})_{i\in\mN}$ which solves \eqref{R:eq:4qth} and satisfies \eqref{R:bd:4vql2}-\eqref{R:bd:supp4} at the level $q+1$ and
\begin{align}
\|v_{q+1}-v_{q}\|_{L^{d_0}_tL^{d_0}(\mathbb{R}^d)}\leq
C_vC_0^{1/d_0}\delta_{q+1}^{1/d_0},\label{R:bd:4vq+1-vqldd}\ \
\|\rho^{i}_{q+1}-\rho^{i}_{q}\|_{L^{d_0'}_tL^{d_0'}(\mathbb{R}^d)}\leq
C_v C_0^{1/{d_0'}}2^{-i/d_0'}\delta_{q+1}^{1/{d_0'}},\\
\|v_{q+1}-v_{q}\|_{L^{r}_tL^{p}(\mathbb{R}^d)}+\|v_{q+1}-v_q\|_{C_tL^s(\mathbb{R}^d)}\leq
\delta_{q+1}^{1/d_0},\label{R:bd:4vq+1-vqlpr}\\
\|\rho^{i}_{q+1}-\rho^{i}_{q}\|_{C_tL^1(\mathbb{R}^d)}\leq
 4^{-i}\delta_{q+1}^{1/d_0'},\ \  \inf_{t\in [0,1]}(\rho^{i}_{q+1} - \rho^{i}_q)\geq - \delta_{q+1}^{1/d_0'} .\label{R:bd:4rhoq+1-rhoql1}
\end{align}
\ep

 \begin{proof}[Proof of   Theorem \ref{R:thm:4converge}]
We intend to start the iteration from 
$(v_0, \rho^{i}_0,  M^{i}_0)_{i\in\mN}$ which are defined as above. 
 Next, we use Proposition \ref{R:prop:case4} to build inductively $(v_q,\rho^{i}_q, M^{i}_q)_{i\in\mN}$ for every $q \geq 1$. By   \eqref{R:bd:4vq+1-vqldd}-\eqref{R:bd:4rhoq+1-rhoql1}, the sequence $\{(v_q,\rho^{i}_q)_{i\in\mN}\}_{q\in \N}$ is
Cauchy in 
$\(L^r([0,1];L^p)\cap L^{d_0}([0,1]\times \mathbb{R}^d)\cap C([0,1];L^s)\)\times \(L^{d_0'}([0,1]\times \mathbb{R}^d)\cap C([0,1];L^1)\)^{\mN}$ 
     and we denote by $(v,\rho^{i})$ its limit.    We then define $\overline \rho^i:=\rho^i+\rho^{st}$.

 By \eqref{R:bd:4vq+1-vqldd}-\eqref{R:bd:4rhoq+1-rhoql1} we have
\begin{align} 
  \|\rho^i-\rho_0^i\|_{C_tL^1}&\leq
4^{-i}\sum_{q=0}^\infty\delta_{q+1}^{1/{d_0'}}\leq 4^{-i} \frac{c_{in}}{2^{d+10}},\  \  \|v \|_{L^{r}_tL^{p}}+\|v  \|_{C_tL^s}\leq 
\sum_{q\geq0}\delta_{q+1}^{1/d_0}\leq\epsilon_0,\notag\\
\inf_{t\in [0,1]}\overline \rho^i &\geq \rho^{st}+ \inf_{t\in [0,1]} \rho_0^i +
\sum_{q=0}^\infty \inf_{t\in [0,1]}(\rho_{q+1}^i- \rho_q^i) \geq (c_{in}-\frac12c_{in}-\sum_{q=0}^\infty \delta_{q+1}^{1/d_0'})1_{[-\frac12,\frac12]}>0,\notag
\end{align}
at which point, together with the fact that $\|\rho^i_0-\rho_0^j\|_{C_tL^1}\geq |4^{-i}-4^{-j}|\cdot \frac{c_{in}}{2^{d+1}}$, we obtain $\overline\rho^i\to \rho^{st}$ in $ C_tL^1$ as $i\to\infty$,   
 $\overline\rho^{i}$ is a  probability density function, and $\overline\rho^i$
 do not coincide with each other. 
 We finish the proof of the first statement. 

 For the second statement,
Since $|v| \in L_{t}^{d_0 }L^{d_0 },\overline \rho^i \in L_{t}^{d_0'}L^{d_0'}$,  and 
$
t\mapsto \overline\rho^i(t)
$
is continuous on $[0,1]$, we finish the proof using the superposition principle.
 \end{proof}

\subsection{ Proof of Proposition \ref{R:prop:case4}} 
We choose the all same parameter as in Section \ref{sec:convex}.

First,  we mollify the first $N_{q+1}$ equations by\begin{align}
v_l=(v_{q }*_x\phi_l)*_t\varphi_l,\ \ 
 \rho_l^i=(\rho_{q}^i*_x\phi_l)*_t\varphi_l,\ \ M_l^i=(M_{q }^i*_x\phi_l)*_t\varphi_l,\notag
\end{align}
where 
$
\phi_l:=\frac1{l^d}\phi\bigl(\tfrac{\cdot}{l}\bigr)
$
is a family of standard radial mollifiers on $[-1,1]^d$, and  
$
\varphi_l:=\frac1l\varphi\bigl(\tfrac{\cdot}{l}\bigr)
$
is a family of standard mollifiers supported in $(0,1)$. By straightforward calculations  we obtain 
\begin{align}
\partial_t \rho_l^i-\div(\beta'(\rho_l^i+\rho^{st})\nabla(\rho_l^i+\rho^{st}))+\div(v_l (\rho_l^i+\rho^{st}))+\div(B(\rho_l^i+\rho^{st}))&=-\div (M^i_l+M^i_{com}),\notag\\ \div v_l&=0,\notag
\end{align} where
\begin{align*}
M^i_{com} :=&-v_l   (\rho^i_l+\rho^{st})+(v_{q}   (\rho^i_{q}+\rho^{st}))* \phi_l*_t\varphi_l\\
&+\beta'(\rho^i _l+\rho^{st}) \nabla(\rho^i _l+\rho^{st})-(\beta'(\rho^i _q+\rho^{st}) \nabla(\rho^i _q+\rho^{st}))*_x\phi_l*_t\varphi_l\\
&-\beta'(\rho^{st}) \nabla\rho^{st}+(\beta'(\rho^{st}) \nabla\rho^{st})*_x\phi_l*_t\varphi_l\\
 &  -B(\rho_l^i+\rho^{st})+B(\rho_q^i+\rho^{st})*_x\phi_l*_t\varphi_l +B(\rho^{st})-B(\rho^{st})*_x\phi_l*_t\varphi_l.\notag
\end{align*} 
  Since $  \supp\rho_q^i,\supp v_q,\supp M_q^i\subset[T_q,1-T_q]\times \Omega_q $, we know that $  \supp\rho_l^i,\supp v_l,\supp M_l^i\subset [T_{q+1},1-T_{q+1}]\times \Omega_{q+1} $, and then $\supp M_{com}^i\subset [T_{q+1},1-T_{q+1}]\times\Omega_{q+1} $.

We next define the perturbations for the drift term. 
For $1\leq i\leq N_{q+1}$, let
\begin{align*}
w_{q+1}^{(p,i)}
:=
\sum_{n\geq3}
&\tilde{\chi}(\zeta^i|M_l^i|-n)
\Bigl(\frac{n}{\zeta^i}\Bigr)^{1/d_0}
\sum_{\xi\in\Lambda^n}
W_{(\xi,n,i)}g_{(\xi,i,d_0)},\\
w_{q+1}^{(c,i)}
:=
\sum_{n\geq3}
\sum_{\xi\in\Lambda^n}
\Bigl(
&
-\tilde{\chi}(\zeta^i|M_l^i|-n)
\Bigl(\frac{n}{\zeta^i}\Bigr)^{1/d_0}
\frac1{(n_*\lambda_{q+1})^2}
\nabla\Phi_{(\xi,n,i)}
\,\xi\cdot\nabla\psi_{(\xi,n,i)}
\\
&\quad
+
\nabla\bigl(
\tilde{\chi}(\zeta^i|M_l^i|-n)
\bigr)
\Bigl(\frac{n}{\zeta^i}\Bigr)^{1/d_0}
:V_{(\xi,n,i)}
\Bigr)
g_{(\xi,i,d_0)},
\end{align*}
and the total perturbation, the new velocity field by
\begin{align}
w_{q+1}
:=
\sum_{i=1}^{N_{q+1}}
\bigl(
w_{q+1}^{(p,i)}
+
w_{q+1}^{(c,i)}
\bigr),
\qquad
v_{q+1}
:=
v_l+w_{q+1}.
\notag
\end{align}
Here we note that $W_{ (\xi,n,i)},V_{ (\xi,n,i)}$ are introduced in Section \ref{sec:4defq+1}, and are seen as a periodic function on $\mathbb{R}^d$.
Since $\supp M_l^i\subset[T_{q+1},1-T_{q+1}]\times\Omega_{q+1}$, we know that $\supp (w_{q+1}^{(p,i)}+w_{q+1}^{(c,i)})\subset[T_{q+1},1-T_{q+1}]\times\Omega_{q+1}$ and so is $v_{q+1}$.  
 
We next define the perturbations for the density functions. 
For $1\leq i\leq N_{q+1}$, we set
\begin{align}
\theta_{q+1}^{(p,i)}
:=
&\sum_{n\geq3}
\chi(\zeta^i|M_l^i|-n)
\Bigl(\frac{n}{\zeta^i}\Bigr)^{1/d_0'}
\sum_{\xi\in\Lambda^n}
\Gamma_\xi\Bigl(\frac{M_l^i}{|M_l^i|}\Bigr)
\Theta_{(\xi,n,i)}
g_{(\xi,i,d_0')},
\notag\\
\theta_{q+1}^{(c,i)}
:=
&-\hat{\rho}^i_q
\int_{\mathbb{R}^d}
\theta_{q+1}^{(p,i)}\,\dif x,
\notag\\
\theta_{q+1}^{(o,i)}
:=
&-
\sigma^{-1}
\sum_{n\geq3}
\sum_{\xi\in\Lambda^n}
h_{(\xi,i,d_0)}
\div\Bigl(
\chi(\zeta^i|M_l^i|-n)
\frac{n}{\zeta^i}
\Gamma_\xi\Bigl(\frac{M_l^i}{|M_l^i|}\Bigr)
\xi
\Bigr).
\notag
\end{align}
We now define, for every $1\leq i\leq N_{q+1}$,
\begin{align}
\theta_{q+1}^i
:=
\theta_{q+1}^{(p,i)}
+
\theta_{q+1}^{(c,i)}
+
\theta_{q+1}^{(o,i)},
\qquad
\rho_{q+1}^i
:=
\rho_l^i+\theta_{q+1}^i.
\notag
\end{align}
For $i>N_{q+1}$, we simply set
$
\rho_{q+1}^i:=\rho_q^i.
$

By construction,
$
\int_{\mathbb{R}^d}\rho_{q+1}^i\,\dif x=0
$
for every $i\in\mathbb{N}$.
Since $\supp M_l^i\subset [T_{q+1},1-T_{q+1}]\times\Omega_{q+1}$, we know that $\supp (\theta_{q+1}^{(p,i)}+\theta_{q+1}^{(c,i)})\subset[T_{q+1},1-T_{q+1}]\times\Omega_{q+1}$ and so is $\rho_{q+1}$.

Now we   establish the corresponding estimates on the perturbations.  
  The definition of the perturbations are also  the same as before, except for $\theta_{q+1}^{(c,i)}$.  
 Here we note that since $\supp_x M_l^i\subset [-\frac12,\frac12]^d$, all estimates  are restricted to the domain $ [-\frac12,\frac12]^d$.
We now collect several key estimates   in the sequel:
  \begin{align*}
\|w_{q+1}^{(p,i)}\|_{L^u_tL^m(\mathbb{R}^d)}
&\lesssim l^{-2d-4}r_\perp^{\frac{d-1}{m}-\frac{d-1}{d_0}} r_\parallel^{\frac{1}{m}-\frac{1}{d_0}}\eta^{\frac1u-\frac1{d_0}}, \\
\|w_{q+1}^{(c,i)}\|_{L^u_tL^m(\mathbb{R}^d)}
&\lesssim l^{-3d-8}r_\perp^{\frac{d-1}{m}-\frac{d-1}{d_0}} r_\parallel^{\frac{1}{m}-\frac{1}{d_0}}\frac{r_\perp}{ r_\parallel}\eta^{\frac1u-\frac1{d_0}},\\
\|\theta_{q+1}^{(p,i)} \|_{L^u_tL^{m}(\mathbb{R}^d)}
&\lesssim l^{-2d-4}r_\perp^{\frac{d-1}{m}-\frac{d-1}{d_0'}} r_\parallel^{\frac{1}{m}-\frac{1}{d_0'}}\eta^{\frac1u-\frac{1}{d_0'}},\\
\|\theta_{q+1}^{(o,i)}\|_{C_tC^1(\mathbb{R}^d)} 
&\lesssim \sigma^{-1} l^{-6d-20},\\
\|\theta_{q+1}^{(p,i)}\|_{L^1_tW^{1,1+\epsilon}(\mathbb{R}^d)}
&\lesssim l^{-4d-12}\lambda_{q+1}r_\parallel^{\frac1{1+\epsilon}-\frac{1}{d_0'}}r_\perp^{\frac{d-1}{1+\epsilon}-\frac{d-1}{d_0'}}\eta^{1-\frac1{d_0'}}.
\end{align*}

 Then the desired estimates on the perturbations are obtained by the same calculations.

 As for the stress term, by the same calculation as Section \ref{sec:4defmq+1}, we could  define  for $1\leq i\leq N_{q+1}$,
$$-{M}^i _{q+1}:=M^i_{osc}+M^i_{nonlin}+M^i_{lin}-M^i_{com},$$
where \begin{align*} 
M^{i}_{nonlin}:&= -  \beta'(\rho^{i}_{q+1}+\rho^{st}) \nabla(\rho^{i}_{q+1}+\rho^{st})+\beta'(\rho^{i}_{l}+\rho^{st})\nabla (\rho^{i}_{l}+\rho^{st} )\\
&\quad-Eb( \rho_{q+1}^i+\rho^{st})(\rho_{q+1}^i+\rho^{st})+Eb(\rho_{l}^i+\rho^{st})(\rho_{l}^i+\rho^{st}),\\
M^{i}_{lin}:&=v_l\theta^{i}_{q+1}+w_{q+1}(\rho^{i}_l+\rho^{st}+\theta_{q+1}^{(c,i)}+\theta_{q+1}^{(o,i)}) +w_{q+1}^{(c,i)}\theta_{q+1}^{(p,i)}.
\end{align*}
 To define the oscillation error,  we define
$\mathbb{P}_{\neq 0,\hat{\rho}^i_q}f:=f- \hat{\rho}^i_q\cdot\int_{\mathbb{R}^d}f\dif x$.  
In particular, for any mean-zero function $f$, we have 
 $\mathbb{P}_{\neq 0,\hat{\rho}^i_q}f=f$. Then, we have 
\begin{align}
M^{i}_{osc}:= &\partial_t\theta^{i}_{q+1}+\div (w_{q+1}^{(p,i)}\theta_{q+1}^{(p,i)}-M^{i}_l)\notag\\
 =&\mathbb{P}_{\neq 0,\hat{\rho}^i_q} (\partial_t\theta^{(p,i)}_{q+1}+\div (w_{q+1}^{(p,i)}\theta_{q+1}^{(p,i)}-M^{i}_l)+\partial_t\theta^{(o,i)}_{q+1}
 )\notag\\ 
=
&
\sum_{n\geq3}
\sum_{\xi\in\Lambda^n}
\mathbb{P}_{\neq 0,\hat{\rho}^i_q}\left(\partial_t
\Bigl[
\chi(\zeta^i|M_l^i|-n)
\Bigl(\frac{n}{\zeta^i}\Bigr)^{1/d_0'}
\Gamma_\xi\Bigl(\frac{M_l^i}{|M_l^i|}\Bigr)
g_{(\xi,i,d_0')}
\Bigr]
\Theta_{(\xi,n,i)}\right)
\notag\\
&+
\sum_{n\geq3}
\sum_{\xi\in\Lambda^n}
\mathbb{P}_{\neq 0,\hat{\rho}^i_q}\left(\nabla
\Bigl[
\chi(\zeta^i|M_l^i|-n)
\frac{n}{\zeta^i}
\Gamma_\xi\Bigl(\frac{M_l^i}{|M_l^i|}\Bigr)
\Bigr]
g_{(\xi,i,d_0)}
g_{(\xi,i,d_0')}
\mathbb{P}_{\neq0}
\bigl(
W_{(\xi,n,i)}
\Theta_{(\xi,n,i)}
\bigr)
\right)\notag\\
&+
\div
\Bigl(
\sum_{n\geq3}
\chi(\zeta^i|M_l^i|-n)
\frac{n}{\zeta^i}
\frac{M_l^i}{|M_l^i|}
-
M_l^i
\Bigr)
\notag\\
&-
\sigma^{-1}
\sum_{n\geq3}
\sum_{\xi\in\Lambda^n}
h_{(\xi,i,d_0)}
\partial_t
\div
\Bigl(
\chi(\zeta^i|M_l^i|-n)
\frac{n}{\zeta^i}
\Gamma_\xi\Bigl(\frac{M_l^i}{|M_l^i|}\Bigr)
\xi
\Bigr).\notag
\end{align}
Now, we  apply the  Bogovskii  operators $\mathcal{B}_{\hat{\rho}_q^i},\mathcal{R}_{\hat{\rho}_q^i}$ in Section \ref{R:tamr} to   define $M^{i}_{osc}:=M^{i}_{osc,t}+M^{i}_{osc,x}+M^{i}_{osc,c}+M^{i}_{osc,o}$ as
\begin{align*}
   M^{i}_{osc,t}&:=\sum_{n\geq3}\sum_{\xi\in\Lambda^{n}}\mathcal{B}_{\hat{\rho}_q^i}\(\partial_t[\chi(\zeta^{i}|M^{i}_l|-n)\(\frac{n}{\zeta^{i}}\)^{1/d_0'}\Gamma_{\xi}\(\frac{M^{i}_l}{|M^{i}_l|}\)g_{(\xi,i,d_0')}]\Theta_{(\xi,n,i)}\),\\
M^{i}_{osc,x}&:=\sum_{n\geq3}\sum_{\xi\in\Lambda^{n}}\mathcal{R}_{\hat{\rho}_q^i}\(\nabla[\chi(\zeta^{i}|M^{i}_l|-n)\frac{n}{\zeta^{i}}\Gamma_{\xi}\(\frac{M^{i}_l}{|M^{i}_l|}\)],\mathbb{P}_{\neq0}(W_{(\xi,n,i)}\Theta_{(\xi,n,i)})\)g_{(\xi,i,d_0)}g_{(\xi,i,d_0')},\\
   M^{i}_{osc,c}&:=\sum_{n\geq3}\chi(\zeta^{i}|M^{i}_l|-n)\frac{n}{\zeta^{i}}\frac{M^{i}_l}{|M^{i}_l|}- M^{i}_l,\\
    M^{i}_{osc,o}&:=-\sigma^{-1}\sum_{n\geq3}\sum_{\xi\in\Lambda^{n}}h_{(\xi,i,d_0)}\partial_t\(\chi(\zeta|M_l^i|-n)\frac{n}{\zeta^i}\Gamma_{\xi}\(\frac{M_l^i}{|M_l^i|}\)\xi\).
\end{align*}
Since $\supp \rho^i_l,\supp \theta_{q+1}^i,\supp M_l^i\subset[T_{q+1},1-T_{q+1}]\times\Omega_{q+1}$, we know that $\supp M_{q+1}^i\subset[T_{q+1},1-T_{q+1}]\times\Omega_{q+1}$.

 We now estimate each term in the definition of $M_{q+1}^i$ separately  for $1\leq i\leq N_{q+1}$.

For the oscillation error $M^{i}_{osc,t}$,  the operator $\mathcal{R}_{\hat{\rho}_q^i}$ satisfies the same estimates as in Theorem \ref{bb1_1}, so the underlying calculations remain unchanged.  
\begin{align}
    \|M^{i}_{osc,t}\|_{L^1_tL^1}
    &\lesssim \sum_{n\geq3}\sum_{\xi\in\Lambda^{n}}\norm{\chi(\zeta^{i}|M^{i}_l|-n)\(\frac{n}{\zeta^{i}}\)^{1/d_0'}\Gamma_{\xi}\(\frac{M^{i}_l}{|M^{i}_l|}\)}_{C_{t,x}^1}\|g_{(\xi,i,d_0')}\|_{W_t^{1,1}} \|\Theta_{(\xi,n,i)}\|_{C_tL^1}\notag\\
    &\lesssim l^{-4d-12}r_\perp^{d-1-\frac{d-1}{d_0'}}r_\parallel^{1-\frac{1}{d_0'}}\sigma \eta^{-\frac{1}{d_0'}} \lesssim2^{-i} \lambda_{q+1}^{-\alpha}.\notag
\end{align}

 By \eqref{bd:gwnp} and the bounds above    we have
\begin{align}
   \|  M^{i}_{osc,o}\|_{L^1_tL^1}&\lesssim \sigma^{-1}\sum_{n\geq3}\sum_{\xi\in\Lambda^{n}}\norm{  \chi(\zeta|M_l^i|-n)\frac{n}{\zeta^i}\Gamma_{\xi}\(\frac{M^{i}_l}{|M^{i}_l|}\)\xi}_{C_{t,x}^1} \lesssim  \sigma^{-1}l^{-4d-12}
   \lesssim 2^{-i}\lambda_{q+1}^{-\alpha}.\notag
\end{align} 
 
For the nonlinear error $M_{nonlin}^i$,  similar to \eqref{arhorho-arhorho}, we have
\begin{align}
    \|M^{i}_{nonlin}\|_{L^1_tL^1} 
      &\lesssim (\|\beta' \|_{C^1 }+\|E \|_{C_{x}^0([-\frac12,\frac12]^d)}\|b\|_{C^1})\notag\\
       &\quad\quad\times ( \|\theta^{i}_{q+1}\|_{C_tL^{1+\epsilon}}+ \|\theta^{i}_{q+1}\|_{L^1_tW^{1,1+\epsilon}})  (1+\|\rho^{i}_{q}+\rho^{st}\|_{C_{t,x}^1})  \notag\\
      &\lesssim( \lambda_{q}^{d+4}+1)2^{-i}\lambda_{q+1}^{-2\alpha}\lesssim  2^{-i} C_0 \lambda_{q+1}^{-\alpha}\leq \frac{1}{5}C_02^{-i}\delta_{q+2}.\notag
\end{align}

For the commutator error $M^i_{com} $,  we obtain
 \begin{align}
  \|M^i_{ com} \|_{L_t^1L^1}  & \lesssim l\|v_q\|_{C_{x}^1}(\|\rho_q^i\|_{C_{t,x}^1}+\|\rho^{st}\|_{C_{x}^1})\notag\\
  &\quad+l (\|\beta' \|_{C^1}+\|E \|_{C_{t,x}^1([-\frac12,\frac12]^d)}\|b\|_{C^1})\notag\\
  &\quad\quad\quad\times (1+  \| \rho_q^i \|_{C_{t,x}^2}+\| \rho_q^i\|_{C_{t,x}^1}^2+\|\rho^{st}\|_{W^{2,\infty}}^2)(\| \rho_q^i \|_{C_{t,x}^1}+\|\rho^{st}\|_{W^{1,\infty}})\notag\\
 &\lesssim  C_0 l\lambda_q^{3d+12}\lesssim  2^{-i}C_0  \lambda_{q+1}^{-\alpha/2}\leq \frac{1}{5}C_02^{-i}\delta_{q+2}.\notag 
   \end{align}

For the case $i>N_{q+1}$, since $\supp_x v_{q+1},\supp_x v_q\subset [-\frac12,\frac12]^d$ and $\rho^{st}=g^{-1}(-\Phi(x)+\mu^{st})$ is constant on $[-\frac12,\frac12]^d$, we obtain that
$
\div((v_{q+1}-v_q) \rho^{st})=0.$ Together with $\rho_{q+1}^i=\rho_q^i=\rho_0^i,
$
we   define
\begin{align}
M_{q+1}^i
:=
M_q^i-(v_{q+1}-v_q)\rho_0^i
=
M_0^i-v_{q+1}\rho_0^i.
\notag
\end{align}
Since $\supp \rho^i_0,\supp M_0^i\subset[T_{q+1},1-T_{q+1}]\times\Omega_{q+1}$, we know that $\supp M_{q+1}^i\subset[T_{q+1},1-T_{q+1}]\times\Omega_{q+1}$.

In the end,  by \eqref{R:bd:4rho_0,M_0}  and \eqref{R:bd:4vql2} we have  
\begin{align*}
    \|M^{i}_{q+1}\|_{L^1_tL^1}&\leq \|M^{i}_{0}\|_{L^1_tL^1}+\|v_{q+1}\|_{L^{d_0}_tL^{d_0}}\|\rho^{i}_0\|_{L^{d_0'}_tL^{d_0'}} \lesssim C_0 C_v 4^{-i}  \leq C_0 2^{-i}\delta_{q+2}.
\end{align*}

\noindent{\bf Acknowledgments.}
The author is deeply grateful to Prof.~Michael R\"ockner and
Prof.~Xiangchan Zhu for their invaluable discussions, insightful comments, and
generous support. Their suggestions have greatly improved both the mathematical
content and the presentation of this work.

\appendix
 \renewcommand{\appendixname}{Appendix~\Alph{section}}
  \renewcommand{\theequation}{A.\arabic{equation}}

\section{$L^{d_0}$-based  intermittent spatial-time jets}\label{gij}

In this section we recall the $L^{d_0}$-based intermittent spatial-time jets as  presented in  \cite[Appendix C.1]{LRZ25}. 

First we introduce  the following geometrical lemma:
\bl\cite[Lemma 3.1]{BCDL21}\label{lem:cv2}
Let $d\geq2$. There exists a finite set $\Lambda\in \mathbb{S}^{d-1}\cap \mathbb{Q}^d $ and non-negative $C^\infty$-function  $\Gamma_{\xi}:\mathbb{S}^{d-1}\to\R $ such
that for every $R \in \mathbb{S}^{d-1}$
$$R=\sum_{\xi\in\Lambda}\Gamma_\xi(R)\xi.$$
\el
 With Lemma \ref{lem:cv2} in hand,   it is easy to generate $2$ disjoint families $\Lambda^{1},\Lambda^{2}$, where each one enjoys the property of Lemma \ref{lem:cv2} by taking suitable rational rotations of one fixed set. For simplicity, we denote $\Lambda:= \Lambda^{1}\cup\Lambda^{2}$. Moreover, we know that  $\{\Gamma_{\xi}\}_{\xi\in \Lambda }$ are uniformly bounded.

For parameters $\lambda,r_\perp, r_\parallel > 0 $, we assume
$$\lambda^{-1}\ll r_{\perp}\ll r_{\parallel}\ll 1,\ \ \lambda r_{\perp}\in\mathbb{N}.$$

For each $\xi\in\Lambda$ let us define $A^i_\xi\in \mathbb{S}^{d-1}\cap \mathbb{Q}^d,\ i=1,2,...,d-1$, such that $\{\xi, A^i_\xi,i=1,...,d-1\}$
form an orthonormal basis in $\mathbb{R}^d$. 
 Let  $n_*\in\mN$ such that$\{n_*\xi, n_*A^i_\xi,i=1,...,d-1\}\subset\mathbb{Z}^d$
for every $\xi\in\Lambda$. 
We define $\phi : \mathbb{R}^{d-1} \to \mathbb{R}$ be a smooth function with support in a ball of radius 1, $\phi\equiv 1$ on $B(0,\frac13)$ and mean-zero. We define $\Phi$ such that $\phi = -\Delta\Phi$. 
Let $\psi : \mathbb{R}\to\mathbb{R}$  be a smooth, mean-zero  function with
support in $B(0,1)$ satisfying $\psi\equiv 1$ on $B(0,\frac13)$. 
Define $\phi' : \mathbb{R}^{d-1} \to \mathbb{R}$ to be a smooth non-negative   function with support in $B(0,\frac13)$ satisfying $\int_{\mathbb{R}^{d-1}}\phi'(x_1,x_2,...,x_{d-1})\dif x_1\dif x_2..\dif x_{d-1}=1,$
and let $\psi' : \mathbb{R} \to \mathbb{R}$ be a smooth non-negative   function with support in $B(0,\frac13)$ such that
$\int_\mathbb{R}\psi'(x_d)\dif x_d=1.$

Let  $d_0\geq1$ be fixed, we define the rescaled cut-off functions
$$\phi _{r_{\perp},{d_0}}(x_1,x_2,...,x_{d-1})=\frac1{r_{\perp}^{(d-1)/{d_0}}}\phi (\frac{x_1}{r_{\perp}},\frac{x_2}{r_{\perp}},...,\frac{x_{d-1}}{r_{\perp}}),$$
$$\Phi _{r_{\perp},{d_0}}(x_1,x_2,...,x_{d-1})=\frac1{r_{\perp}^{(d-1)/{d_0}}}\Phi (\frac{x_1}{r_{\perp}},\frac{x_2}{r_{\perp}},...,\frac{x_{d-1}}{r_{\perp}}),$$
$$\psi _{r_{\parallel},{d_0}}(x_d)=\frac1{r_{\parallel}^{1/{d_0}}}\psi (\frac{x_d}{r_{\parallel}}).$$
Similarly,  for a conjugate exponent $d_0'\in[1,\infty]$, we define $\phi '_{r_{\perp},d_0'}$ and $\psi'_{r_{\parallel},d_0'}$ as the same manner. Then we periodize $\phi _{r_{\perp},v},\Phi _{r_{\perp},{d_0}},$ $\psi _{r_{\parallel},{d_0}}$, $\phi '_{r_{\perp},d_0'}$ and $\psi'_{r_{\parallel},{d_0}'}$ so that they can be viewed as functions on $\mathbb{T}^{d-1}$ and $\mathbb{T}$ respectively.
Consider a large time oscillation parameter $  \mu=r_\perp^{-\frac{d-1}{d_0}}r_\parallel^{-\frac1{d_0}}> 0$. For every $\xi\in\Lambda$ we introduce
$$\psi _{(\xi,d_0)}(t,x):=
\psi _{r_{\parallel},d_0}(n_*r_{\perp}\lambda(x\cdot \xi-\mu t)),$$
$$\Phi _{(\xi,d_0)}(x):=
\Phi _{r_{\perp},d_0}(n_*r_{\perp}\lambda x\cdot A^1_{\xi},...,n_*r_{\perp}\lambda x\cdot A^{d-1}_{\xi}),$$
$$\phi _{(\xi,d_0)}(x):=
\phi _{r_{\perp},d_0}(n_*r_{\perp}\lambda x\cdot A^1_{\xi},...,n_*r_{\perp}\lambda x\cdot A^{d-1}_{\xi}).$$
Similarly  we define  $\phi '_{(\xi,d_0')}$  and $\psi'_{(\xi,d_0')}$. 

The building blocks $W_{(\xi,d_0)} :\mathbb{R} \times\mathbb{T}^d \to \mathbb{R}^d,\Theta_{(\xi,d_0')} :\mathbb{R} \times\mathbb{T}^d \to \mathbb{R}$ are defined as
$$W_{(\xi,d_0)}(t,x):=\xi\psi_{(\xi,d_0)}(t,x)\phi_{(\xi,d_0)}(x),$$
$$\Theta_{(\xi,d_0')}(t,x):=\psi'_{(\xi,d_0')}(t,x)\phi'_{(\xi,d_0')}(x).$$
By the definition and the choice of $\mu$ we have that 
\begin{align}
\int _{\mathbb{T}^d}W_{(\xi,d_0)}\Theta_{(\xi,d_0')}\dif x=\xi,\label{eq:intwthe}\\
\partial_t\Theta_{(\xi,d_0')}+ \div (W_{(\xi,d_0)}\Theta_{(\xi,d_0')})=0.\label{eq:ptthe+}
\end{align}

Since $W_{(\xi,d_0)}$ is not divergence-free, we
introduce the skew-symmetric corrector term
\begin{align}
V_{(\xi,d_0)}:=\frac{1}{(n_*\lambda)^2}(\xi\otimes\nabla\Phi_{(\xi,d_0)}-\nabla\Phi_{(\xi,d_0)}\otimes\xi)\psi_{(\xi,d_0)}\notag
\end{align}
satisfying
\begin{align}
\div V_{(\xi,d_0)}
=W_{(\xi,d_0)} -\frac{1}{(n_*\lambda)^2}\nabla\Phi_{(\xi,d_0)}\xi\cdot\nabla\psi_{(\xi,d_0)}.\label{divOmega}
\end{align}
Finally, we obtain that for $N, M \geq 0$ and $p\in [1, \infty]$ the
following holds
\begin{align}
\|\nabla^N\partial_t^M\psi_{(\xi,d_0)}\|_{C_tL^p }\lesssim r_\parallel^{\frac{1}{p}-\frac{1}{d_0}}(\frac{r_\perp\lambda}{r_\parallel})^N(\frac{r_\perp\lambda\mu}{r_\parallel})^M,
&\label{int2}\\
\|\nabla^N\phi_{(\xi,d_0)}\|_{L^p }+\|\nabla^N\Phi_{(\xi,d_0)}\|_{L^p }\lesssim r_\perp^{\frac{d-1}{p}-\frac{d-1}{d_0}}\lambda^{N},
&\label{int3}\\
\|\nabla^N\partial_t^MW_{(\xi,d_0)}\|_{C_tL^p }+\lambda\|\nabla^N\partial_t^MV_{(\xi,d_0)}\|_{C_tL^p }
\lesssim r_\perp^{\frac{d-1}{p}-\frac{d-1}{d_0}}r_\parallel^{\frac{1}{p}-\frac{1}{d_0}}\lambda^{N}(\frac{r_\perp\lambda\mu}{r_\parallel})^M,&\label{int4}\\ \|\nabla^N\partial_t^M\Theta_{(\xi,d_0')}\|_{C_tL^p }\lesssim r_\perp^{\frac{d-1}{p}-\frac{d-1}{d_0'}}r_\parallel^{\frac{1}{p}-\frac{1}{d_0'}}\lambda^{N}(\frac{r_\perp\lambda\mu}{r_\parallel})^M,& \label{int4theta}
\end{align}
where the implicit constants may depend on $p,N$ and $M$, but are independent of $\lambda,r_\perp,r_\parallel,\mu$. These estimates can be easily deduced from the definitions.

Then let us  introduce a family of temporal functions   to oscillate the building blocks intermittently in time.
Let  $K\in\mN$ be fixed, and $G \in C_c^\infty(0, 1)$ be non-negative and 
$ \int_0^1G^2(t)\dif t=1.$
  Let $\eta >0$  be a small constant satisfying $ \eta K\ll1$. For $\xi\in\Lambda$ as defined above,  and $1\le i\le K,$ we define $\tilde{g}_{(\xi,i,d_0)}: \mathbb{T}\to \mathbb{R}$ as the 1-periodic extension of $\eta^{-1/d_0}G(\frac {t-t_{(\xi,i)}}{\eta})$,  where $t_{(\xi,i)}$ are chosen so that $\tilde{g}_{(\xi,i,d_0)}$ have disjoint supports for distinct $(\xi,i)$.  
We will also oscillate the perturbations at a large
frequency $\sigma\in\mathbb{N}$. So, we define
$$g_{(\xi,i,d_0)}(t)=\tilde{g}_{(\xi,i,d_0)}(\sigma t).$$
Similarly we define $\tilde{g}_{(\xi,i,d_0')} $ and $g_{(\xi,i,d_0')} $.

For the corrector term we define $H_{(\xi,i,d_0)},h_{(\xi,i,d_0)}:\mathbb{T}\to\mathbb{R}$ by
\begin{align}
H_{(\xi,i,d_0)}(t)=\int_0^{t}g_{(\xi,i,d_0)}(s)\dif s,\ \ h_{(\xi,i,d_0)}(t)=\int_0^{\sigma t}(\tilde{g}_{(\xi,i,d_0)}(s)\tilde{g}_{(\xi,i,d_0')}(s)-1)\dif s, \label{eq:parth}
\end{align}
where we recall that $\frac1{d_0'}+\frac1{d_0}=1$, In view of the zero-mean condition for $\tilde{g}_{(\xi,i,d_0)}(t)\tilde{g}_{(\xi,i,d_0')}(t)-1$, we see that $h_{(\xi,i,d_0)}$ is $\mathbb{T}/\sigma$-periodic, and for any $N\geq0,p\geq1$
\begin{align}
    \|g_{(\xi,i,d_0)}\|_{W^{N,p}}\lesssim(\frac{\sigma}{\eta})^N\eta^{1/p-1/d_0},\  \  \|h_{(\xi,i,d_0)}\|_{L^\infty}\leq1,\label{bd:gwnp}
\end{align}
where the universal constant is independent of the choices of  $i$ and $\xi$.

\end{document}